\documentclass[10pt]{article}
\providecommand{\keywords}[1]{\textbf{Keywords:} #1}
\usepackage{amsfonts}
\usepackage{amsmath}
\usepackage{amssymb}
\usepackage{hyperref}
\usepackage{bm}
\usepackage[alphabetic]{amsrefs}
\usepackage{fullpage}
\usepackage{theorem}
\usepackage{tensor}
\usepackage{cancel}
\usepackage{etex}
\usepackage[all]{xy}

\newtheorem{thm}{Theorem}[section]
\newtheorem{cor}[thm]{Corollary}
\newtheorem{lem}[thm]{Lemma}
\newtheorem{prop}[thm]{Proposition}
\newtheorem{prop-defn}[thm]{Proposition/Definition}

\newtheorem{conjec}[thm]{Conjecture}

\theorembodyfont{\rmfamily}
\newtheorem{defn}[thm]{Definition}
\newtheorem{exa}[thm]{Example}

\newtheorem{ass}[thm]{Assumptions}
\newtheorem{rem}[thm]{Remark}

\numberwithin{equation}{section}
\newenvironment{proof}{\noindent \emph{Proof.}}{\hspace{\stretch{1}}$\Box$}

\newcommand{\mcF} {\mathcal{F}}

\newcommand{\mcL} {\mathcal{L}}
\newcommand{\mcM} {\mathcal{M}}
\newcommand{\mcN} {\mathcal{N}}

\newcommand{\mcS} {\mathcal{S}}

\newcommand{\ii} {\mathrm{i}}
\newcommand{\ee} {\mathrm{e}}

\newcommand{\ind} {\indices}

\newcommand{\lp} [1] {{\left( #1 \right. }}
\newcommand{\rp} [1] {{\left. #1 \right) }}
\newcommand{\lb} [1] {{\left[ #1 \right. }}
\newcommand{\rb} [1] {{\left. #1 \right] }}

\newcommand{\Rho} {\mathrm{P}}

\newcommand{\im} {\mathop{\mathrm{im}}}

\newcommand{\gr} {\mathrm{gr}}

\newcommand{\Mat} {\mathrm{Mat}}

\newcommand{\Spin} {\mathrm{Spin}}
\newcommand{\SL} {\mathrm{SL}}
\newcommand{\GL} {\mathrm{GL}}
\newcommand{\Sp} {\mathrm{Sp}}
\newcommand{\SO} {\mathrm{SO}}

\newcommand{\Gr} {\mathrm{Gr}}

\newcommand{\G} {\mathrm{G}}
\newcommand{\Tgt} {\mathrm{T}}

\newcommand{\so} {\mathfrak{so}}

\newcommand{\slie} {\mathfrak{sl}}
\newcommand{\glie} {\mathfrak{gl}}
\newcommand{\g} {\mathfrak{g}}

\newcommand{\prb} {\mathfrak{p}}

\newcommand{\mfz} {\mathfrak{z}}

\newcommand{\mfC} {\mathfrak{C}}

\newcommand{\mfN} {\mathfrak{N}}
\newcommand{\mfV} {\mathfrak{V}}
\newcommand{\mfU} {\mathfrak{U}}
\newcommand{\mfM} {\mathfrak{M}}

\newcommand{\mfS} {\mathfrak{S}}

\newcommand{\mfA} {\mathfrak{A}}

\newcommand{\mfF} {\mathfrak{F}}

\newcommand{\mfW} {\mathfrak{W}}

\newcommand{\Pp} {\mathbb{P}}

\newcommand{\CP} {\mathbb{CP}}

\newcommand{\C} {\mathbb{C}}

\newcommand{\Cl} {\mathcal{C}\ell}

\newcounter{mnotecount}[section]
\renewcommand{\themnotecount}{\thesection.\arabic{mnotecount}}

\newcommand{\mnote}[1]
{\protect{\stepcounter{mnotecount}}$^{\mbox{\footnotesize
$
\bullet$\themnotecount}}$ \marginpar{\color{red}
\raggedright\tiny\em
$\!\!\!\!\!\!\,\bullet$\themnotecount: #1} }

\begin{document}

\title{Pure spinors, intrinsic torsion and curvature in odd dimensions}
 \author{Arman Taghavi-Chabert\\
 {\small Masaryk University, Faculty of Science, Department of Mathematics and Statistics,}\\
  {\small Kotl\'{a}\v{r}sk\'{a} 2, 611 37 Brno, Czech Republic } }
\date{}

\maketitle

\begin{abstract}
We study the geometric properties of a $(2m+1)$-dimensional complex manifold $\mathcal{M}$ admitting a holomorphic reduction of the frame bundle to the structure group $P \subset \mathrm{Spin}(2m+1,\mathbb{C})$, the stabiliser of the line spanned by a pure spinor at a point. Geometrically, $\mathcal{M}$ is endowed with a holomorphic metric $g$, a holomorphic volume form, a spin structure compatible with $g$, and a holomorphic pure spinor field $\xi$ up to scale. The defining property of $\xi$ is that it determines an almost null structure, i.e.\ an $m$-plane distribution $\mathcal{N}_\xi$ along which $g$ is totally degenerate.

We develop a spinor calculus, by means of which we encode the geometric properties of $\mathcal{N}_\xi$ and of its rank-$(m+1)$ orthogonal complement $\mathcal{N}_\xi^\perp$ corresponding to the algebraic properties of the intrinsic torsion of the $P$-structure. This is the failure of the Levi-Civita connection $\nabla$ of $g$ to be compatible with the $P$-structure. In a similar way, we examine the algebraic properties of the curvature of $\nabla$.

Applications to spinorial differential equations are given. Notably, we relate the integrability properties of $\mathcal{N}_\xi$ and $\mathcal{N}_\xi^\perp$ to the existence of solutions of odd-dimensional versions of the zero-rest-mass field equation. We give necessary and sufficient conditions for the almost null structure associated to a pure conformal Killing spinor to be integrable. Finally, we conjecture a Goldberg--Sachs-type theorem on the existence of a certain class of almost null structures when $(\mathcal{M},g)$ has prescribed curvature.

We discuss applications of this work to the study of real pseudo-Riemannian manifolds.
\end{abstract}


\section{Introduction and motivation}
The present article is the odd-dimensional counterpart of the author's work presented in \cite{Taghavi-Chabert2016}. Both articles work share the same motivations and goals, and the reader should refer to the latter work for further details.

Let $(\mcM, g)$ be an $n$-dimensional complex Riemannian manifold, where $n=2m+1$. We shall assume that $(\mcM,g)$ is also equipped with a global holomorphic volume form and a holomorphic spin structure so that the structure group of the holomorphic frame bundle is reduced to $G:=\Spin(n,\C)$. We work in the holomorphic category. We shall be considering a \emph{projective pure spinor field} $[\xi]$, i.e. a spinor field up to scale that annihilates a totally null $m$-plane, or \emph{$\gamma$-plane}, distribution. This will also be referred to as its associated \emph{almost null structure} $\mcN_\xi$. The structure group of the frame bundle of $(\mcM,g)$ is reduced to $P$, the stabiliser of $[\xi]$ at a point. Denote by $\g$ and $\prb$ the respective Lie algebras of $G$ and $P$, and by $\mfV$ the standard representation of $\g$. The main aim of the article is to examine the geometric properties of the $P$-structure on $(\mcM,g)$. More specifically, we will
\begin{itemize}
\item give a $P$-invariant decomposition of the space $\mfW := \mfV \otimes \left( \g/\prb \right)$ of intrinsic torsions;
\item give $P$-invariant decompositions of the spaces of curvature tensors, in particular, tracefree Ricci tensors, Cotton-York tensors and Weyl tensors;
\item apply these decompositions to the study of almost null structures and pure spinor fields on complex Riemannian manifolds.
\end{itemize}
The methodology will be a synthesis of representation theory and a spinor calculus adapted to the $P$-structure. Before we proceed, we first highlight the crucial differences between the odd- and even-dimensional cases:
\begin{itemize}
\item there is only one irreducible spinor representation of $G$ as opposed to two chiral ones -- paradoxically, this makes the spinor calculus more fiddly;
\item the stabiliser $\prb$ of $[\xi]$ induces a $|2|$-grading on $\g$, rather than a $|1|$-grading;
\item the orthogonal complement $\mcN_\xi^\perp$ of $\mcN_\xi$ is $(m+1)$-dimensional and contains $\mcN_\xi$, rather than $\mcN_\xi^\perp = \mcN_\xi$.
\end{itemize}
Consequently, one has to encode the properties of \emph{both} $\mcN_\xi$ and $\mcN_\xi^\perp$ in terms of differential conditions on $[\xi]$, although there is some degree of interdependency between $\mcN_\xi$ and $\mcN_\xi^\perp$. Making the move from even to odd dimensions is therefore not always straightforward.  A case in point is when $\mcN_\xi$ is integrable. In even dimensions, $\mcN_\xi$ would be automatically totally geodetic, but in odd dimensions, this condition is stronger. In addition, one could have the extra requirement for $\mcN_\xi^\perp$ to be also integrable, and or even totally geodetic. This is particularly relevant to generalisations of the Robinson theorem, which can be strikingly different.

The present article can, if not should, be read in conjunction with \cite{Taghavi-Chabert2016} for comparison and ease of understanding of the notions introduced in the latter. Indeed, these two papers are broadly `mirror images' of each other: the overall structure is the same in both papers as far as the numbering of the sections is concerned. For the sake of conciseness, we have not always deemed it necessary to re-establish notations and conventions.

\paragraph{Structure of the paper:} Our spinor calculus will first be developed in section \ref{sec-algebra}. New results include Propositions \ref{prop-purity_cond} and \ref{prop-pure-intersect2}, and Corollary \ref{cor-pure-intersect}, which provide simpler alternatives to some of Cartan's formulae on pure spinors. Proposition \ref{prop-intorsion-class-alg} in section \ref{sec-intrinsic-torsion} is concerned with the decomposition of the space of intrinsic torsions of a $P$-structure. In the same vein, in section \ref{sec-curvature}, Propositions \ref{prop-TFRicci-classification}, \ref{prop-Cotton-York-classification} and \ref{prop-Weyl-classification} give $P$-invariant decompositions of the spaces of tracefree Ricci tensors, Cotton-York tensors and Weyl tensors respectively.

Section \ref{sec-geometry} focuses on the geometric applications.  Proposition \ref{prop-intorsion-class}  is the geometric articulation of Proposition  \ref{prop-intorsion-class-alg}. Proposition \ref{prop-foliating-spinor}, Lemma \ref{lem-nice-impl} and Proposition \ref{prop-conformal} are concerned with geometric interpretations of $\mcN_\xi$ in terms of $\nabla [\xi]$. Three distinct generalisations of the Robinson theorems for three distinct types of zero-rest-mass fields are given in Theorems \ref{thm-Robinson-zrm}, \ref{thm-Robinson-co-zrm} and \ref{thm-Robinson-nc-zrm}. Applications to conformal Killing spinors are given in Propositions \ref{prop-conformal-Killing-foliating}, \ref{prop-Killing-foliating} and \ref{prop-cKsp2par}. Conjecture \ref{conjec-GS} postulates a generalisation of the Goldberg--Sachs theorem given in \cite{Taghavi-Chabert2012}. Integrability conditions for solutions of the field equations involved are also given in Propositions \ref{prop-foliation-int-cond}, \ref{prop-recurrent-spinor}, \ref{prop-int-cond-par}, \ref{prop-int-zrm}, \ref{prop-int-cond-twistor-spinor} and \ref{prop-int-cond-Killing-spinor} among others.

Appendix \ref{sec-spinor-descript} contains useful formulae to characterise tracefree Ricci, Cotton--York  and Weyl tensors in the light of the decompositions given in section \ref{sec-curvature}. A brief discussion of spinor calculus in dimensions three and five can be found in appendix \ref{sec-3-5dim}. In appendix \ref{sec-conformal}, we describe conformal transformations of spinor fields.

\section{Spinor calculus}\label{sec-algebra}
Conventions follow those of \cite{Taghavi-Chabert2016}, based on \cites{Penrose1984,Penrose1986}. Further background on spinors can be found in \cites{Cartan1981,Penrose1986,Budinich1989,Harnad1995,Kopczy'nski1997} and on representation theory in \cites{Baston1989,vCap2009}. 

\subsection{Clifford algebras and spinor representations}
Let $\mfV$ be an $n$-dimensional complex vector space equipped with a non-degenerate symmetric bilinear form $g_{ab} = g_{(ab)} \in \odot^2 \mfV^*$, by means of which we shall identify $\mfV$ with its dual $\mfV^*$. We choose an orientation, and denote the Hodge star operator by $\star$. Denote the Clifford algebra of $(\mfV , g )$ by $\Cl ( \mfV , g )$ and the Clifford multiplication by a dot $\cdot$. We recall that $\Cl ( \mfV , g) \cong \wedge^\bullet \mfV$ as vector spaces. Henceforth, we assume $n=2m+1$. The spin group $G:= \Spin(2m+1,\C)$ has a single $2^m$-dimensional irreducible representation, the \emph{spinor space} $\mfS$ of $(\mfV,g )$. We can realise $\mfS$ as follows. We split $\mfV$ as $\mfV \cong \mfN \oplus \mfN^* \oplus \mfU$ where $\mfN$ and $\mfN^*$ are two totally null $m$-dimensional subspaces of $\mfV$, dual to each other $\mfN^*$, and the one-dimensional complement $\mfU$ is non-null. Then $\mfS$ can be identified with $\wedge^\bullet \mfN$ as an $\Cl(\mfV,g)$-module: for any $(v,w,u) \in \mfV$, the action of the Clifford algebra on $\mfS$ is given by $(v,w,u) \cdot \xi = v \wedge \xi - w \lrcorner \xi + \ii \epsilon u \, \xi$
where $\ii^2=-1$ and $\epsilon = 1$ if $\xi \in \wedge^m \mfN \oplus \wedge^{m-2} \mfN \oplus \wedge^{m-4} \mfN \oplus \ldots$, and $\epsilon = -1$ if $\xi \in \wedge^{m-1} \mfN \oplus \wedge^{m-3} \mfN \oplus \wedge^{m-5} \mfN \oplus \ldots$.


The Clifford algebra can be shown to be isomorphic to a direct sum of two inequivalent copies of the algebra $\Mat(2^m,\C)$ of $2^m \times 2^m$-matrices over $\C$ acting on $\mfS$. Elements of $\mfS$ will carry upstairs upper-case Roman indices, e.g. $\xi^A$, and similarly for elements of its dual $\mfS^*$, with downstairs indices, e.g. $\eta_A$. The Clifford algebra $\Cl(\mfV,g )$ is generated by the \emph{$\gamma$-matrices $\gamma \ind{_a _A ^B}$} which satisfy
\begin{align}\label{eq-red_Clifford_property}
 \gamma \ind{_{\lp{a}} _A ^C} \gamma \ind{_{\rp{b}} _C ^B} & = - g_{ab} \delta \ind*{_A^B} \, .
\end{align}
Thus, only skew-symmetrised products of $\gamma$-matrices count, and we shall make use of the notational short hand
$\gamma \ind{_{a_1 a_2 \ldots a_p} _A ^B} := \gamma \ind{_{\lb{a_1}} _A ^{C_1}} \gamma \ind{_{a_2} _{C_1} ^{C_2}} \ldots \gamma \ind{_{\rb{a_p}} _{C_p} ^{B}}$ for any $p$. These realise the linear isomorphism $\wedge^\bullet \mfV \cong \Cl(\mfV,g) \cong \Mat(2^m,\C) \oplus \Mat(2^m,\C)$, and the two copies of $\Mat(2^m,\C)$ will be identified by Hodge duality.

The spinor space $\mfS$ and its dual $\mfS^*$ are equipped with non-degenerate bilinear forms, denoted $\gamma_{AB}$, with which one can in effect raise or lower spinor indices. In particular, we have bilinear maps $\gamma_{a_1 \ldots a_p AB}$ from $\mfS \times \mfS$ to $\wedge^p \mfV$ for any $p$. Depending on the values of $m$ and $p$, these can be either symmetric or skewsymmetric.
Our treatment will be largely dimension independent, and we will in general dispense of their use. Nonetheless, we shall make use of the following result:
\begin{lem}\label{lem-technical}
 We have
\begin{multline*}
 \gamma \ind{_a_A^B} \gamma \ind{_{b_1 \ldots b_p}_{BD}} \gamma \ind{_c_C^D} 
= (-1)^m \left( \gamma \ind{_{a b_1 \ldots b_p c}_{AC}} - (p + 1 )\gamma \ind{_{\lb{a} b_1 \ldots b_{p-1}}_{AC}} g \ind{_{\rb{b_p} c}} \right. \\
\left. - p \, g \ind{_{a \lb{b_1}}} \gamma \ind{_{b_2 \ldots \rb{b_p}c}_{AC}} + p (p+1) g \ind{_{a \lb{b_1}}} \gamma \ind{_{b_2 \ldots b_{p-1}}_{AC}} g \ind{_{\rb{b_p}c}} \right) \, .
\end{multline*}
In particular, $\gamma \ind{^a_A^B} \gamma \ind{_{b_1 \ldots b_p}_{BD}} \gamma \ind{_a_C^D} = 
 (-1)^{m + p} ( 2p - 2m - 1 )\gamma \ind{_{b_1 \ldots b_p}_{AC}}$.
\end{lem}

\subsection{Null structures and pure spinors}
\begin{defn}
A \emph{null structure} or \emph{$\gamma$-plane} on $\mfV$ is an $m$-dimensional vector subspace $\mfN \subset \mfV$ that is totally null, ie $g_{ab} X^a Y^b =0$ for all $X^a, Y^a \in \mfN$.
\end{defn}

Let $\xi \ind*{^A}$ be a non-zero spinor in $\mfS$, and consider the map $\xi \ind*{_a^A} := \xi \ind*{^B} \gamma  \ind{_a_B^A} : \mfV \rightarrow \mfS$. By \eqref{eq-red_Clifford_property}, the kernel of $\xi \ind*{_a^A} : \mfV \rightarrow \mfS$ is \emph{totally null}.
\begin{defn}
 A non-zero spinor $\xi \ind*{^A}$ is said to be \emph{pure} if the kernel of $\xi \ind*{_a^A}: \mfV \rightarrow \mfS$ is $m$-dimensional, and thus defines a null structure.
 
 The projectivisation of the line $\langle \xi^A \rangle$ spanned by a pure spinor $\xi^A$ will be referred to as a \emph{projective pure spinor} $[\xi^A] \in \Pp \mfS$.
\end{defn}

\begin{prop}[\cite{Cartan1981}] \label{prop-fundamental}
There is a one-to-one correspondence between projective pure spinors and $\gamma$-planes on $(\mfV, g)$.
\end{prop}


Henceforth, $\xi^A$ will denote a fixed pure spinor. The crucial departure from the even-dimensional case is that a null structure is contained in its orthogonal complement, that is, $\xi^A$ induces a filtration
\begin{align}\label{eq-filtration}
\{ 0 \} =: \mfV^2 \subset \mfV^1 \subset \mfV^0 \subset \mfV^{-1} \, ,
\end{align}
where $\mfV^{-1} := \mfV$, $\mfV^1 := \ker \xi \ind*{_a^A} : \mfV \rightarrow \mfS$ and the orthogonal complement $\mfV^0$ of $\mfV^1$ with respect to $g_{ab}$ is $(m+1)$-dimensional.
%
The map $\xi_a^A$ allows us to identify elements of $\mfV$ with elements of $\mfS$, notably
\begin{align*}
 \left( \mfV^{-1}/ \mfV^1 \right) \otimes \mfS^{\frac{m}{2}} & \cong \mfS^{\frac{m-2}{2}} \, , &
   ( \mfV^0/ \mfV^1 ) \otimes \mfS^{\frac{m}{2}} & \cong \mfS^{\frac{m}{2}} \, , &
 ( \mfV^{-1} / \mfV^0 ) \otimes \mfS^{\frac{m}{2}} & \cong \mfS^{\frac{m-2}{2}}/\mfS^{\frac{m}{2}} \, ,
 \end{align*}
 where  $\langle \xi^A \rangle =: \mfS^{\frac{m}{2}} \subset \mfS^{\frac{m-2}{2}} := \im \xi \ind*{_a^A} : \mfV \rightarrow \mfS$.
Dually, we also have
\begin{align*}
\mfV^0 & \cong \mfS^{\frac{m}{2}} \otimes \left( \mfS^{-\frac{m}{2}} / \mfS^{-\frac{m-4}{2}} \right) \, , &
\mfV^0 / \mfV^1 & \cong \mfS^{\frac{m}{2}} \otimes \left( \mfS^{-\frac{m}{2}} / \mfS^{-\frac{m-2}{2}} \right) \, , &
\mfV^1 & \cong \mfS^{\frac{m}{2}} \otimes \left( \mfS^{-\frac{m-2}{2}} / \mfS^{-\frac{m-4}{2}} \right) \, ,
\end{align*}
where $\mfS^{-\frac{m}{2}} := \mfS^*$, $\mfS^{-\frac{m-2}{2}} := \ker \xi^A : \C \leftarrow \mfS^*$ and $\mfS^{-\frac{m-4}{2}} := \ker \xi \ind*{_a ^A} : \mfV^* \leftarrow \mfS^*$. Using  \eqref{eq-red_Clifford_property}, we can check that $\mfS^{-\frac{m-4}{2}}  \subset \mfS^{-\frac{m-2}{2}}  \subset \mfS^{-\frac{m}{2}}$.
More concretely, we have
\begin{lem}\label{lem-vector_decomposition}
 Let $V^a$ be a non-zero vector in $\mfV$. Then
 \begin{itemize}
 \item $V^a$ is an element of $\mfV^0$ if and only if $V^a = \xi \ind*{^a^A} v \ind{_A}$ for some non-zero $v_A \in \mfS^{-\frac{m}{2}} / \mfS^{-\frac{m-4}{2}}$;
 \item $V^a$ is an element of $\mfV^1$ if and only if $V^a = \xi \ind*{^a^A} v \ind{_A}$ for some non-zero $v_A \in \mfS^{-\frac{m-2}{2}} / \mfS^{-\frac{m-4}{2}}$.
 \end{itemize}
\end{lem}
As a direct consequence, a pure spinor $\xi^A$ must satisfy $\xi \ind*{^a^A} \xi \ind*{_a^B} = \lambda \, \xi^A \xi^B$ for some $\lambda$. Contracting each side by $\xi \ind*{^c^C} \gamma \ind{_c_A^D}$ and a little algebra then leads to $\lambda=-1$.
\begin{prop}\label{prop-purity_cond}
A non-zero spinor $\xi^A$ is pure if and only if it satisfies
\begin{align}\label{eq-purity_cond}
 \xi \ind*{^a ^C} \xi \ind*{_a ^D} & = - \xi^C \xi^D \, .
\end{align}
\end{prop}
By Lemma \ref{lem-technical}, we can express \eqref{eq-purity_cond} equivalently as the following more familiar algebraic characterisation.
\begin{prop}[\cite{Cartan1981}]\label{prop-Cartan_char}
A non-zero spinor $\xi^A$ is pure if and only if it satisfies
\begin{align}\label{eq-purity_cond_Cartan}
\begin{aligned}
 \gamma \ind{_{a_1 \ldots a_p} _{A B}} \xi^A \xi^B & = 0 \, , & \mbox{for all $p<m$, $p \equiv m,m+1 \pmod 4$,} \\
 \gamma \ind{_{A B}} \xi^A \xi^B & = 0 \, , & \mbox{when $m \equiv 0 , 3 \pmod 4$,} \\
 \gamma \ind{_{a_1 \ldots a_m} _{A B}} \xi^A \xi^B & \neq 0 \, .
 \end{aligned}
\end{align}
\end{prop}
We shall refer to both equations \eqref{eq-purity_cond} and \eqref{eq-purity_cond_Cartan} as the \emph{purity conditions} of a spinor $\xi^A$. These are vacuous when $m\leq2$, i.e. all spinors are pure when $m \leq 2$.

The only non-vanishing irreducible component of the tensor product $\xi^A \xi^B$ is thus the $m$-form
$\phi \ind{_{a_1 \ldots a_m}} := \gamma \ind{_{a_1 \ldots a_m}_{AB}} \xi \ind*{^A} \xi \ind*{^B}$, which can be seen to annihilate $\mfV^0$. It is \emph{null} (or \emph{simple} or \emph{decomposable}) in the sense that
$\phi \ind{_{a_1 \ldots a_m}} = \xi \ind*{_{a_1}^{A_1}} \ldots \xi \ind*{_{a_m}^{A_m}} \varepsilon \ind{_{A_1 \ldots A_m}} \in \wedge^m \mfV^1$ for some $\varepsilon \ind{_{A_1 \ldots A_m}} \in \wedge^m \left(\mfS^{-\frac{m-2}{2}} / \mfS^{-\frac{m-4}{2}} \right)$. Similarly, its Hodge dual  $(* \phi ) \ind{_{a_1 \ldots a_{m+1}}} \in \wedge^{m+1} \mfV^0$ annihilates $\mfV^1$ and is represented by some $\varepsilon \ind{_{A_1 \ldots A_{m+1}}} \in \wedge^{m+1} \left(\mfS^{-\frac{m}{2}} / \mfS^{-\frac{m-4}{2}} \right)$.

\begin{prop}[\cite{Cartan1981}]\label{prop-pure-intersect}
Let $\alpha^A$ and $\beta^A$ be two spinors not proportional to each other. Then the $\gamma$-planes associated to $\alpha^A$ and $\beta^A$ intersect in a totally null $(m-k)$-plane if and only if
 \begin{align*}
  \gamma \ind{_{a_1 \ldots a_p}_{AB}} \alpha^A \beta^B & = 0 \, , & \mbox{for all $p < m-k$,} \\
   \gamma \ind{_{AB}} \alpha^A \beta^B & \neq 0 \, , \\
  \gamma \ind{_{a_1 \ldots a_{m-k}}_{AB}} \alpha^A \beta^B & \neq 0 \, .
 \end{align*}
for $k=1, \ldots m$.
\end{prop}
As a consequence of Lemma \ref{lem-technical}, we have, in the special case when $k=1,2$, the equivalent characterisations.
\begin{prop}\label{prop-pure-intersect2}
Let $\alpha^A$ and $\beta^A$ be two spinors not proportional to each other. Then 
\begin{itemize}
\item the $\gamma$-planes associated to $\alpha^A$ and $\beta^A$ intersect in a totally null $(m-1)$-plane if and only if
\begin{align}\label{eq-intersect-m-1-short}
  \alpha \ind{^a^A} \beta \ind*{_a^B} & = \alpha^A \beta^B - 2 \, \beta^A \alpha^B = - \alpha^{(A} \beta^{B)} + 3 \, \alpha^{[A} \beta^{B]} \, ;
 \end{align}
\item the $\gamma$-planes associated to $\alpha^A$ and $\beta^A$ intersect in a totally null $(m-k)$-plane where $k=1$ or $2$, if and only if
\begin{align}\label{eq-intersect-m-2-short}
 \alpha \ind{^a^{(A}} \beta \ind*{_a^{B)}} & = - \alpha \ind{^{(A}} \beta \ind{^{B)}} \, .
\end{align}
\end{itemize}
\end{prop}
Finally, in the context of our present notation, we conclude
\begin{cor}\label{cor-pure-intersect}
Let $\xi^A$ be a pure spinor in $(\mfV,g)$ and let $\mfS^{\frac{m-2}{2}} := \im \xi \ind*{_a^A} : \mfV \rightarrow \mfS$ as before. Then
\begin{itemize}
\item Any non-zero spinor in $\mfS^{\frac{m-2}{2}}$ is pure.
\item The $\gamma$-planes associated to any two pure spinors in $\mfS^{\frac{m-2}{2}}$ intersect in a totally null $(m-k)$-plane where $k$ can be either $0$ or $1$ or $2$.
\end{itemize}
\end{cor}
We omit the proof which is essentially the same as in the even-dimensional case and consists in checking the veracity of the algebraic conditions \eqref{eq-purity_cond}, \eqref{eq-intersect-m-1-short} and \eqref{eq-intersect-m-2-short}


\paragraph{Splitting}
It is convenient to choose a splitting of the filtration \eqref{eq-filtration} as
\begin{align}\label{eq-null-grading}
 \mfV & = \mfV_{-1} \oplus \mfV_{0} \oplus \mfV_1 \, ,
\end{align}
where $\mfV_1 := \mfV^{-1}$, and $\mfV_i$ are subspaces such that $\mfV^i = \mfV_i \oplus \mfV^{i+1}$, each linearly isomorphic to $\mfV^i/\mfV^{i+1}$. Now, $\mfV_{-1}$ is a $\gamma$-plane dual to $\mfV^1$ to which we associate a pure spinor $\eta_A$ dual to $\xi^A$, i.e.
$\mfV_{-1} := \ker \eta \ind{_a_A} : \mfV \rightarrow \mfS^*$, where $\eta \ind{_a_A} := \eta_B \gamma \ind{_a_A^B}$. Conversely, any choice of spinor dual to $\xi^A$ induces a splitting \eqref{eq-null-grading}.

For convenience,  we choose $\xi^A$ and $\eta_A$ to satisfy $\xi^A \eta_A = -\frac{1}{2}$, and define
\begin{align}\label{eq-u-h}
u^a & := 2 \, \eta \ind*{^a_A} \xi^A \, , & h \ind{_{ab}} & := g \ind{_{ab}} + u \ind{_a} u \ind{_b} \, .
\end{align}
Then, $u^a$ spans $\mfV_0$, and satisfies $u^a u_a = -1$, $u^a \xi \ind*{_a^A} = \xi^A$ and $u^a \eta \ind{_a_A} = \eta_A$. Further, $h_{ab}$ is a non-degenerate symmetric bilinear form on $\mfV_1 \oplus \mfV_{-1}$, i.e. $h_{ab} u^a = 0$, $h \ind{_a^c} h \ind{_c^b} = h \ind{_a^b}$, and $h \ind{_a^a} = n-1$.

Next, define $\mfS_{-\frac{m-2}{2}} := \{ \im \eta \ind{_a_A} : \mfV \rightarrow \mfS^* \} \cap \{ \ker \xi^A : \C \leftarrow \mfS^* \}$. This is the dual of $\mfS_{\frac{m-2}{2}}$, the complement of $\mfS^{\frac{m}{2}} = \langle \xi^A \rangle$ in $\mfS^{\frac{m-2}{2}}$. Elements of $\mfV_1$ and $\mfV_{-1}$ must be of the form $\xi^{aA} v_A$ and $\eta \ind*{^a_A} w^A$ respectively, for some $v_A$ in $\mfS_{-\frac{m-2}{2}}$ and $w^A$ in $\mfS_{\frac{m-2}{2}}$, i.e.\ $v_A \xi^A =0$ and $w^A \eta_A =0$.

Finally, we introduce the map
\begin{align}\label{eq-identity-map}
I \ind*{_A^B} & := \eta \ind{_a_A} \xi \ind*{^a^B} + \eta \ind{_A} \xi \ind*{^B} \, ,
\end{align}
which can be seen to be the identity element on $\mfS_{\frac{m-2}{2}}$, or dually, on $\mfS_{-\frac{m-2}{2}}$. In particular $\xi^A I \ind*{_A^B} = \eta_B I \ind*{_A^B}= 0$.

\subsection{The stabiliser of a projective pure spinor in $\so(2m+1,\C)$ for $m>1$}
In what follows, the Lie algebra $\g :=\so(2m+1,\C)$ will be freely identified with $\wedge^2 \mfV$ or $\wedge^2 \mfV^*$. At this stage, we also assume $m>1$, the case $m=1$ being treated briefly in section \ref{sec-3dim}.
\paragraph{Filtration}
The filtration \eqref{eq-filtration} induces a filtration of vector subspaces
there is a filtration 
\begin{align}\label{eq-filtration-g}
 \{ 0 \} =: \g^3 \subset \g^2 \subset \g^1 \subset \g^0 \subset \g^{-1} \subset \g^{-2} := \g \, ,
\end{align}
on $\g$, where
\begin{align*}
 \g^{-1} & := \{ \phi \ind{_{ab}} \in \g : \xi \ind*{^a ^{\lb{A}}} \xi \ind*{^b ^B} \phi \ind{_{a b}} \xi \ind*{^{\rb{C}}} = 0 \} \, ,&
 \g^0 & := \{ \phi \ind{_{ab}} \in \g : \xi \ind*{^a ^A} \xi \ind*{^b ^B} \phi \ind{_{a b}} = 0 \} \, , \\
 \g^1 & := \{ \phi \ind{_{ab}} \in \g : \xi \ind*{^a ^{\lb{A}}} \phi_{ab} \xi \ind*{^{\rb{B}}} = 0 \} \, , &
 \g^2 & := \{ \phi \ind{_{ab}} \in \g : \xi \ind*{^a ^A} \phi_{ab} = 0 \} \, .
\end{align*}
The Lie bracket $[ \cdot , \cdot ]:\g \times \g \rightarrow \g$ on $\g$ is compatible with this filtration, i.e. $[\g^i,\g^j] \subset \g^{i+j}$, with the convention that $\g^i = \{ 0 \}$ for $i \geq 3$, and $\g^i = \g$ for all $i \leq -2$, i.e. $\g$ is a filtered Lie algebra.

Using the useful identities
\begin{align}
 \phi \ind{_{ab}} \xi \ind*{^a^A} \xi \ind*{^b^B} \gamma \ind{^c_A^C} \gamma \ind{_c_B^D} & = - \phi \ind{_{ab}} \left( \xi \ind*{^a^C} \xi \ind*{^b^D} + 4 \, \xi \ind*{^{ab}^{\lb{C}}} \xi \ind*{^{\rb{D}}} \right) \, , &
 \phi \ind{_{ab}} \xi \ind*{^{ab}^A} \gamma \ind{^c_A^C} \xi \ind*{_c^D} & = - \phi \ind{_{ab}} \left( \xi \ind*{^{ab}^C} \xi \ind*{^D} + 4 \, \xi \ind*{^a^C} \xi \ind*{^b^D} \right) \, , \label{eq-useful-id}
\end{align}
or otherwise, one can show $\phi \ind{_{ab}} \xi \ind*{^a^A} \xi \ind*{^b^B} = 0$ if and only if $\phi \ind{_{ab}} \xi \ind*{^{ab}^{\lb{A}}} \xi \ind*{^{\rb{B}}} = 0$, and conclude:

\begin{prop}
The Lie subalgebra $\prb := \g^0$ is the stabilizer of $[ \xi^A ]$, i.e. $\phi \ind{_{ab}} \xi \ind*{^{ab}^A}  \propto \xi \ind*{^A}$.
\end{prop}

%

The stabiliser $\prb$ of $[\xi^A]$ is a \emph{parabolic Lie subalgebra} of $\g$ \cites{Fulton1991,vCap2009}.

\paragraph{Splitting}
Splitting \eqref{eq-filtration-g} yields a $|2|$-grading $\g = \g_{-2} \oplus \g_{-1} \oplus \g_0 \oplus \g_1 \oplus \g_2$ with $[ \g_i , \g_j ] \subset \g_{i+j}$, for all $i,j$, with the convention that $\g_i = \{ 0 \}$ for all $|i|>2$. In relation to \eqref{eq-null-grading}, we have $\g_{\pm2} \cong \wedge^2 \mfV_{\pm1}$, $\g_{\pm1} \cong \mfV_0 \otimes \mfV_{\pm1}$ and $\g_0 \cong \mfV_{-1} \otimes \mfV_1$.
The Lie subalgebra $\g_0$ is isomorphic to $\glie(m,\C)$, and thus splits further as $\g_0 = \mfz_0 \oplus \slie_0$ where $\mfz_0$ is the center $\g_0$ and $\slie_0 = \slie(m,\C)$ is the simple part of $\g_0$. The center is spanned by the \emph{grading element} $E \ind{_{ab}} := - 2 \, \xi \ind*{_{\lb{a}}^A} \eta \ind{_{\rb{b}}_A}$, with image in $\Cl(\mfV,g)$ given by $E \ind{_A^B} := - \frac{1}{4} E \ind{_{ab}} \gamma \ind{^{ab}_A^B}$. For consistency with \cite{Taghavi-Chabert2016}, we also set
\begin{align}\label{eq-omega}
\omega \ind{_{ab}} := E \ind{_{ab}} = - 2 \, \xi \ind*{_{\lb{a}}^A} \eta \ind{_{\rb{b}}_A} \, .
\end{align}
An element $\phi_{ab}$ of $\slie_0$ can then be written as $\phi_{ab} = 2 \, \xi \ind*{_{\lb{a}}^A} \eta \ind{_{\rb{b}}_B} \phi \ind{_A^B}$ for some \emph{tracefree} $\phi \ind{_A^B} \in \mfS_{-\frac{m-2}{2}} \otimes \mfS_{\frac{m-2}{2}}$ in the sense that $\phi \ind{_A^B} I \ind*{_B^A} = 0$ where $I \ind*{_B^A}$ is defined by \eqref{eq-identity-map}.

\paragraph{Parabolic Lie subgroups}
At the group level, we denote by $P$ the stabiliser of $[\xi^A]$ in $G$. This is a parabolic Lie subgroup of $G$ with Lie algebra $\prb$. Its Levy decomposition is given by $P = G_0 \ltimes P_+$, where the image $G_0$ in $G \rightarrow \SO(2m+1,\C)$ under the covering map is the complex general linear group $\GL(m,\C)$, and $P_+$ is the nilpotent Lie group generated by $\g_1 \oplus \g_2$. All our $\prb$- and $\g_0$-modules will also be $P$- and $G_0$-modules. The spinor calculus developed here is then manifestly $P$-invariant.

\paragraph{Associated graded vector space}
We now introduce the associated graded $\prb$-module $\gr(\g) = \bigoplus_{i=-2}^2 \gr_i (\g)$ where $\gr_i (\g) = \g^i/\g^{i+1}$. Each $\gr_i(\g)$ is linear isomorphic to the $\g_0$-module $\g_i$, and we have a direct sum decomposition $\gr_0(\g) = \g_0^0 \oplus \g_0^1$, where $\g_0^0 := \left(\g^1 + \mfz_0 \right) / \g^1$ and $\g_0^1  := \left(\g^1 + \slie_0 \right) / \g^1$. Writing $\xi \ind*{_{ab}^A} := \xi \ind*{^B} \gamma \ind{_{ab}_B^A} : \wedge^2 \mfV \rightarrow \mfS$ and $\mfS^{\frac{m-4}{2}} := \im \xi \ind*{_{ab}^A} : \wedge^2 \mfV \rightarrow \mfS$, we can define
\begin{align*}
{}^\g  _\xi \Pi^0_0 ( \phi )  & := \xi \ind*{^{ab} ^A} \phi_{ab} \, , &
 {}^\g _\xi \Pi^1_0 ( \phi )  & := \xi \ind*{^a ^{\lb{A}}} \phi_{ab} \xi \ind*{^{\rb{B}}} + \frac{1}{n-1} \xi \ind*{^{ac} ^C} \phi _{ac} \gamma \ind{_b _C ^{\lb{A}}} \xi \ind*{^{\rb{B}}} \, .
\end{align*}
Then $\g^1 + \mfz_0 = \{ \phi_{ab} \in \g : {}^\g  _\xi \Pi^1_0 ( \phi ) = 0 \}$ and $\g^1 + \slie_0 = \{ \phi_{ab} \in \g : {}^\g  _\xi \Pi^0_0 ( \phi ) = 0 \}$. For convenience, we also set $\g_i^0 := \gr_i (\g)$ for $i=\pm1,\pm2$.

\subsection{Generalisation}
As explained in \cite{Taghavi-Chabert2016}, the parabolic subalgebra $\prb$ induces a filtration $\{ \mfM^i \}$ of indecomposable $\prb$-modules on any finite $\g$-module $\mfM$. We can split the filtration as a direct sum of $\g_0$-modules $\mfM_i$ isomorphic to $\mfM^i/\mfM^{i+1}$, on which the grading element $E$ acts diagonalisably with eigenvalue $i$. Each $\mfM^i/\mfM^{i+1}$, respectively $\mfM_i$, splits into a direct sum of irreducible $\prb$-submodules $\mfM_i^j$, respectively $\g_0$-submodules $\breve{\mfM}_i^j$, with $\mfM_i^j \cong \breve{\mfM}_i^j$ as vector spaces. We record the action of $\g_1 \subset \prb$ on each $\breve{\mfM}_i^j$ by an arrow as in \cite{Taghavi-Chabert2016}.


To deal with the spinor representation $\mfS$, we define the maps
$\xi \ind*{_{a_1 \ldots a_k}^A} := \xi \ind*{^B} \gamma \ind{_{a_1 \ldots a_k}_B ^A} : \wedge^k \mfV \rightarrow \mfS$ for $k=1, \ldots, 2m+1$. Then the spinor module $\mfS \cong \mfS^*$ admits a $P$-invariant filtration
\begin{align*}
  \mfS^{\frac{m}{2}} \subset \mfS^{\frac{m-2}{2}} \subset \ldots \subset \mfS^{-\frac{m-2}{2}} \subset \mfS^{-\frac{m}{2}}  := \mfS \, ,
\end{align*}
where $\mfS^{\frac{m}{2}} = \langle \xi^A \rangle$, $\mfS^{\frac{m-2k}{2}} := \im \xi \ind*{_{a_1 \ldots a_k}^A} : \wedge^k \mfV \rightarrow \mfS$, $\mfS^{-\frac{m-2}{2}} = \ker \xi^A : \C \leftarrow \mfS^*$ and $\mfS^{-\frac{m-2k-2}{2}} = \ker \xi \ind*{_{a_1 \ldots a_k} ^A} : \wedge^k \mfV^* \leftarrow \mfS^*$ for $k=1, \ldots m$. Further, we can choose subspaces $\mfS_i \subset \mfS^i$ such that $\mfS^i = \mfS_i \oplus \mfS^{i+1}$ such that
\begin{align*}
 \mfS & = \mfS_{\frac{m}{2}} \oplus \mfS_{\frac{m-2}{2}} \oplus \ldots \oplus \mfS_{-\frac{m-2}{2}} \oplus \mfS_{-\frac{m}{2}} \, .
\end{align*}
The grading element $E \ind{_{ab}}$ in $\mfz_0$ and the spanning element $u_a$ of $\mfV_0$ have eigenvalues $\frac{m-2k}{2}$ and $(-1)^k$ respecively on $\mfS_{\frac{m-2k}{2}}$. This description is consistent with the identification of $\mfS$ with $\wedge^\bullet \mfN$.


\subsection{Null Grassmanians}
The space of all null structures or $\gamma$-planes in $(\mfV,g)$ is the null (or isotropic) Grassmanian $\Gr_m (\mfV,g)$. Proposition \ref{prop-fundamental} allows us to identify $\Gr_m (\mfV,g)$ as the space of projective pure spinors of $(\mfV,g)$. This is a compact complex subvariety of $\Pp \mfS$ defined by the purity conditions \eqref{eq-purity_cond_Cartan}, and it is isomorphic to the $\frac{1}{2}m(m+1)$-dimensional homogeneous space $G/P$. When $m=1,2$, this space is isomorphic to the complex projective space $\CP^{\frac{1}{2}m(m+1)}$.


\subsection{Real pure spinors}\label{sec-real}
When  $\mfV$ is a real $(2m+1)$-dimensional vector space equipped with a definite or indefinite non-degenerate symmetric bilinear form of signature $(p,q)$, the spinor representation is complex and equipped with a real or quaternionic structure, by means of which a (complex) pure spinor $\xi$ is sent to its complex conjugate $\overline{\xi}$, and correspondingly, its associated (complex) null structure $\mfN_\xi$ to its complex conjugate  $\overline{\mfN}_{\overline{\xi}}$. In contrast to even dimensions, the \emph{real index} $r$ of $\xi$, being the dimension of $\mfN_\xi \cap \overline{\mfN}_{\overline{\xi}}$, can take any integer value from $0$ to $\min(p,q)$ -- see \cite{Kopczy'nski1992}. When $g$ is positive definite, $r$ is always $0$, and $\mfN_\xi$ defines a metric-compatible CR structure, also referred to, rather inappropriately, as a \emph{contact Riemannian structure}. When $g$ is of Lorentzian signature, i.e.\ $(1,2m)$ or $(2m,1)$, $r$ may be $0$ or $1$. In the latter case, one obtains a \emph{Robinson structure} \cites{Nurowski2002,Trautman2002,Taghavi-Chabert2011,Taghavi-Chabertb}. When $g$ has signature $(m,m+1)$ or $(m+1,m)$, and $r=m$, we obtain a totally real analogue of the above discussion, i.e.\ $\xi$, $\mfN_\xi$ and the stabiliser $P$ of $[\xi]$ in the connected identity component of $\Spin(m,m+1)$ are all real.

\section{Decomposition of the intrinsic torsion}\label{sec-intrinsic-torsion}
Define the $\prb$-module $\mfW := \mfV \otimes \left( \g / \prb \right)$, where as before $\g := \so(2m+1,\C)$, $\mfV$ its standard representation, and $\prb \subset \g$ stabilises a projective pure spinor $[ \xi^A ]$. We assume $m>1$, leaving the case $m=1$ to appendix \ref{sec-3-5dim}.

\begin{rem}
In what follows, $\circledcirc$ denotes the Cartan product, and $\g_0$-modules and $\prb$-modules are abbreviated to $\g_0$-mod and $\prb$-mod respectively.
\end{rem}

\begin{prop}\label{prop-intorsion-class-alg}
The $\prb$-module $\mfW$ admits a filtration
\begin{align*}
 \mfW^0 \subset \mfW^{-1} \subset \mfW^{-2} \subset \mfW^{-3} \, ,
\end{align*}
where
\begin{align*}
 \mfW^{-3} & := \mfV^{-1} \otimes \left( \g^{-2} / \g^0 \right) \, , & \mfW^{-2} & := \left( \mfV^{-1} \otimes \left( \g^{-1} / \g^0 \right) \right) \oplus \left( \mfV^0 \otimes \left( \g^{-2} / \g^0 \right) \right) \, , \\
 \mfW^{-1} & := \left( \mfV^0 \otimes \left( \g^{-1} / \g^0 \right) \right) \oplus \left( \mfV^1 \otimes \left( \g^{-2} / \g^0 \right) \right) \, , & \mfW^0 & :=  \mfV^1 \otimes \left( \g^{-1} / \g^0 \right) \, .
\end{align*}
The associated graded $\prb$-module
\begin{align*}
\gr(\mfW) & = \gr_{-3} (\mfW) \oplus \gr_{-2} (\mfW) \oplus \gr_{-1} (\mfW) \oplus \gr_0 (\mfW)
\end{align*}
decomposes into a direct sum
\begin{align*}
 \gr_{-3}(\mfW) & \cong \mfW_{-3}^0 \oplus \mfW_{-3}^1 \, , & \gr_{-2}(\mfW) & \cong \mfW_{-2}^0 \oplus \mfW_{-2}^1 \oplus \mfW_{-2}^2 \, , \\
 \gr_{-1}(\mfW) & \cong \mfW_{-1}^0 \oplus \mfW_{-1}^1 \oplus \mfW_{-1}^2 \, , & \gr_0 (\mfW) & \cong \mfW_0^0 \oplus \mfW_0^1 \, ,
\end{align*}
of irreducible $\prb$-modules as described below
\vspace{-6.5mm}
\begin{center}
\begin{minipage}[b]{0.45\linewidth}\centering
\begin{displaymath}
{\renewcommand{\arraystretch}{1.8}
\begin{array}{||c|c|c||}
\hline
\text{$\prb$-mod} & \text{$\g_0$-mod} & \text{Dimension} \\
\hline
\mfW_{-3}^0 & \wedge^3 \mfV_{-1} 
& {\frac{1}{3!}m(m-1)(m-2)} \\
\mfW_{-3}^1 & \mfV_{-1} \circledcirc \g_{-2} 
& {\frac{1}{3}m(m^2-1) } \\
\hline
\mfW_{-2}^0 & \mfV_0 \otimes \g_{-2} 
& {\frac{1}{2}m(m-1) } \\
\mfW_{-2}^1 & \wedge^2 \mfV_{-1} \otimes \mfV_0 
& {\frac{1}{2}m(m-1) } \\
 \mfW_{-2}^2 & \odot^2 \mfV_{-1} \otimes \mfV_0 
& {\frac{1}{2}m(m+1) } \\
\hline
\end{array}
}
\end{displaymath}
\end{minipage}
\begin{minipage}[b]{0.45\linewidth}\centering
\begin{displaymath}
{\renewcommand{\arraystretch}{1.8}
\begin{array}{||c|c|c||}
\hline
\text{$\prb$-mod} & \text{$\g_0$-mod} & \text{Dimension} \\
\hline
\mfW_{-1}^0 & \mfz_0 \otimes \mfV_{-1} 
& m \\
\mfW_{-1}^1 & \slie_0 \circledcirc \mfV_{-1} 
& \frac{1}{2}m(m+1)(m-2) \\
\mfW_{-1}^2 & \mfV_0 \otimes \g_{-1}
& m \\
\hline
\mfW_0^0 & \mfz_0 \otimes \mfV_0
& 1 \\
\mfW_0^1 & \slie_0 \otimes \mfV_0
& m^2-1 \\
\hline
\end{array}
}
\end{displaymath}
\end{minipage}
\end{center}
with the proviso that $\mfW_{-1}^1$, $\mfW_{-3}^0$ occur only when $m>2$.

Further,
\begin{align*}
\mfW_i^j & = \left\{ \Gamma_{abc} \xi^{bB} \xi^{cC} \in \mfW^i : {}^\mfW _\xi \Pi_i^k (\Gamma) = 0 \, , \mbox{for all $k \neq j$} \right\} / \mfW^{i+1} \, , & i & =-3,-2,-1,0 \, ,
\end{align*}
where
\begin{align*}
 {}^\mfW _\xi \Pi_{-3}^0 (\Gamma) & := \Gamma \ind{_{abc}} \xi \ind*{^{a[A}} \xi \ind*{^{bB}} \xi \ind*{^{cC}} \xi \ind*{^{D]}} \, , \\
 {}^\mfW _\xi \Pi_{-3}^1 (\Gamma) & := \Gamma \ind{_{abc}} \xi \ind*{^{[A}} \xi \ind*{^{aB]}} \xi \ind*{^{b[C}} \xi \ind*{^{D]}} \xi \ind*{^{c[E}} \xi \ind*{^{F]}} + \Gamma \ind{_{abc}} \xi \ind*{^{[C}} \xi \ind*{^{aD]}} \xi \ind*{^{b[A}} \xi \ind*{^{B]}} \xi \ind*{^{c[E}} \xi \ind*{^{F]}} \, , \\
 {}^\mfW _\xi \Pi_{-2}^0 (\Gamma) & := \Gamma \ind{_{abc}} \xi^{aA} \xi^{b[B} \xi^{cC} \xi^{D]} \, , \\
 {}^\mfW _\xi \Pi_{-2}^1 (\Gamma) & := \Gamma \ind{_{abc}} \xi \ind*{^{a[A}} \xi \ind*{^{bcB}} \xi \ind*{^{C]}} \, , \\
 {}^\mfW _\xi \Pi_{-2}^2 (\Gamma) & := \Gamma \ind{_{abc}} \xi \ind*{^{[A}} \xi \ind*{^{aB]}} \xi \ind*{^{b[C}} \xi \ind*{^{D]}} \xi \ind*{^{cE}} + \Gamma \ind{_{abc}} \xi \ind*{^{[C}} \xi \ind*{^{aD]}} \xi \ind*{^{b[A}} \xi \ind*{^{B]}} \xi \ind*{^{cE}} \, , \\
 {}^\mfW _\xi \Pi_{-1}^0 (\Gamma) & := 
 2 \, \gamma \ind{^a_D^A} \Gamma \ind{_{abc}} \xi \ind*{^b^D} \xi \ind*{^c^{[B}} \xi^{C]} - \Gamma \ind{_{abc}} \xi \ind*{^a^A} \xi \ind*{^{bc}^{[B}} \xi \ind*{^{C]}} \, , \\
 {}^\mfW _\xi \Pi_{-1}^1 (\Gamma) & :=
 \Gamma \ind{_{abc}} \xi \ind*{^{b[B}} \xi \ind*{^{cC}} \xi \ind*{^{D]}} + \frac{1}{2(m-1)} \gamma \ind{_a_E^{[B|}} \left(  2 \, \gamma \ind{^d_F^E} \Gamma \ind{_{dbc}} \xi \ind*{^b^F} \xi \ind*{^c^{|C}} - \Gamma \ind{_{dbc}} \xi \ind*{^d^E} \xi \ind*{^{bc}^{|C}} \right) \xi \ind*{^{D]}} \, , \\
 {}^\mfW _\xi \Pi_{-1}^2 (\Gamma) & := \Gamma \ind{_{abc}} \xi \ind*{^{aA}} \xi \ind*{^{bc[B}} \xi \ind*{^{C]}} \, , \\
 {}^\mfW _\xi \Pi_0^0 (\Gamma) & := \gamma \ind{^a_C^A} \Gamma \ind{_{abc}} \xi \ind*{^{bcC}} \xi \ind*{^{B}} - \xi \ind*{^a^A} \Gamma \ind{_{abc}} \xi^{bcB} \, , \\
 {}^\mfW _\xi \Pi_0^1 (\Gamma) & := \Gamma \ind{_{abc}} \xi^{bc[B} \xi^{C]} + \frac{1}{2m} \left( \gamma \ind{^d_A^{[B|}} \Gamma \ind{_{dbc}} \xi \ind*{^{bcA}} \xi \ind*{_a^{|C]}} - \xi \ind*{^d^{[B|}} \Gamma \ind{_{dbc}} \xi \ind*{^{bcA}} \gamma \ind{_a_A^{|C]}} \right) \, , 
\end{align*}
where $\Gamma \ind{_{abc}} \in \mfV \otimes \g$. For $m=2$, we have made use of the $\Spin(5,\C)$-invariant skewsymmetric bilinear forms $\gamma \ind{_{AB}}$ and $\gamma \ind{^{AB}}$.

Finally, the $\prb$-module $\gr(\mfW)$ can be expressed by means of the directed graph
 
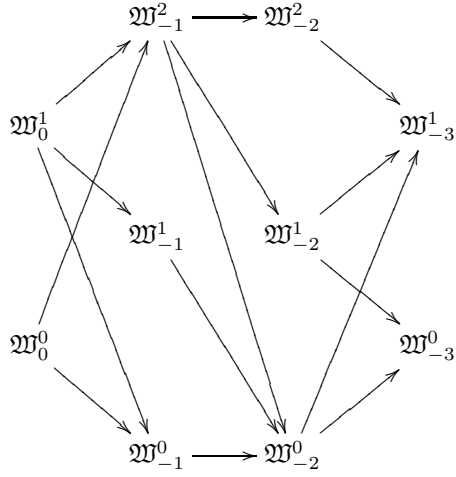
\begin{align*}
 \xymatrix{ & \mfW_{-1}^2 \ar[rdd] \ar[r] \ar[rdddd] & \mfW_{-2}^2 \ar[rd] & \\
 \mfW_0^1 \ar[ru] \ar[rd] \ar[rddd] & & & \mfW_{-3}^1 \\
 & \mfW_{-1}^1 \ar[rdd] & \mfW_{-2}^1 \ar[rd] \ar[ru] & \\
 \mfW_0^0 \ar[uuur] \ar[dr] & & & \mfW_{-3}^0 \\
 & \mfW_{-1}^0 \ar[r] & \mfW_{-2}^0 \ar[ru] \ar[ruuu] & }
\end{align*}
with the proviso that $\mfW_{-1}^1$, $\mfW_{-3}^0$ occur only when $m>2$. Here, an arrow from  $\mfW_i^j$ to $\mfW_{i-1}^k$ for some $i,j,k$ implies that $\breve{\mfW}_i^j \subset \g_1 \cdot \breve{\mfW}_{i-1}^k$ for any choice of irreducible $\g_0$-modules $\breve{\mfW}_i^j$ and $\breve{\mfW}_{i-1}^k$ isomorphic to $\mfW_i^j$ and $\mfW_{i-1}^k$ respectively.
\end{prop}

\begin{proof}
The idea of the proof is to choose a splitting \eqref{eq-null-grading} for $\mfV$, and thus for the filtration on $\mfW$. We can then decompose an element $\Gamma \ind{_{abc}} \in \mfV \otimes \wedge^2 \mfV$, in the obvious notation,
\begin{align}\label{eq-connection1form}
\begin{aligned}
 \Gamma \ind{_{abc}} \xi \ind*{^{bB}} \xi \ind*{^{cC}} & = \xi \ind*{_a^A} \Gamma \ind{_A^{BC}} - u \ind{_a} \Gamma \ind{^{BC}} + \eta \ind{_a_A} \Gamma \ind{^{ABC}} + 2 \, \xi \ind*{_a^A} \Gamma \ind{_{A:}^{[B}} \xi \ind*{^{C]}} - 2 \, u \ind{_a} \Gamma \ind{^{[B}} \xi \ind*{^{C]}} + 2 \, \eta \ind{_a_A} \Gamma \ind{^{A:[B}} \xi \ind*{^{C]}} \, , \\
 \Gamma \ind{_{abc}} \xi \ind*{^{bc}^D} 
&  = \left( \xi \ind*{_a^A} \Gamma \ind{_A^{EC}} - u \ind{_a} \Gamma \ind{^{EC}} 
 + \eta \ind{_a_A} \Gamma \ind{^{AEC}} \right) \eta \ind{_c_C} \gamma \ind{^c_E^D} 
 - 2 \, \left( \xi \ind*{_a^A} \Gamma \ind{_{A:}^D} - u \ind{_a} \Gamma \ind{^D} + \eta \ind{_a_A} \Gamma \ind{^{A:D}} \right) \\
 & \qquad \qquad \qquad \qquad \qquad \qquad \qquad \qquad \qquad \qquad \qquad
 + 2 \, \left( \xi \ind*{_a^A} \Gamma \ind{_{AB}^B} - u \ind{_a} \Gamma \ind{_{B}^B} + \eta \ind{_a_A} \Gamma \ind{^A_{B}^B} \right) \xi^D \, , \\
 \gamma \ind{^a_D^A}  \Gamma \ind{_{abc}} \xi \ind*{^{bc}^D}  
& = 4 \, \Gamma \ind{_{C:}^C} \xi^A + 4 \, \Gamma \ind{_C^{CA}} - 2 \, \Gamma \ind{^A} 
+ \eta \ind{_a_B} \gamma \ind{^a_C^A} \left( \Gamma \ind{^{BC}} - \Gamma \ind{^{B:C}} + \Gamma \ind{^{C:B}} \right) + \eta \ind{_b_B} \eta \ind{_c_C} \Gamma \ind{^{DBC}} \gamma \ind{^{bc}_D^A} \\
&  \qquad \qquad \qquad \qquad \qquad \qquad \qquad \qquad \qquad \qquad \qquad \qquad \qquad \qquad \qquad 
- 2 \, \Gamma \ind{_{B}^B} \xi^A + 2 \, \Gamma \ind{^A_{B}^B} \, .
\end{aligned}
\end{align}
Here, $\Gamma \ind{^{ABC}} := \Gamma \ind{_{abc}} \eta \ind*{^{aA}} \xi \ind*{^{bB}} \xi \ind*{^{cC}}$, $\Gamma \ind{^{BC}} := \Gamma \ind{_{abc}} u^a \xi \ind*{^{bB}} \xi \ind*{^{cC}}$ and $\Gamma \ind{_A^{BC}} := \Gamma \ind{_{abc}} \eta^a_A \xi \ind*{^{bB}} \xi \ind*{^{cC}}$ are skew-symmetric in their last two indices, and the colon $:$ in $\Gamma \ind{_{A:}^C} := \Gamma \ind{_{abc}} \eta \ind*{^a_A} u^b \xi \ind*{^{cC}} $ and $\Gamma \ind{^{A:C}} :=\Gamma \ind{_{abc}} \xi \ind{^{aA}} u^b \xi \ind*{^{cC}}$ separates the $1$-form index from the Lie algebra indices. Then, elements of the $\g_0$-modules $\breve{\mfW}_i^j$ linearly isomorphic  to $\mfW_i^j$  are given by
\begin{align*}
 \Gamma \ind{^{[ABC]}} & \in  \breve{\mfW}_{-3}^0 \, , &  \Gamma \ind{^{(AB)C}} & \in \breve{\mfW}_{-3}^1 \, , \\
 \Gamma \ind{^{AB}} & \in  \breve{\mfW}_{-2}^0 \, , &  \Gamma \ind{^{[A:B]}} & \in \breve{\mfW}_{-2}^1 \, , &  \Gamma \ind{^{(A:B)}} & \in \breve{\mfW}_{-2}^2 \, , \\
 \Gamma \ind{_B^{BA}} & \in \breve{\mfW}_{-1}^0 \, , &
 \Gamma \ind{_A^{BC}} - \frac{2}{m-1} I \ind*{_A^{[B|}} \Gamma \ind{_D^{D|C]}} & \in \breve{\mfW}_{-1}^1 \, ,  & \Gamma \ind{^A} & \in \breve{\mfW}_{-1}^2 \, ,  \\
 \Gamma \ind{_{A:}^A} & \in \breve{\mfW}_0^0 \, , & \Gamma \ind{_{A:}^B} - \frac{1}{m} I \ind*{_A^B} \Gamma \ind{_{C:}^C} & \in \breve{\mfW}_0^1 \, .
\end{align*}
Details are analogous to the even-dimensional case, and are left to reader.
\end{proof}

\section{Decomposition of the curvature}\label{sec-curvature}
Assume $m>1$, and consider the following $\g$-modules
\begin{displaymath}
{\renewcommand{\arraystretch}{1.5}
\begin{array}{||c|c|c||}
\hline
\text{$\g$-mod} & \text{Dimension} & \text{Description} \\
\hline
\mfF  & 
 m(2m+3) &
 \{ \Phi_{ab} \in \otimes^2 \mfV^* : \Phi_{ab} = \Phi_{(ab)} \, , \Phi \ind{_c^c} = 0 \}
 \\[.5em]
\hline
\mfA & \frac{1}{3}(2m-1)(2m+1)(2m+3) & 
 \{ A_{abc} \in \otimes^3 \mfV : A_{abc} = A_{a[bc]} \, , A_{[abc]} = 0 \, , A \ind{^a_{ac}} = 0 \}
\\[.5em]
\hline
\mfC & \frac{1}{3}(m-1)(m+1)(2m+1)(2m+3) &
\{ C_{abcd} \in \otimes^4 \mfV : C_{abcd} = C_{[ab][cd]} \, , C_{[abc]d} = 0 \, , C \ind{^a_{bad}} = 0 \} \\[.5em]
\hline
\end{array}}
\end{displaymath}
The tracefree Ricci tensor, Cotton-York tensor and the Weyl tensor of a Levi-Civita connection at a point belong to $\mfF$, $\mfA$ and $\mfC$ respectively. We now give $\prb$-invariant decompositions of these modules, where $\prb$ stabilises a projective pure spinor $[ \xi^A ]$ as described in section \ref{sec-algebra}.

\subsection{Decomposition of the space of tracefree Ricci tensors}

\begin{prop}\label{prop-TFRicci-classification}
The space $\mfF$ of tracefree symmetric $2$-tensors admits a filtration
\begin{align*}
\{ 0 \} =: \mfF^3 \subset \mfF^2 \subset \mfF^1 \subset \mfF^0 \subset \mfF^{-1} \subset \mfF^{-2} := \mfF \, ,
\end{align*}
of $\prb$-modules
\begin{align*}
 \mfF^i & := \{ \Phi \ind{_{ab}} \in \mfF : {}^\mfF _\xi \Pi_{i-1}^k (\Phi) = 0 \, , \, \, \mbox{for all $k$} \} \, , & i=-1,0,1,2,
\end{align*}
where the maps ${}^\mfF _\xi \Pi_i^j$ are defined in appendix \ref{sec-inv-maps}.
 
The associated graded $\prb$-module $\gr(\mfF)=\bigoplus_{i=-2}^2 \gr_i(\mfF)$, where $\gr_i(\mfF) := \mfF^i/\mfF^{i+1}$, splits into a direct sum
\begin{align*}
 \gr_{\pm2} (\mfF) & = \mfF_{\pm2}^0 \, , & \gr_{\pm1} (\mfF) & = \mfF_{\pm1}^0 \, , &
 \gr_0 (\mfF) & = \mfF_0^0 \oplus \mfF_0^1 \, , 
\end{align*}
of irreducible $\prb$-modules $\mfF_i^j$ as described below:
\vspace{-6.5mm}
\begin{center}
\begin{minipage}[b]{0.45\linewidth}\centering
\begin{displaymath}
{\renewcommand{\arraystretch}{1.5}
\begin{array}{||c|c|c||}
\hline
\text{$\prb$-mod} & \text{$\g_0$-mod} & \text{Dimension} \\
\hline
\mfF_{\pm2}^0 & \mfV_{\pm1} \circledcirc \mfV_{\pm1} & \frac{1}{2}m(m+1)  \\
\hline
\mfF_{\pm1}^0 & \mfV_0 \circledcirc \mfV_{\pm1}  & m \\
\hline
\end{array}}
\end{displaymath}
\end{minipage}
\begin{minipage}[b]{0.45\linewidth}\centering
\begin{displaymath}
{\renewcommand{\arraystretch}{1.5}
\begin{array}{||c|c|c||}
\hline
\text{$\prb$-mod} & \text{$\g_0$-mod} & \text{Dimension} \\
\hline
\mfF_0^0 & \mfV_0 \circledcirc \mfV_0 & 1 \\
\mfF_0^1 & \mfV_{\pm1} \circledcirc \mfV_{\mp1} & m^2-1 \\
\hline
\end{array}}
\end{displaymath}
\end{minipage}
\end{center}
Further,
\begin{align*}
 \mfF_0^j & := \{ \Phi \ind{_{ab}} \in \mfF^0 : {}^\mfF _\xi \Pi_0^k(\Phi) = 0 \, , \, \, \mbox{for $k \neq j$} \} / \mfF^1 \, .
\end{align*}

Finally, the $\prb$-module $\gr(\mfF)$ can be expressed by means of the directed graph
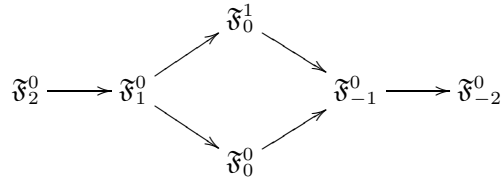
\begin{equation*}
 \xymatrix@R=1em{ & & \mfF_0^1 \ar[dr] & & \\
\mfF_2^0 \ar[r] & \mfF_1^0 \ar[ur] \ar[dr] &  & \mfF_{-1}^0 \ar[r] & \mfF_{-2}^0 \\
& & \mfF_0^0 \ar[ur] & & }
\end{equation*}
where an arrow from $\mfF_i^j$ to $\mfF_{i-1}^k$ for some $i,j,k$ implies that $\breve{\mfF}_i^j \subset \g_1 \cdot \breve{\mfF}_{i-1}^k$ for any choice of irreducible $\g_0$-modules $\breve{\mfF}_i^j$ and $\breve{\mfF}_{i-1}^k$ isomorphic to $\mfF_i^j$ and $\mfF_{i-1}^k$ respectively, or equivalently that $\ker  {}^\mfF _\xi \Pi_i^j \subset \ker  {}^\mfF _\xi \Pi_{i-1}^k$.
\end{prop}

\subsection{Decomposition of the space of Cotton-York tensors}

\begin{prop}\label{prop-Cotton-York-classification}
The space $\mfA$ of tensors with Cotton-York symmetries admits a filtration
\begin{align*}
\{ 0 \} =: \mfA^4 \subset \mfA^3 \subset \mfA^2 \subset \mfA^1 \subset \mfA^0 \subset \mfA^{-1} \subset \mfA^{-2} \subset \mfA^{-3} := \mfA \, ,
\end{align*}
of $\prb$-modules
\begin{align*}
 \mfA^i & = \{ A_{abc} \in \mfA : {}^\mfA _\xi \Pi_{i-1}^k(A) = 0 \, , \, \, \mbox{for all $k$} \} \, ,  & i & = -2,-1,0,1,2,3,
\end{align*}
where the maps ${}^\mfC _\xi \Pi_i^j$ are defined in appendix \ref{sec-inv-maps}.

The associated graded $\prb$-module $\gr(\mfA)=\bigoplus_{i=-3}^3 \gr_i(\mfA)$, where $\gr_i(\mfA) := \mfA^i/\mfA^{i+1}$, splits into a direct sum
\begin{align*}
 \gr_{\pm3} (\mfA) & = \mfA_{\pm3}^0 \, , & 
 \gr_{\pm2} (\mfA) & = \mfA_{\pm2}^0 \oplus \mfA_{\pm2}^1 \, , \\
 \gr_{\pm1} (\mfA) & = \mfA_{\pm1}^0 \oplus \mfA_{\pm1}^1 \oplus \mfA_{\pm1}^2 \oplus \mfA_{\pm1}^3 \, , &
 \gr_0 (\mfA) & = \mfA_0^0 \oplus \mfA_0^1 \oplus \mfA_0^2 \, ,
\end{align*}
of irreducible $\prb$-modules $\mfA_i^j$ as described below:
\vspace{-6.5mm}
\begin{center}
\begin{minipage}[b]{0.45\linewidth}\centering
\begin{displaymath}
{\renewcommand{\arraystretch}{1.5}
\begin{array}{||c|c|c||}
\hline
\text{$\prb$-mod} & \text{$\g_0$-mod} & \text{Dimension} \\
\hline
\mfA_{\pm3}^0 & \mfV_{\pm1} \circledcirc \g_{\pm2} & \frac{1}{3}m(m+1)(m-1) \\
\hline
\mfA_{\pm2}^0 & \mfV_0 \circledcirc \g_{\pm2} & \frac{1}{2}m(m-1) \\
\mfA_{\pm2}^1 & \mfV_{\pm1} \circledcirc \g_{\pm1} & \frac{1}{2}m(m+1) \\
\hline
\mfA_{\pm1}^0 & \mfV_{\pm1} \circledcirc \mfz_0 & m \\
\mfA_{\pm1}^1 & \mfV_0 \circledcirc \g_{\pm1} & m \\
\mfA_{\pm1}^2 & \mfV_{\mp1} \circledcirc \g_{\pm2} & \frac{1}{2}m(m-2)(m+1) \\
\mfA_{\pm1}^3 & \mfV_{\pm1} \circledcirc \slie_0 & \frac{1}{2}m(m+2)(m-1) \\
\hline
\end{array}}
\end{displaymath}
\end{minipage}
\begin{minipage}[b]{0.45\linewidth}\centering
\begin{displaymath}
{\renewcommand{\arraystretch}{1.5}
\begin{array}{||c|c|c||}
\hline
\text{$\prb$-mod} & \text{$\g_0$-mod} & \text{Dimension} \\
\hline
\mfA_0^0 & \mfV_0 \circledcirc \mfz_0 & 1 \\
\mfA_0^1 & \mfV_0 \circledcirc \slie_0 & (m-1)(m+1) \\
\mfA_0^2 & \mfV_1 \circledcirc \g_{-1} & (m-1)(m+1) \\
\hline
\end{array}}
\end{displaymath}
\end{minipage}
\end{center}
with the proviso that when $m=2$,  $\mfA_{\pm1}^2$ does not occur.  Further,
\begin{align*}
 \mfA_i^j & = \{ A_{abc} \in \mfA^i : {}^\mfA _\xi \Pi_i^k(A) = 0 \, , \, \, \mbox{for all $k \neq j$} \} / \mfA^{i+1} \, ,  & \mbox{for $|i|\leq 2$.}
\end{align*}

Finally, the $\prb$-module $\gr(\mfA)$ can be expressed by means of the directed graph
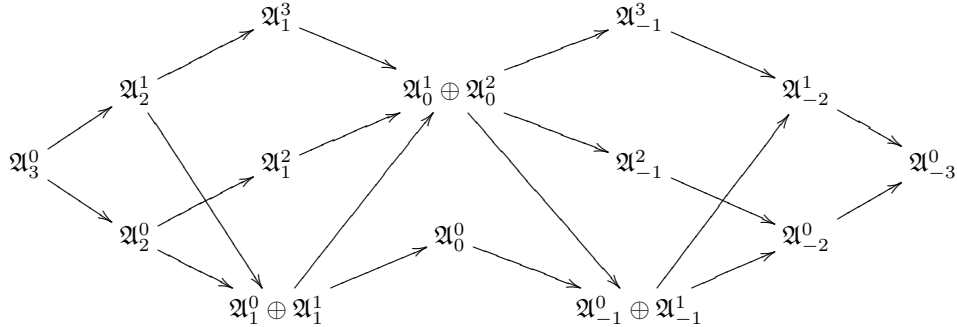
\begin{equation*}
 \xymatrix@R=1em{
	    & & \mfA_1^3 \ar[dr] &  & \mfA_{-1}^3 \ar[dr] &    &  \\
&  \mfA_2^1 \ar[ur] \ar[dddr]   &  & \mfA_0^1 \oplus \mfA_0^2 \ar[ur] \ar[dddr] \ar[dr] & &  \mfA_{-2}^1 \ar[dr]        	& \\
\mfA_3^0 \ar[ur] \ar[dr]   &    &  \mfA_1^2 \ar[ur] & &  \mfA_{-1}^2 \ar[dr] &  	& \mfA_{-3}^0  \\
&  \mfA_2^0 \ar[dr] \ar[ur]  &  & \mfA_0^0 \ar[dr]  & & \mfA_{-2}^0 \ar[ur] 	& \\
	    & & \mfA_1^0 \oplus \mfA_1^1 \ar[ur] \ar[uuur] &  & \mfA_{-1}^0 \oplus \mfA_{-1}^1 \ar[ur] \ar[uuur] &        &  }
\end{equation*}
where an arrow from  $\mfA_i^j$ to $\mfA_{i-1}^k$ for some $i,j,k$ implies that $\breve{\mfA}_i^j \subset \g_1 \cdot \breve{\mfA}_{i-1}^k$ for any choice of irreducible $\g_0$-modules $\breve{\mfA}_i^j$ and $\breve{\mfA}_{i-1}^k$ isomorphic to $\mfA_i^j$ and $\mfA_{i-1}^k$ respectively.
\end{prop}

\begin{rem}\label{rem-isotopicA}
The presence of the isotopic pairs of $\prb$-modules $\{\mfA_{\pm1}^0,\mfA_{\pm1}^1\}$ and $\{\mfA_0^1,\mfA_0^2\}$ in the decomposition of $\gr(\mfA)$ allows us to define further $\prb$-submodules whereby there are algebraic relations among them. For instance, one distinguish $\{ A_{abc} \in \mfA_1^0 \oplus \mfA_1^1 : {}^\mfA _\xi \Pi_2^0 (A) = 0 \}$ and $\{ A_{abc} \in \mfA_1^0 \oplus \mfA_1^1 : {}^\mfA _\xi \Pi_2^1 (A) = 0 \}$. In particular, it is certainly not true that $\ker {}^\mfA _\xi \Pi_2^1 \subset \ker {}^\mfA _\xi \Pi_1^0$ or $\ker {}^\mfA _\xi \Pi_2^1 \subset \ker {}^\mfA _\xi \Pi_1^1$, and so on. It thus makes it difficult to characterise the arrows of the diagram in terms of inclusions of kernels of $\ker {}^\mfA _\xi \Pi_j^i$ as we did in \cite{Taghavi-Chabert2016}.

\end{rem}

\subsection{Decomposition of the space of Weyl tensors}

\begin{prop}\label{prop-Weyl-classification}
The space $\mfC$ of tensors with Weyl symmetries admits a filtration
\begin{align*}
 \{ 0 \} =: \mfC^5 \subset \mfC^4 \subset \mfC^3 \subset \mfC^2 \subset \mfC^1 \subset \mfC^0 \subset \mfC^{-1} \subset \mfC^{-2} \subset \mfC^{-3} \subset \mfC^{-4} := \mfC \, ,
\end{align*}
of $\prb$-modules
\begin{align*}
 \mfC^i & = \{ C_{abcd} \in \mfC : {}^\mfC _\xi \Pi_i^k(C) = 0 \, , \, \, \mbox{for all $k$} \} \, , & i = -3,-2,-1,0,1,2,3,4,
\end{align*}
where the maps ${}^\mfC _\xi \Pi_i^j$ are defined in appendix \ref{sec-inv-maps}.

The associated graded $\prb$-module $\gr(\mfC)=\bigoplus_{i=-4}^4 \gr_i(\mfC)$, where $\gr_i(\mfC) := \mfC^i/\mfC^{i+1}$, splits into a direct sum
\begin{align*}
 \gr_{\pm4} (\mfC) & = \mfC_{\pm4}^0 \, , &
 \gr_{\pm3} (\mfC) & = \mfC_{\pm3}^0 \, , & 
 \gr_{\pm2} (\mfC) & = \mfC_{\pm2}^0 \oplus \mfC_{\pm2}^1 \oplus \mfC_{\pm2}^2 \, , \\
 \gr_{\pm1} (\mfC) & = \mfC_{\pm1}^0 \oplus \mfC_{\pm1}^1 \oplus \mfC_{\pm1}^2 \, , &
 \gr_0 (\mfC) & = \mfC_0^0 \oplus \mfC_0^1 \oplus \mfC_0^2 \oplus \mfC_0^3 \oplus \mfC_0^4  \, ,
\end{align*}
of irreducible $\prb$-modules $\mfC_i^j$ as described below:
\vspace{-6.5mm}
\begin{center}
\begin{minipage}[b]{0.45\linewidth}\centering
\begin{displaymath}
{\renewcommand{\arraystretch}{1.5}
\begin{array}{||c|c|c||}
\hline
\text{$\prb$-mod} & \text{$\g_0$-mod} & \text{Dimension} \\
\hline
\mfC_{\pm4}^0 & \g_{\pm2} \circledcirc \g_{\pm2} & \frac{1}{12}m^2 (m^2-1) \\
\hline
\mfC_{\pm3}^0 & \g_{\pm1} \circledcirc \g_{\pm2} & \frac{1}{3}m (m^2-1) \\
\hline
\mfC_{\pm2}^0 & \mfz_0 \circledcirc \g_{\pm2} & \frac{1}{2}m (m-1) \\
\mfC_{\pm2}^1 & \g_{\pm1} \circledcirc \g_{\pm1} & \frac{1}{2} m (m +1) \\
\mfC_{\pm2}^2 & \slie_0 \circledcirc \g_{\pm2} & \frac{1}{3} m^2 (m^2 - 4) \\
\hline
\mfC_{\pm1}^0 & \mfz_0 \circledcirc \g_{\pm1} & m \\
\mfC_{\pm1}^1 & \g_{\mp1} \circledcirc \g_{\pm2} & \frac{1}{2} m (m-2) (m + 1)\\
\mfC_{\pm1}^2 & \slie_0 \circledcirc \g_{\pm1} & \frac{1}{2} m (m+2) (m - 1) \\
\hline
\end{array}}
\end{displaymath}
\end{minipage}
\begin{minipage}[b]{0.45\linewidth}\centering
\begin{displaymath}
{\renewcommand{\arraystretch}{1.5}
\begin{array}{||c|c|c||}
\hline
\text{$\prb$-mod} & \text{$\g_0$-mod} & \text{Dimension} \\
\hline
\mfC_0^0 & \mfz_0 \circledcirc \mfz_0 & 1 \\
\mfC_0^1 & \slie_0 \circledcirc \mfz_0 & m^2 - 1 \\
\mfC_0^2 & \g_1 \circledcirc \g_{-1} & m^2 - 1 \\
\mfC_0^3 & \g_2 \circledcirc \g_{-2} & \frac{1}{4}m^2 (m+1) (m-3) \\
\mfC_0^4 & \slie_0 \circledcirc \slie_0 & \frac{1}{4}m^2 (m-1) (m+3) \\
\hline
\end{array}}
\end{displaymath}
\end{minipage}
\end{center}
with the proviso that when $m=2$, the modules $\mfC_{\pm2}^2$, $\mfC_{\pm1}^1$, $\mfC_0^1$ and $\mfC_0^3$ do not occur, and when $m=3$, the module $\mfC_0^3$ does not occur. Further,
\begin{align*}
 \mfC_i^j & = \{ C_{abcd} \in \mfC^i : {}^\mfC _\xi \Pi_i^k(C) = 0 \, , \, \, \mbox{for all $k \neq j$} \} / \mfC^{i+1} \, , & \mbox{for $|i|\leq 3$.}
\end{align*}

Finally, the $\prb$-module $\gr(\mfC)$ can be expressed by means of the directed graph
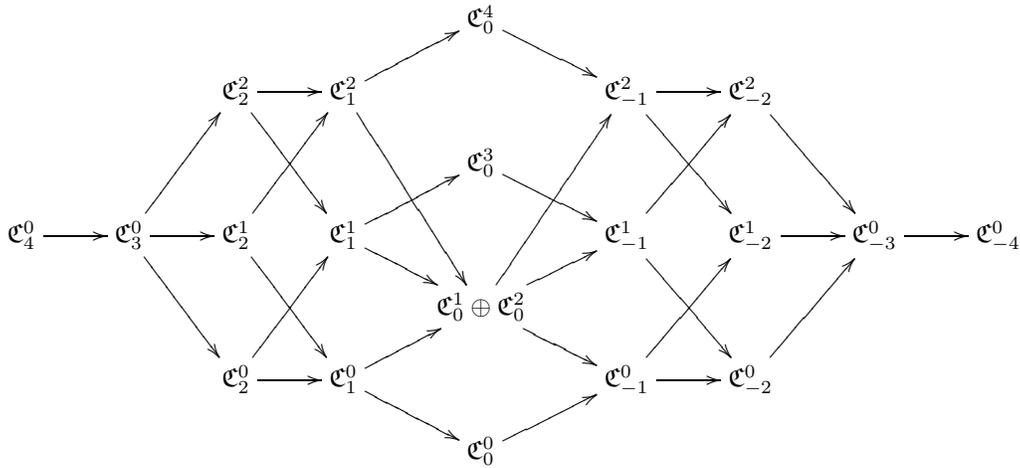
\begin{equation*}
\xymatrix@R=1em{ & & & & & & & & \\
& & & & \mfC_0^4 \ar[dr]   & & & & \\
& & \mfC_2^2 \ar[ddr] \ar[r] & \mfC_1^2 \ar[dddr]  \ar[ur] & & \mfC_{-1}^2 \ar[r] \ar[ddr] & \mfC_{-2}^2 \ar[ddr] & & \\
 & &  &  & \mfC_0^3 \ar[dr] &  &  & & \\
\mfC_4^0 \ar[r] & \mfC_3^0 \ar[ddr] \ar[uur] \ar[r] & \mfC_2^1 \ar[uur] \ar[ddr] & \mfC_1^1 \ar[dr] \ar[ur] & & \mfC_{-1}^1 \ar[uur] \ar[ddr] & \mfC_{-2}^1 \ar[r] & \mfC_{-3}^0 \ar[r] & \mfC_{-4}^0 \\
 & &  &  & \mfC_0^1 \oplus \mfC_0^2 \ar[uuur] \ar[dr] \ar[ur]  &  &  & & \\
& & \mfC_2^0 \ar[r] \ar[uur] & \mfC_1^0 \ar[dr] \ar[ur] & & \mfC_{-1}^0 \ar[r] \ar[uur] & \mfC_{-2}^0 \ar[uur] & & \\
& & & & \mfC_0^0 \ar[ur]  & & & & \\
& & & & & & & &}
\end{equation*}
where an arrow from  $\mfC_i^j$ to $\mfC_{i-1}^k$ for some $i,j,k$ implies that $\breve{\mfC}_i^j \subset \g_1 \cdot \breve{\mfC}_{i-1}^k$ for any choice of irreducible $\g_0$-modules $\breve{\mfC}_i^j$ and $\breve{\mfC}_{i-1}^k$ isomorphic to $\mfC_i^j$ and $\mfC_{i-1}^k$ respectively.
\end{prop}

\begin{rem}
Analogous to Remark \ref{rem-isotopicA}, one can define additional $\prb$-submodules from the isotopic pair of $\prb$-modules $\{\mfC_0^1,\mfC_0^2\}$. For instance, one has $\{ C_{abcd} \in \mfC_0^1 \oplus \mfC_0^2 : {}^\mfC _\xi \Pi_1^0 (C) = 0 \}$ and $\{ C_{abcd} \in \mfC_0^1 \oplus \mfC_0^2 : {}^\mfC _\xi \Pi_1^1 (C) = 0 \}$, and so on. Again, it is not true that $\ker {}^\mfC _\xi \Pi_1^1 \subset \ker {}^\mfC _\xi \Pi_0^1$ or $\ker {}^\mfC _\xi \Pi_1^1 \subset \ker {}^\mfC _\xi \Pi_0^2$. This is why we have not characterise the arrows of the diagram in terms of inclusions of kernels of $\ker {}^\mfC _\xi \Pi_j^i$ unlike in \cite{Taghavi-Chabert2016}. Proposition \ref{prop-recurrent-spinor} in section \ref{sec-geometry} will illustrate the issue.
\end{rem}

\section{Differential geometry of pure spinor fields}\label{sec-geometry}
As before, conventions are taken from \cite{Taghavi-Chabert2016} and references therein. Throughout, $(\mcM,g)$ will denote an $n$-dimensional oriented complex Riemannian manifold, where $n=2m+1$, with holomorphic tangent denoted by $\Tgt \mcM$ and so on. The holomorphic Levi-Civita connection will be denoted $\nabla_a$, the Riemann tensor $R \ind{_{abcd}}$, the Weyl tensor $C_{abcd}$, the Ricci tensor $R_{ab}$, with tracefree part $\Phi_{ab}$, and the Ricci scalar $R$, their relation being given by
\begin{align}\label{eq-Riem_decomposition}
R \ind{_{abcd}} & =
C \ind{_{abcd}} + \frac{4}{n-2} \Phi \ind{_{[a|[c}} g \ind{_{d]|b]}} + \frac{2}{n(n-1)} R \, g \ind{_{a[c}} g \ind{_{d]b}} \, .
\end{align}
In dimension $n=3$, the Weyl tensor vanishes identically, i.e.\ $R \ind{_{abcd}} = 4 \, \Phi \ind{_{[a|[c}} g \ind{_{d]|b]}} + \frac{1}{3} R \, g \ind{_{a[c}} g \ind{_{d]b}}$.

%

We assume $(\mcM,g)$ to be spin so that the structure group of the frame bundle $\mcF \mcM$ of $\mcM$ is $\Spin(2m+1,\C)$. The connection on the spinor bundle $\mcS$ will also be denoted $\nabla_a$, and preserves the Clifford module structure of $\mcS$, i.e.\ $\nabla_a \gamma \ind{_b_C^D} = 0$, and recall that $2 \, \nabla \ind{_{\lb{a}}} \nabla \ind{_{\rb{b}}} \xi^A = - \frac{1}{4} R \ind{_{abcd}} \gamma \ind{^{cd} _B^A} \xi^B$ for any holomorphic spinor field $\xi \ind*{^A}$, and similarly for dual spinor fields.

\begin{rem}[Notation]
As in the previous sections, we shall make use of the short-hand notation $\xi \ind*{_{a_1 \ldots a_k} ^A} := \xi \ind*{^B} \gamma \ind{_{a_1 \ldots a_k} _B ^A}$ for any holomorphic spinor field $\xi^A$ and any $k>0$.
\end{rem}

\begin{ass}
We work in the holomorphic category throughout, and $\Gamma( \cdot)$ denotes the space of holomorphic sections of a holomorphic fiber bundle. See section \ref{sec-ps-Riem} for extensions to real manifolds.

Henceforth, we assume $n>3$ for definiteness, relegating the case $n=3$ to appendix \ref{sec-3dim}. Nonetheless, many of the statements made in this section still apply by setting $C \ind{_{abcd}} = 0$.

Finally, we stress that the results presented herein are local in nature.
\end{ass}

\subsection{Projective pure spinor fields}
\begin{defn}
An \emph{almost null structure} $\mcN$ on $(\mcM,g)$ is a rank-$m$ distribution that is totally null, i.e.\ $g(v,w)=0$ for all sections $v$, $w$ of $\mcN$.
\end{defn}

An almost null structure $\mcN$ will also be referred to as a $\gamma$-plane distribution. The orthogonal complement $\mcN^\perp$ of $\mcN$ is a rank-$(m+1)$ subbundle of $\Tgt \mcM$ that contains $\mcN$. The bundle of all almost null structures on $(\mcM,g)$ will be denoted $\Gr_m (\Tgt \mcM, g)$. We can use the spin structure on $(\mcM,g)$ to identify an almost null structure as a projective pure spinor field, i.e.\ a spinor field defined up to scale and which is pure at every point.



Now, let $[\xi^A]$ be a holomorphic projective pure spinor field on $\mcM$, i.e.\ a (global) holomorphic section of $\Gr_m (\Tgt \mcM,g)$, with associated holomorphic almost null structure $\mcN_\xi$ and orthogonal complement $\mcN_\xi^\perp$. This geometric data is equivalent to a reduction of the structure group of $\mathrm{F} \mcM$ to the stabiliser $P$ of $[\xi^A]$. The representation theory of $P$, or of its Lie algebra $\prb$, which we have described in sections \ref{sec-algebra}, \ref{sec-curvature} and \ref{sec-intrinsic-torsion}, gives rise to holomorphic vector bundles in the standard way as already explicated in \cite{Taghavi-Chabert2016}. In particular, the pointwise algebraic degeneracy of the curvature tensors will be expressed in terms of the maps ${}^\mfF _\xi \Pi_i^j$, ${}^\mfA _\xi \Pi_i^j$ and ${}^\mfC _\xi \Pi_i^j$ given in Appendix \ref{sec-inv-maps}.

\subsubsection{Intrinsic torsion}
For simplicity, we choose a holomorphic connection $1$-form $\Gamma \ind{_{ab}^c}$ for $\nabla_a$ such that
\begin{align}\label{eq-nabla-spin-Gamma}
\nabla_a \xi^A & = -\frac{1}{4} \Gamma \ind{_{abc}} \gamma \ind{^{bc}_B^A} \xi^B \, .
\end{align}
We can identify the notion of \emph{intrinsic torsion} \cites{Chern1953,Bernard1960,Salamon1989} of the $P$-structure defined by $[\xi^A]$ with
\begin{align*}
\Gamma \ind{_{abc}} \xi \ind*{^b^B} \xi \ind*{^c^C} \qquad \in \qquad \mfV^{-1} \otimes \wedge^2 \mfS^{\frac{m-2}{2}} \, ,
\end{align*}
which, at a point, we identify as an element of the $\prb$-module $\mfW := \mfV \otimes \g/\prb$ defined in section \ref{sec-intrinsic-torsion}.
When the intrinsic torsion vanishes, the Levi-Civita connection preserves $[\xi^A]$,
i.e. \begin{align}\label{eq-recurrent}
\nabla \ind{_a} [ \xi \ind*{^A} ] & = 0 \, , & \mbox{i.e.} & &  \nabla \ind{_a} \xi \ind*{^A} & = \alpha \ind{_a} \xi \ind*{^A} \, ,
\end{align}
for some $1$-form $\alpha_a$. If the intrinsic torsion does not vanish, we can nevertheless investigate the differential and geometric properties of $[\xi^A]$, $\mcN_\xi$ and $\mcN_\xi^\perp$ in terms of the decomposition of $\mfW$ given in Proposition \ref{prop-intorsion-class-alg}.

Before we proceed, we compute, from \eqref{eq-nabla-spin-Gamma} and \eqref{eq-useful-id}, the formula
\begin{align*}
 \left( \nabla \ind{_a} \xi \ind*{^b^B} \right) \xi \ind*{_b^C} = - \left( \nabla \ind{_a} \xi \ind*{^B} \right) \xi \ind*{^C} + \Gamma \ind{_{abc}} \xi \ind*{^b^B} \xi \ind*{^c^C} \, 
\end{align*}
from which we deduce
\begin{align*}
\left( \nabla \ind{_a} \xi \ind*{^b^{(B}} \right) \xi \ind*{_b^{C)}} & = - \left( \nabla \ind{_a} \xi \ind*{^{(B}} \right) \xi \ind*{^{C)}} \, , &
 \left( \nabla \ind{_a} \xi \ind*{^b^{[B}} \right) \xi \ind*{_b^{C]}} & = - \left( \nabla \ind{_a} \xi \ind*{^{[B}} \right) \xi \ind*{^{C]}} + \Gamma \ind{_{abc}} \xi \ind*{^b^B} \xi \ind*{^c^C} \, , \\
  \left( \nabla \ind{_a} \xi \ind*{^b^{[B}} \right) \xi \ind*{_b^C} \xi \ind*{^{D]}} & = \Gamma \ind{_{abc}} \xi \ind*{^b^{[B}} \xi \ind*{^c^C} \xi \ind*{^{D]}} \, , &
  \left( \nabla \ind{_a} \xi \ind*{^b^{[A}} \right)  \xi \ind*{^{B]}} \xi \ind*{_b^{[C}} \xi \ind*{^{D]}} & = \Gamma \ind{_{abc}} \xi \ind*{^b^{[A}} \xi \ind*{^{B]}} \xi \ind*{^c^{[C}} \xi^{D]} \, .
\end{align*}
The first of these identities is trivially satisfied by virtue of the purity condition. These formulae together with Proposition \ref{prop-intorsion-class-alg} prove the following result.

\begin{prop}\label{prop-intorsion-class}
 Let $[ \xi^A ]$ be a holomorphic projective pure spinor field on $(\mcM,g)$, and let $\Gamma _{abc} \xi^{bB} \xi^{cC} \in \mfW$ be its associated intrinsic torsion. Then, pointwise,
 \begin{itemize}
  \item ${}^\mfW _\xi \Pi_{-3}^0 (\Gamma ) = 0$ if and only if ($m>2$ only)
  \begin{align}\label{eq-cond-30}
  \left( \xi \ind*{^{a[A}} \nabla \ind{_a} \xi \ind*{^b^{B}} \right) \xi \ind*{_b^C} \xi^{D]} & = 0 \, ;
  \end{align}
  \item ${}^\mfW _\xi \Pi_{-3}^1 (\Gamma ) = 0$ if and only if
  \begin{align}\label{eq-cond-31}
      \xi \ind*{^{[A}} \left( \xi \ind*{^{aB]}} \nabla \ind{_a} \xi \ind*{^b^{[C}} \right) \xi^{D]} \xi \ind*{_b^{[E}} \xi^{F]} + \xi \ind*{^{[C}} \left( \xi \ind*{^{aD]}} \nabla \ind{_a} \xi \ind*{^b^{[A}} \right) \xi^{B]} \xi \ind*{_b^{[E}} \xi^{F]} & = 0 \, ;
  \end{align}
  \item ${}^\mfW _\xi \Pi_{-2}^0 (\Gamma ) = 0$ if and only if
    \begin{align}\label{eq-cond-20}
      \left( \xi \ind*{^{aA}} \nabla \ind{_a} \xi \ind*{^b^{[B}} \right) \xi \ind*{_b^{C}} \xi^{D]} & = 0 \, ;
  \end{align}
  \item ${}^\mfW _\xi \Pi_{-2}^1 (\Gamma ) = 0$ if and only if
  \begin{align}\label{eq-cond-21}
   \left( \xi \ind*{^{a[A}} \nabla \ind{_a} \xi \ind*{^{B}} \right) \xi^{C]} & = 0 \, ;
  \end{align}
  \item ${}^\mfW _\xi \Pi_{-2}^2 (\Gamma ) = 0$ if and only if
    \begin{align}\label{eq-cond-22}
  \xi \ind*{^{[A}} \left( \xi \ind*{^{aB]}} \nabla \ind{_a} \xi \ind*{_b^E}  \right) \xi \ind*{^b^{[C}} \xi \ind*{^{D]}} + \xi \ind*{^{[C}} \left( \xi \ind*{^{aD]}} \nabla \ind{_a} \xi \ind*{_b^E} \right) \xi \ind*{^b^{[A}} \xi^{B]} &  = 0 \, ;
  \end{align}
  \item ${}^\mfW _\xi \Pi_{-1}^0 (\Gamma ) = 0$ if and only if
    \begin{align}
   \left( \gamma \ind{^a_D^A} \nabla \ind{_a} \xi \ind*{^{bD}} \right) \xi \ind*{_b^{[B}} \xi^{C]} + 2 \left( \xi \ind*{^{aA}} \nabla \ind{_a} \xi \ind*{^{[B}} \right) \xi^{C]} & = 0 \, ; \label{eq-cond-10} 
  \end{align}
  \item ${}^\mfW _\xi \Pi_{-1}^1 (\Gamma ) = 0$ if and only if ($m>2$ only)
  \begin{align}\label{eq-cond-11}
      \left( \nabla \ind{_a} \xi \ind*{^b^{[B}} \right) \xi \ind*{_b^{C}} \xi^{D]} + \frac{1}{m-1} \gamma \ind{_a_E^{[B}} \left( \left( \gamma \ind{^c_F^E} \nabla \ind{_c} \xi \ind*{^{bF}} \right) \xi \ind*{_b^{|C}} + 2 \left( \xi \ind*{^{bE}} \nabla \ind{_b} \xi \ind*{^{|C}} \right) \right) \xi \ind*{^{D]}} & = 0 \, ;
  \end{align}
  \item ${}^\mfW _\xi \Pi_{-1}^2 (\Gamma ) = 0$ if and only if
  \begin{align}\label{eq-cond-12}
   \left( \xi \ind*{^{aA}} \nabla \ind{_a} \xi \ind*{^{[B}} \right) \xi^{C]} & = 0 \, ;
  \end{align}
  \item ${}^\mfW _\xi \Pi_0^0 (\Gamma ) = 0$ if and only if
  \begin{align}\label{eq-cond00}
   \left( \nabla \ind{_a} \xi \ind*{^{aA}} \right) \xi \ind*{^B} - \xi \ind*{^{aA}} \nabla \ind{_a} \xi \ind*{^B}  & = 0 \, ;
  \end{align}
  \item ${}^\mfW _\xi \Pi_0^1 (\Gamma ) = 0$ if and only if
  \begin{align}\label{eq-cond01}
 \left( \nabla \ind{_a} \xi \ind*{^{[B}} \right) \xi^{C]} - \frac{2}{m} \left( \left( \nabla \ind{_b} \xi \ind*{^{b[B}} \right) \xi \ind*{_a^{C]}} - \xi \ind*{^{b[B}}  \nabla \ind{_b} \xi \ind*{_a^{C]}} \right) & = 0 \, .
  \end{align}
 \end{itemize}
These statements are independent of the scale of $\xi^A$.
\end{prop}

\begin{rem}
The case $m=2$, i.e. $n=5$, is also dealt separately in Appendix \ref{sec-5dim}, where the spinor calculus simplifies the formulae above.
\end{rem}

\subsubsection{Geometric properties}
\begin{defn}
An almost null structure $\mcN$ is said to be \emph{integrable} if $[\Gamma(\mcN), \Gamma(\mcN)] \subset \Gamma(\mcN)$, \emph{totally geodetic} if $\nabla_X Y \in \Gamma( \mcN)$ for all $X, Y \in \Gamma(\mcN)$, \emph{co-integrable} if $[\Gamma(\mcN^\perp), \Gamma(\mcN^\perp)] \subset \Gamma(\mcN^\perp)$, and \emph{totally co-geodetic} if $\nabla_X Y \in \Gamma( \mcN^\perp )$ for all $X, Y \in \Gamma(\mcN^\perp)$.
\end{defn}
The geometric properties of $\mcN_\xi$ and $\mcN_\xi^\perp$ can be encoded in terms of differential conditions on $[\xi^A]$.
\begin{prop}\label{prop-foliating-spinor} Let $\mcN_\xi$ be an almost null structure with associated projective pure spinor field $[\xi^A]$ on $(\mcM,g)$. Then
 \begin{itemize}
  \item $[\Gamma(\mcN_\xi) , \Gamma(\mcN_\xi) ] \subset \Gamma(\mcN_\xi^\perp)$ if and only if
\begin{align}\label{eq-almost-foliating}
      \xi \ind*{^{[A}} \left( \xi \ind*{^{aB]}} \nabla \ind{_a} \xi \ind*{^b^{[C}} \right)  \xi \ind*{_b^{D}} \xi^{E]} & = 0 \, ;
\end{align}
  \item $\mcN_\xi$ is integrable if and only if \eqref{eq-cond-21} holds, i.e.
\begin{align*}
   \left( \xi \ind*{^{a[A}} \nabla \ind{_a} \xi \ind*{^{B}} \right) \xi^{C]} & = 0 \, ;
\end{align*}
  \item $\mcN_\xi$ is totally geodetic if and only if
   \begin{align}\label{eq-N-geodetic}
   \left( \xi \ind*{^{[A}} \xi \ind*{^{aB]}} \nabla \ind{_a} \xi \ind*{^{[B}} \right) \xi^{C]} & = 0 \, .
  \end{align}
  \item $\mcN_\xi$ is co-integrable if and only if
\begin{align}\label{eq-Nperp-int}
 \left( \xi \ind*{^a^{[A}} \nabla \ind{_a} \xi \ind*{^b^{B]}} \right) \xi \ind*{_b^{[C}} \xi \ind*{^{D]}} & = 0 \, ;
\end{align}
  \item $\mcN_\xi$ is integrable and co-integrable if and only if
  \begin{align}\label{eq-foliating}
      \left( \xi \ind*{^{aA}} \nabla \ind{_a} \xi \ind*{^b^{[B}} \right) \xi \ind*{_b^{C}} \xi^{D]} & = 0 \, ;
&
      \xi \ind*{^{[A}} \left( \xi \ind*{^{aB]}} \nabla \ind{_a} \xi \ind*{^{[C}} \right) \xi \ind*{^{D]}} & = 0 \, .
  \end{align}
    \item $\mcN_\xi$ is totally co-geodetic if and only if \eqref{eq-cond-12} holds, i.e.
   \begin{align*}
   \left( \xi \ind*{^{aA}} \nabla \ind{_a} \xi \ind*{^{[B}} \right) \xi^{C]} & = 0 \, .
  \end{align*}
   \end{itemize}
\end{prop}

\begin{proof}
We compute each of the conditions in turn using \eqref{eq-nabla-spin-Gamma} and \eqref{eq-connection1form} in terms of the connection components. It then suffices to interpret the vanishing of these components in terms of the Lie bracket relations (since $\nabla_a$ is torsionfree). More explicitly, these are given by
\begin{itemize}
\item $\Gamma \ind{^{ABC}} =0$,
\item $\Gamma \ind{^{ABC}} = \Gamma \ind{^{[A:B]}} = 0$,
\item $\Gamma \ind{^{ABC}} = \Gamma \ind{^{A:B}} = 0$,
\item $\Gamma \ind{^{ABC}}=0$ and $\Gamma \ind{^{A:B}} = \Gamma \ind{^{AB}}$, (in particular, $\Gamma \ind{^{(A:B)}} = 0$),
\item $\Gamma \ind{^{ABC}} = \Gamma \ind{^{AB}} = \Gamma \ind{^{A:B}} = 0$,
\item $\Gamma \ind{^{ABC}} = \Gamma \ind{^{AB}} = \Gamma \ind{^{A:B}} = \Gamma \ind{^A} = 0$,
\end{itemize}
respectively.
\end{proof}

In contrast to the even-dimensional case, a (co-)integrable almost null structure is not necessarily totally (co-)geodetic. However, it is straightforward to show, as an consequence of Proposition \ref{prop-foliating-spinor}, or otherwise:
\begin{lem}\label{lem-nice-impl}
Let $[\xi^A]$ be a projective pure spinor. Then \eqref{eq-cond-12} $\Rightarrow$ \eqref{eq-foliating} $\Rightarrow$ \eqref{eq-N-geodetic} $\Rightarrow$ \eqref{eq-cond-21} . Equivalently, for any almost null structure $\mcN$,
\begin{itemize}
\item if $\mcN$ is totally co-geodetic, then it is integrable and co-integrable;
\item if $\mcN$ is integrable and co-integrable, then it is totally geodetic;
\item if $\mcN$ is totally geodetic, then it is integrable.
\end{itemize}
\end{lem}

\begin{defn}
 Let $[ \xi^A ]$ be a holomorphic projective pure spinor field on $(\mcM,g)$ with almost null structure $\mcN_\xi$. 
We say that $\xi^A$ is \emph{geodetic}, respectively \emph{co-geodetic}, if $\mcN_\xi$ is totally geodetic, respectively co-geodetic.
\end{defn}

\begin{rem}
Proposition \ref{prop-intorsion-class} can also be used to characterise the properties given in Proposition \ref{prop-foliating-spinor} in terms of the intrinsic torsion $\Gamma _{abc} \xi^{bB} \xi^{cC} \in \mfW$ of the $P$-structure. In particular, \eqref{eq-almost-foliating} holds if and only if $\Gamma _{abc} \xi^{bB} \xi^{cC} \in \mfW^{-2}$. Similarly, \eqref{eq-foliating} holds if and only if $\Gamma _{abc} \xi^{bB} \xi^{cC} \in \mfW^{-1}$.
\end{rem}

\paragraph{Conformal invariance}
With reference to appendix \ref{sec-conformal}, we prove
\begin{prop}\label{prop-conformal}
 Conditions \eqref{eq-cond-30}, \eqref{eq-cond-31}, \eqref{eq-cond-20}, \eqref{eq-cond-21}, \eqref{eq-cond-22}, \eqref{eq-cond-11} and \eqref{eq-Nperp-int}, (and  thus \eqref{eq-almost-foliating}, \eqref{eq-foliating}) are conformally invariant.
 
 Suppose that $[ \xi \ind*{^A} ]$ satisfies \eqref{eq-cond-11} and 
     \begin{align*}
   \left( \gamma \ind{^a_D^A} \nabla \ind{_a} \xi \ind*{^{bD}} \right) \xi \ind*{_b^{[B}} \xi^{C]} + 2 \left( \xi \ind*{^{aA}} \nabla \ind{_a} \xi \ind*{^{[B}} \right) \xi^{C]} & = (n-3) \, \xi^A \xi^{[B} \xi \ind*{^b^{C]}} \nabla_b f  \, ;
  \end{align*}
for some holomorphic function $f$. Then there exists a conformal rescaling for which $[ \xi \ind*{^A} ]$ satisfies
    \begin{align}\label{eq-N-parallel}
      \left( \nabla \ind{_a} \xi \ind*{^b^{[B}} \right) \xi \ind*{_b^{C}} \xi^{D]} & = 0 \, ,
  \end{align}
i.e.\ $\nabla X \in \Gamma ( \mcN_\xi^\perp )$ for all $X \in \Gamma ( \mcN_\xi )$, where $\mcN_\xi$ is the almost null structure associated to $[ \xi^A ]$.
\end{prop}


\paragraph{Curvature conditions}
The integrability conditions for these equations can easily be computed by differentiation a second time and commuting the covariant derivatives. 
\begin{prop}\label{prop-foliation-int-cond}
 Let $\xi^A$ be a geodetic spinor on $(\mcM,g)$, i.e.\ $\xi^A$ satisfies \eqref{eq-N-geodetic}. Then
\begin{align*}
{}^\mfC _\xi \Pi_{-3}^0 (C) & = 0 \, , & & \mbox{i.e.} & \xi \ind*{^{[A}} \xi \ind*{^a ^B} \xi \ind*{^b ^{C]}} \xi \ind*{^c ^D} \xi \ind*{^d ^E} C_{abcd} & = 0 \, .
\end{align*}
Suppose further that $\xi^A$ is co-geodetic, i.e.\ $\xi^A$ satisfies \eqref{eq-cond-12}. Then $\xi \ind*{^a ^A} \xi \ind*{^b ^B} \xi \ind*{^c ^C} \xi \ind*{^d ^D} R_{abcd} = 0$ and 
\begin{align*}
{}^\mfF _\xi \Pi_{-2}^0 (\Phi) & = 0 & \Longleftrightarrow & & {}^\mfC _\xi \Pi_{-2}^1 (C) & = 0 \, ,
\end{align*}
i.e.\ $\xi \ind*{^{[A}} \xi \ind*{^a^{B]}} \Phi \ind{_{ab}} \xi \ind*{^b ^{[C}} \xi \ind*{^{D]}} = 0$ if and only if $\xi \ind*{^a ^A} \xi \ind*{^b ^B} \xi \ind*{^c ^C} \xi \ind*{^d ^D} C_{abcd} = 0$.
\end{prop}

For a parallel projective pure spinor, we have the following -- see also \cite{Galaev2013} in more generality.
\begin{prop}\label{prop-recurrent-spinor}
Let $[ \xi^A ]$ be a parallel projective pure spinor on $(\mcM,g)$, i.e.\ $\xi^A$ satisfies \eqref{eq-recurrent}. Then
\begin{align}
\xi \ind*{^a ^A} \xi \ind*{^b ^B} R_{abcd} & = 0 \, , \label{eq-int_cond_Riem} \\
\xi \ind*{^a ^A} \xi \ind*{^b ^B} R_{ab} & = 0 \, , \label{eq-int_cond_Ric}\\
\xi \ind*{^a ^A} \xi \ind*{^b ^{\lb{B}}} \xi \ind*{^{\rb{C}}} \Phi_{ab} & = 0 \, , & & \mbox{i.e.} & {}^\mfF _\xi \Pi_{-1}^0 ( \Phi) & = 0 \, , \label{eq-int_cond_Phi} \\
\xi \ind*{^a ^A} \xi \ind*{^b ^B} \xi \ind*{^c ^{\lb{C}}} \xi \ind*{^{\rb{D}}} C_{abcd} & = 0 & & \mbox{i.e.} & {}^\mfC _\xi \Pi_{-1}^0 ( C ) = {}^\mfC _\xi \Pi_{-1}^1 ( C ) = {}^\mfC _\xi \Pi_{-1}^2 ( C ) & = 0 \, , \label{eq-int_cond_Weyl}
\end{align}
and in addition, when $m>2$,
\begin{align}
{}^\mfC _\xi \Pi_1^1 (C) & = 0 \, . \label{eq-m3d-cond}
\end{align}
Further,
\begin{align}
 R & = 0 & & \Longleftrightarrow & {}^\mfF _\xi \Pi_0^0 (\Phi) & = 0 & & \Longleftrightarrow & {}^\mfC _\xi \Pi_0^0 (C) & = 0  \, , \label{eq-1d-cond} \\
 {}^\mfF _\xi \Pi_0^1 (\Phi) & = 0 \, , & & \Longleftrightarrow & {}^\mfC _\xi \Pi_0^1 (C) & = 0 & & \Longleftrightarrow & {}^\mfC _\xi \Pi_0^2 (C) & = 0 \, , \label{eq-m2-1d-cond} \\
 {}^\mfF _\xi \Pi_1^0 (\Phi) & = 0 & & \Longleftrightarrow & {}^\mfC _\xi \Pi_1^0 (C) & = 0 \, . \label{eq-md-cond} 
\end{align}
\end{prop}

\begin{proof}
Equations \eqref{eq-int_cond_Riem} and \eqref{eq-int_cond_Ric} are is a direct consequence of \eqref{eq-recurrent}. Equation \eqref{eq-int_cond_Phi} follows from relating $R_{ab}$ and $\Phi_{ab}$ as $\Phi \ind{_{ab}} \xi \ind*{^a^A} \xi \ind*{^b^B} = \frac{1}{n} R \, \xi \ind*{^A} \xi \ind*{^B}$, from which we also conclude the first part of \eqref{eq-1d-cond}. Next, \eqref{eq-Riem_decomposition} yields \eqref{eq-int_cond_Weyl}. To conclude the remaining conditions, we use the definitions of ${}^\mfF _\xi \Pi_i^j$ and ${}^\mfC _\xi \Pi_i^j$ together with the computations
\begin{align*}
 \xi \ind*{^c^C} C \ind{_{c[ab]d}} \xi \ind*{^d^D} & = \frac{2}{n-2} \xi \ind*{_{[a}^{[C}} \Phi \ind{_{b]d}} \xi \ind*{^{dD]}} + \frac{1}{n(n-1)}  R \, \xi \ind*{_{[a}^C} \xi \ind*{_{b]}^D} \, , \\
 C \ind{_{bcad}} \xi \ind*{^{bc}^B} \xi \ind*{^d^D} & = - \frac{2(n-4)}{n-2} \xi \ind*{^B} \xi \ind*{^d^D} \Phi \ind{_{ad}} - \frac{2}{n-2} \xi \ind*{^D} \xi \ind*{^{dB}} \Phi \ind{_{ad}}  - \frac{2(n-2)}{n(n-1)} R \, \xi \ind*{_a^D} \xi \ind*{^B} + \frac{2}{n(n-1)(n-2)} R \, \xi \ind*{_a^B} \xi \ind*{^D} \, , \\
    C \ind{_{bcad}} \xi \ind*{^{bc}^{(B}} \xi \ind*{^d^{D)}} & = - \frac{2(n-3)}{n-2} \left( \xi \ind*{^{(B}} \xi \ind*{^d^{D)}} \Phi \ind{_{ad}} + \frac{1}{n} R \, \xi \ind*{^{(B}} \xi \ind*{_a^{D)}} \right) \, , \\
 C \ind{_{bcad}} \xi \ind*{^{ad}^A} \xi \ind*{^{bc}^B} & = - 2 \frac{n-3}{n-2} R \, \xi \ind*{^A} \xi \ind*{^B}  \, , \\
 C \ind{_{abcd}} \xi \ind*{^b^B} \xi \ind*{^c^C} \xi \ind*{^d^D} & = - \frac{2}{n-2} \xi \ind*{^B} \xi \ind*{^{[C}} \Phi \ind{_{ad}} \xi \ind*{^d^{D]}} + \frac{2}{n(n-1)(n-2)} R \, \xi \ind*{^B} \xi \ind*{^{[C}} \xi \ind*{_a^{D]}} \, .
\end{align*}
In particular, we note that the dimensions of the irreducible $\prb$-invariant parts of the Weyl tensor must match those of the tracefree Ricci tensor. From the invariant diagram of Proposition \ref{prop-Weyl-classification}, one sees that condition \eqref{eq-m3d-cond} imposes algebraic conditions on elements of the isotopic modules $\mfC_0^1$ and $\mfC_0^2$, which, by dimension counting must match $\mfF_0^1$. More explicitly, on referring to the maps ${}^\mfC _\xi \Pi_i^j$, we have
\begin{align*}
{}^\mfC _\xi \Pi_0^1 ( C ) \ind{^{BC}_a}  \xi^A & = - 2 \frac{n-5}{n-3} \xi \ind*{^{(B}} {}^\mfF _\xi \Pi_0^1 ( \Phi ) \ind{_a^{C)A}} \, , &
 {}^\mfC _\xi \Pi_0^2 ( C ) \ind{^{A(BC)}_d} & = - \frac{1}{n-2} \xi \ind*{^{(B}} {}^\mfF _\xi \Pi_0^1 ( \Phi ) \ind{_d^{C)A}} \pmod{ \xi^B \xi^C \alpha_a^D } \, , 
\end{align*}
where we have rewritten ${}^\mfF _\xi \Pi_0^1 ( \Phi ) := \xi \ind*{^A} \xi \ind*{^b ^{B}} \Phi \ind{_{b a}} - \frac{1}{n-1}  \xi \ind*{^b ^A} \xi \ind*{^c ^C} \Phi \ind{_{bc}} \gamma \ind{_b _C ^B} \pmod{ \xi^A \xi^B \alpha_a}$. Condition \eqref{eq-m2-1d-cond} now follows.
\end{proof}

\subsection{Spinorial differential equations}
\subsubsection{Scale-dependent geodetic and co-geodetic spinors}
A scale-dependent variation of \eqref{eq-N-geodetic} is given by $\xi \ind*{^{[A}} \xi \ind*{^{aB]}} \nabla \ind{_a} \xi \ind*{^{B}} = 0$, with integrability condition ${}^\mfC _\xi \Pi_{-2}^0 (C) = \xi \ind*{^{[A}} \xi \ind*{^a ^B} \xi \ind*{^b ^{C]}} \xi \ind*{^{cd} ^D} C_{abcd} = 0$. This is conformally invariant provided $\xi^A$ has conformal weight $-1$. 

Similarly, a scale-dependent variation of \eqref{eq-cond-12} is given by $\xi \ind*{^{aA}} \nabla \ind{_a} \xi \ind*{^{B}} = 0$, with integrability conditions given by $\xi \ind*{^a ^A} \xi \ind*{^b ^B} \xi \ind*{^{cd} ^C} R_{abcd} = 0$. Further, ${}^\mfF _\xi \Pi_{-1}^0 ( \Phi ) = \xi \ind{^{\lb{A}}} \xi \ind{^a ^{\rb{B}}} \Phi \ind{_{a b}} \xi \ind{^b ^C} =0$ if and only if ${}^\mfC _\xi \Pi_{-1}^0 (C) = \xi \ind{^a ^A} \xi \ind{^b ^B} C \ind{_{abcd}} \xi \ind{^{cd} ^C}=0$

\subsubsection{Parallel pure spinors}
The next proposition follows from Proposition \ref{prop-recurrent-spinor}.
\begin{prop}\label{prop-int-cond-par}
 Let $\xi^A$ be a parallel pure spinor field on $(\mcM,g)$, i.e.\ $\nabla_a \xi^A = 0$. Then
$R \ind{_{abcd}} \xi \ind*{^{cd}^D} = 0$, ${}^\mfF _\xi \Pi_1^0 (\Phi) = \Phi \ind{_{ab}} \xi \ind*{^b^B} = 0$, $R  = 0$, and ${}^\mfC _\xi \Pi_2^0 (C) = C \ind{_{abcd}} \xi \ind*{^{cd}^D} = 0$.
\end{prop}
\subsubsection{Null zero-rest-mass fields}
The smaller irreducible part of the covariant derivative of a spinor field $\xi^A$ leads to the \emph{(Weyl-)Dirac equation} 
\begin{align}\label{eq-Weyl-Dirac}
 \gamma \ind{^a _A^B} \nabla_a \xi^A & = 0 \, ,
\end{align}
In contrast to even dimensions, this equation admits not one, but \emph{two} generalisations to irreducible spinor fields of higher valence.

\begin{defn}
Let $\phi^{A_1 A_2 \ldots A_k} = \phi^{(A_1A_2 \ldots A_k)}$ be a holomorphic spinor field on $(\mcM,g)$ irreducible in the sense that $\gamma \ind{^a_{A_1}^C} \gamma \ind{_a_{A_2}^D} \phi^{A_1 A_2 A_3 \ldots A_k} = - \phi^{C D A_3\ldots A_k}$. We say that $\phi \ind{^{A_1 \ldots A_k}}$ is a \emph{zero-rest-mass (zrm) field} if it satisfies
\begin{align}
 \gamma \ind{^a _{B}^{\lp{A_1}}} \nabla_a \phi^{A_2 \ldots \rp{A_k}B} & = 0 \, , \label{eq-zrm}
 \end{align}
and a \emph{co-zero-rest-mass (co-zrm) field} if it satisfies
\begin{align}
 \gamma \ind{^a _{B}^{\lb{A_1}}} \nabla_a \phi^{\rb{A_2} A_3 \ldots A_k B} & = 0 \, . \label{eq-co-zrm}
\end{align}
\end{defn}

\begin{rem}
When $k=2$, an irreducible spinor field as above is simply an $m$-form, or by Hodge duality, an $(m+1)$-form. Equation \eqref{eq-zrm}, respectively, \eqref{eq-co-zrm} are then equivalent to this $m$-form to be closed, respectively, co-closed, hence the use of terminology. This follows from the fact that matrices $\gamma \ind{_{a_1 \ldots a_{m+1}}^{AB}}$ and $\gamma \ind{_{a_1 \ldots a_{m-1}}^{AB}}$ are symmetric and skewsymmetric respectively.
\end{rem}

Equations \eqref{eq-zrm} and \eqref{eq-co-zrm} are conformally invariant provided that $\phi^{A_1 \ldots A_k}$ is of conformal weight $-m-\frac{k}{2}$ and $-m-k$ respectively. In particular, a solution of both \eqref{eq-zrm} and \eqref{eq-co-zrm}, i.e.\
\begin{align}
 \gamma \ind{^a _{B}^{A_1}} \nabla_a \phi^{A_2 \ldots A_k B} & = 0 \, , \label{eq-nc-zrm} 
\end{align}
is not conformally invariant. In the case $k=2$, such a solution corresponds to a closed and co-closed $m$-form.

The integrability condition for the existence of solutions to equations \eqref{eq-zrm} and \eqref{eq-co-zrm} of valence greater than two is given by the following lemma.
\begin{prop}\label{prop-int-zrm}
For $k>2$, let $\phi^{A_1 A_2 \ldots A_k}$ be a solution of \eqref{eq-zrm} or \eqref{eq-co-zrm} on $(\mcM,g)$. Then
 \begin{align}\label{eq-int-cond-zrm}
  \gamma \ind{^a_{C_1}^{A}} \gamma \ind{^b_{C_2}^{B}} C \ind{_{abcd}} \gamma \ind{^{cd}_D^{(C_3}} \phi \ind{^{C_4 \ldots C_k) C_1 C_2 D}} = 0 \, .
 \end{align}
 If $\phi^{A_1 A_2 \ldots A_k}$ is a solution of \eqref{eq-nc-zrm}, then we have in addition
 \begin{align}\label{eq-int-cond-nc-zrm}
\gamma \ind{^b_{C_2}^{[A|}} \Phi \ind{_{bd}} \gamma \ind{^d_D^{(C_3}} \phi \ind{^{C_4 \ldots C_k)C_2 D |B]}} = 0 \, .
 \end{align}
\end{prop}

\begin{proof}
Equations \eqref{eq-zrm}, \eqref{eq-co-zrm} and \eqref{eq-nc-zrm} can be rewritten as $\gamma \ind{^a _{B}^{A_1}} \nabla_a \phi^{A_2 \ldots A_k B}  = \psi \ind{^{A_1 A_2 \ldots A_k}}$, where $\psi \ind{^{(A_1 A_2 \ldots A_k)}} =0$, $\psi \ind{^{[A_1 A_2] A_3\ldots A_k}} =0$, and $\psi \ind{^{A_1 A_2 \ldots A_k}} = 0$ respectively.
Taking a second covariant derivative and commuting lead to
 \begin{multline*}
(k-2) \,  \gamma \ind{^a_{C_1}^{A}} \gamma \ind{^b_{C_2}^{B}} C \ind{_{abcd}} \gamma \ind{^{cd}_D^{(C_3}} \phi \ind{^{C_4 \ldots C_k) C_1 C_2 D}} - 4 (k-2) \, \gamma \ind{^b_{C_2}^{[A|}} \Rho \ind{_{bd}} \gamma \ind{^d_D^{(C_3}} \phi \ind{^{C_4 \ldots C_k)C_2 D |B]}} \\
= 2 \, \gamma \ind{^a_D^{[A}} \nabla \ind{_a} \psi \ind{^{B] C_3 C_4 \ldots C_k D}} \, .
 \end{multline*}
By the conformal invariance of \eqref{eq-zrm} and \eqref{eq-co-zrm}, the first term on the LHS must vanish identically, while the second term on the LHS cancels the RHS, hence \eqref{eq-int-cond-zrm}. When \eqref{eq-nc-zrm} holds, conformal invariance is broken, and one has the additional constraint \eqref{eq-int-cond-nc-zrm}.
\end{proof}

A spinor field $\phi^{A_1 A_2 \ldots A_k}$ is referred to as \emph{null} if it takes the form $\phi^{A_1 A_2 \ldots A_k} = \ee^\psi \xi^{A_1} \xi^{A_2} \ldots \xi^{A_k}$ for some holomorphic pure spinor field $\xi^A$, and holomorphic function $\psi$. Specialising Proposition \ref{prop-int-zrm} yields
\begin{cor}
For $k>2$, suppose that $\phi^{A_1 A_2 \ldots A_k} := \ee^\psi \xi^{A_1} \xi^{A_2} \ldots \xi^{A_k}$ is a solution of \eqref{eq-zrm} or \eqref{eq-co-zrm} on $(\mcM,g)$. Then
 \begin{align}\label{eq-simple-zrm-int-cond}
 {}^\mfC _\xi \Pi_{-1}^0 (C) & = 0 \, , & & \mbox{i.e.} & \xi \ind*{^a^A} \xi \ind*{^a^B } C \ind{_{abcd}} \xi \ind*{^{cd}^{C}} & = 0 \, , 
 \end{align}
Further, if $\phi^{A_1 A_2 \ldots A_k}$ is a solution of \eqref{eq-nc-zrm}, then we have in addition
 \begin{align}\label{eq-simple-zrm-int-cond-xtra}
 {}^\mfF _\xi \Pi_{-1}^0 (\Phi) & = 0 \, , & & \qquad \mbox{i.e.} & \xi \ind*{^a^{A}} \xi \ind*{^b^{[B}} \Phi \ind{_{ab}} \xi \ind*{^{C]}} = 0 \, .
 \end{align}
\end{cor}

The relation between null solutions of the zrm-field equation and the existence of foliating spinors is known as the Robinson theorem \cite{Robinson1961} in four dimensions, and was later generalised to even dimensions in \cite{Hughston1988}. Here, we give odd-dimensional versions of the theorem.

\begin{thm}[Robinson theorem for zrm fields]\label{thm-Robinson-zrm}
 Let $\xi^A$ be a holomorphic pure spinor field on $(\mcM,g)$  with almost null structure $\mcN_\xi$. Let $\psi$ be a holomorphic function and suppose that $\phi^{A_1 A_2 \ldots A_k} := \ee^\psi \xi^{A_1} \xi^{A_2} \ldots \xi^{A_k}$ satisfies the zrm field equation \eqref{eq-zrm}. Then locally, $\xi^A$ satisfies 
\begin{align}
\xi \ind*{^{[A}} \left( \xi \ind*{^a^{B]}} \nabla \ind{_a} \xi \ind*{^{[C}} \right) \xi \ind*{^{D]}} + \xi \ind*{^{[C}} \left( \xi \ind*{^a^{D]}} \nabla \ind{_a} \xi \ind*{^{[A}} \right) \xi \ind*{^{B]}} & = 0 \, , \label{eq-Nperp-int-zrm1} \\
 \left( \xi \ind*{^a^{[A}} \nabla \ind{_a} \xi \ind*{^b^{B]}} \right) \xi \ind*{_b^{[C}} \xi \ind*{^{D]}} - \frac{k-2}{2k} \left(  \xi \ind*{^{[A}} \left( \xi \ind*{^a^{B]}} \nabla \ind{_a} \xi \ind*{^b^C} \right) \xi \ind*{_b^D} + \xi \ind*{^{[A}} \left( \xi \ind*{^a^{B]}} \nabla \ind{_a} \xi \ind*{^{C}} \right) \xi \ind*{^{D}} \right) & = 0 \, . \label{eq-Nperp-int-zrm2}
\end{align}
In particular, $[\Gamma(\mcN_\xi) , \Gamma( \mcN_\xi )] \subset \Gamma ( \mcN^\perp_\xi )$. When $k=2$, $\xi^A$ locally satisfies \eqref{eq-Nperp-int}, i.e.\ $\mcN_\xi$ is co-integrable.

Suppose that $\xi^{A}$ satisfies \eqref{eq-Nperp-int}, i.e.\ $\mcN_\xi$ is co-integrable. Then locally there exists a holomorphic function $\psi$ such that the spinor field $\phi^{A B} := \ee^\psi \xi \ind*{^A} \xi \ind*{^B}$ satisfies \eqref{eq-zrm}. There is the freedom of adding to $\psi$ a holomorphic function constant along the leaves of $\mcN_\xi^\perp$.
\end{thm}

\begin{proof}
 For any $\phi^{A_1 A_2 \ldots A_k} := \ee^\psi \xi^{A_1} \xi^{A_2} \ldots \xi^{A_k}$, we have, in regions where $\phi^{A_1 A_2 \ldots A_k}$ does not vanish,
\begin{align}\label{eq-master}
\gamma \ind{^a_B^{A_1}}  \nabla \ind{_a} \phi \ind{^{A_2 \ldots A_k B}} & = \ee^\psi \left( \xi \ind*{^{A_2}} \ldots \xi \ind*{^{A_k}} \xi \ind*{^a^{A_1}} \nabla \ind{_a} \psi + (k-1) \left( \xi \ind*{^a^{A_1}} \nabla \ind{_a} \xi \ind*{^{(A_2}} \right) \xi \ind*{^{A_3}} \ldots \xi \ind*{^{A_k)}} +  \left( \nabla \ind{_a} \xi \ind*{^a^{A_1}} \right) \xi \ind*{^{A_2}} \ldots \xi \ind*{^{A_k}} \right) \, .	
\end{align}
If $\phi \ind{^{A_1 \ldots A_k}}$ satisfies \eqref{eq-zrm}, then we have
\begin{align}\label{eq-master-zrm}
 0 & = \xi \ind*{^{(A_2}} \xi \ind*{^{A_3}} \ldots \xi \ind*{^{A_k}} \xi \ind*{^a^{A_1)}} \nabla \ind{_a} \psi + (k-1) \left( \xi \ind*{^a^{(A_1}} \nabla \ind{_a} \xi \ind*{^{A_2}} \right) \xi \ind*{^{A_3}} \ldots \xi \ind*{^{A_k)}} +  \left( \nabla \ind{_a} \xi \ind*{^a^{(A_1}} \right) \xi \ind*{^{A_2}} \xi \ind*{^{A_3}} \ldots \xi \ind*{^{A_k)}} \, .
\end{align}
Tensoring with $\xi^B \xi^C$ and skewing over $A_1 B$ and $A_2 C$ lead to \eqref{eq-Nperp-int-zrm1}. Working in the splitting \ref{eq-null-grading} with a choice of spinor $\eta_A$ dual to $\xi^A$, and using \eqref{eq-connection1form}, this implies $\Gamma ^{(B:C)} = \Gamma ^{ABC} = 0$, i.e.\ $\xi^A$ satisfies \eqref{eq-almost-foliating}. Expanding \eqref{eq-master-zrm} now yields
\begin{align*}
0 & = \left( - \frac{1}{4} \left( k \, \Gamma \ind{^{BC}} - 2 \, \Gamma \ind{^{B:C}} \right) \eta \ind{_a_B} \gamma \ind{^a_C^{(A_1}} + \psi \ind{^{(A_1}} \right) \xi \ind*{^{A_2}} \xi \ind*{^{A_3}} \ldots \xi \ind*{^{A_k)}} \, ,
\end{align*}
for some $\psi \ind{^{A}} \in \mfS^{\frac{m-2}{2}}$. Since the first term on the RHS lies in $\mfS_{\frac{m-4}{2}}$, we must have $k \, \Gamma \ind{^{BC}} = 2 \, \Gamma \ind{^{B:C}}$, i.e.\ \eqref{eq-Nperp-int-zrm2} holds. When $k=2$, \eqref{eq-Nperp-int-zrm2} reduces to \eqref{eq-Nperp-int}.

For the converse when $k=2$, we follow the geometrical proof given in \cites{Eastwood1995,Mason1995}. Suppose that $\xi^{A}$ satisfies \eqref{eq-Nperp-int}, i.e.\ $\mcN_\xi$ is co-integrable. Then, locally, $\mcM$ is fibered over the leaf space $\mcL$ of $\mcN_\xi^\perp$. Choose a holomorphic section $\phi$ of the tautological line bundle $\wedge^m \Tgt^* \mcL$ of $\mcL$. Then, $\phi$ is clearly closed. Its pull-back to $\mcM$ must be orthogonal to each leaf of the foliation, i.e.\ it must be of the form $\phi^{A B} := \phi \ind{_{a_1 \ldots a_m}} \gamma \ind{^{a_1 \ldots a_m}^{AB}} = \ee^\psi \xi \ind*{^A} \xi \ind*{^B}$ for some holomorphic function $\psi$. Further, since the exterior derivative commutes with the pull-back, $\phi$ is also closed, i.e.\ $\phi^{AB}$ satisfy \eqref{eq-zrm}. 

Finally, in both cases, adding any holomorphic function constant along the leaves of $\mcN_\xi^\perp$ to $\psi$, i.e.\ annihilated by $\xi^{aA} \nabla_a$, leaves the relevant field equations unchanged.
\end{proof}

\begin{thm}[Robinson theorem for co-zrm fields]\label{thm-Robinson-co-zrm}
 Let $\xi^A$ be a holomorphic pure spinor field on $(\mcM,g)$ with almost null structure $\mcN_\xi$. Let $\psi$ be a holomorphic function and suppose that $\phi^{A_1 A_2 \ldots A_k} := \ee^\psi \xi^{A_1} \xi^{A_2} \ldots \xi^{A_k}$ satisfies the co-zrm field equation \eqref{eq-co-zrm}. Then locally $\xi^A$ satisfies \eqref{eq-cond-21}, i.e.\ $\mcN_\xi$ is integrable. Further, when $k>2$, $\xi^A$ satisfies \eqref{eq-N-geodetic}, i.e. $\mcN_\xi$ is totally geodetic.

Suppose that $\xi^{A}$ satisfies \eqref{eq-cond-21}, i.e. $\mcN_\xi$ is integrable. Then locally there exists a holomorphic function $\psi$ such that the pure spinor field $\phi^{A B} = \ee^\psi \xi \ind*{^A} \xi \ind*{^B}$ satisfies \eqref{eq-co-zrm}. Further, if $\xi^A$ satisfies \eqref{eq-N-geodetic}, i.e. $\mcN_\xi$ is totally geodetic, and the curvature condition \eqref{eq-simple-zrm-int-cond}, then locally, for every $k>2$, there exists a holomorphic function $\psi$ such that the spinor field $\phi^{A_1 A_2 \ldots A_k}=\ee^\psi \xi^{A_1} \xi^{A_2} \ldots \xi^{A_k}$ satisfies \eqref{eq-co-zrm}. In both cases, there is the freedom of adding to $\psi$ a holomorphic function constant along the leaves of $\mcN_\xi$.
\end{thm}

\begin{proof}
For $k \geq 2$, if $\phi \ind{^{A_1 \ldots A_k}}$ satisfies \eqref{eq-co-zrm}, then equation \eqref{eq-master} becomes
\begin{multline}\label{eq-master-co-zrm}
 0 = \xi \ind*{^{A_3}} \ldots \xi \ind*{^{A_k}} \xi \ind*{^{[A_2}}\xi \ind*{^a^{A_1]}} \nabla \ind{_a} \psi + \left( \xi \ind*{^a^{[A_1}} \nabla \ind{_a} \xi \ind*{^{A_2]}} \right) \xi \ind*{^{A_3}} \ldots \xi \ind*{^{A_k}} \\
+ (k-2) \left( \xi \ind*{^{[A_2}} \xi \ind*{^a^{A_1]}} \nabla \ind{_a} \xi \ind*{^{(A_3}} \right) \xi \ind*{^{A_4}} \ldots \xi \ind*{^{A_k)}} +  \left( \nabla \ind{_a} \xi \ind*{^a^{[A_1}} \right) \xi \ind*{^{A_2]}} \xi \ind*{^{A_3}} \ldots \xi \ind*{^{A_k}} \, .
\end{multline}
Then, tensoring with $\xi \ind*{^B}$ and skewing over $A_1 A_2 B$ yield \eqref{eq-cond-21}, i.e. $\mcN_\xi$ is integrable. When $k>2$, one can also tensor with $\xi \ind*{^B}$ and skew over $A_3 B$, and conclude  \eqref{eq-N-geodetic}, i.e. $\mcN_\xi$ is totally geodetic.

For the converse, the case $k=2$ is similar to the proof of Theorem \ref{thm-Robinson-zrm} except that one obtains a closed $(m+1)$-form, which is Hodge dual to a co-closed $m$-form. So we focus on the case $k>2$ and assume that condition \eqref{eq-N-geodetic} holds. This is equivalent to
\begin{align}\label{eq-Robinson-co-zrm-proof}
\xi \ind*{^a^A} \nabla \ind{_a} \xi \ind*{^B} & = \xi \ind*{^A} A \ind{^B} + \xi \ind*{^A} B \ind{^B} + C \, \xi \ind*{^A} \xi \ind*{^B} + D \ind{^A} \xi \ind*{^B} \, , &
\nabla \ind{_a} \xi \ind*{^a^A} & = E \, \xi \ind*{^A} + F \ind{^A} - B \ind{^A} + A \ind{^B} + C \, \xi \ind*{^A} + D \ind{^A} \, ,
\end{align}
for some functions $C$, $E$, spinors $B \ind{^A}$, $D \ind{^A}$, $F \ind{^A}$ in $\mfS^{\frac{m-2}{2}} = \im \xi_a^A$, and $A \ind{^A}$ in $\mfS^{\frac{m-4}{2}} \im \xi_{ab}^A$. We want to show that locally there exists a holomorphic function $\psi$ such that \eqref{eq-master-co-zrm} holds, i.e.\
\begin{align}\label{eq-Robinson-co-zrm-psi}
\xi \ind*{^{[A}} \xi \ind*{^a^{B]}} \nabla \ind{_a} \psi & = \xi \ind*{^{a[A}} \nabla \ind{_a} \xi \ind*{^{B]}} + \left( \nabla \ind{_a} \xi \ind*{^a^{[A}} \right) \xi \ind*{^{B]}} - (k-2) \, \xi \ind*{^{[A}} D \ind{^{B]}} 
= \xi \ind*{^{[A}} \left( 2 \, B \ind{^{B]}} - k D \ind{^{B]}} - F \ind{^{B]}} \right) =: \xi \ind*{^{[A}} \psi \ind{^{B]}} \, .
\end{align}
Differentiating the above equation with respect to $\xi^{[A} \xi \ind*{^a^{B]}} \nabla_a$, i.e.\ along $\mcN_\xi$, yields the integrability condition
\begin{align}\label{eq-Robinson-co-zrm-intcond}
\xi \ind*{^{[A}} D \ind{^{B}} \psi \ind{^{C]}} & = \xi \ind*{^{[A}} \xi \ind*{^a^B} \nabla \ind{_a} \psi \ind{^{C]}} \, .
\end{align}
We expand the RHS of \eqref{eq-Robinson-co-zrm-intcond} using the expression \eqref{eq-Robinson-co-zrm-psi} for $\psi^A$:
\begin{align*}
\xi \ind*{^{[A}} \xi \ind*{^a^B} \nabla \ind{_a} \psi \ind{^{C]}} & =  - \xi \ind*{^a^{[A}} \nabla \ind{_a} \left( \xi \ind*{^{B}} \psi \ind{^{C]}} \right) + \left( \xi \ind*{^a^{[A}} \nabla \ind{_a} \xi \ind*{^B} \right) \psi \ind{^{C]}}
\\
& = - \xi ^{a[A} \nabla_a \left( \xi^{bB} \nabla_b \xi^{C]} \right) - \xi^{a[A} \nabla_a \left( \left( \nabla_b \xi^{bB} \right) \xi^{C]} \right) + ( k - 2) \xi \ind*{^a^{[A}} \nabla \ind{_a} \left( \xi \ind*{^B} D \ind{^{C]}} \right) + \left( \xi \ind*{^a^{[A}} \nabla \ind{_a} \xi \ind*{^B} \right) \psi \ind{^{C]}} \, .
\end{align*}
We compute each term in turn using the assumption \eqref{eq-simple-zrm-int-cond}. For the third term, we find
\begin{align*}
\xi \ind*{^a^{[A}} \nabla \ind{_a} \left( \xi \ind*{^B} D \ind{^{C]}} \right) \xi \ind*{^D} & = \xi \ind*{^a^{[A}} \nabla \ind{_a} \left( \xi \ind*{^B} D \ind{^{C]}} \xi \ind*{^D} \right) - \left( \xi \ind*{^a^{[A|}} \nabla \ind{_a} \xi \ind*{^D} \right) \xi \ind*{^{|B}} D \ind{^{C]}}  = \xi \ind*{^a^{[A}} \nabla \ind{_a} \left( \xi \ind*{^B} \xi \ind*{^b^{C]}} \nabla \ind{_b} \xi \ind*{^D} \right) \\
& = \left( \xi \ind*{^a^{[A}} \nabla \ind{_a} \xi \ind*{^B} \right) \left( \xi \ind*{^b^{C]}} \nabla \ind{_b} \xi \ind*{^D} \right) - \xi \ind*{^{[A}} \left( \xi \ind*{^a^{B}}  \nabla \ind{_a} \xi \ind*{^b^{C]}} \right) \nabla \ind{_b} \xi \ind*{^D} - \cancel{ \frac{1}{8} \xi \ind*{^{[A}} \xi \ind*{^a^{B}} \xi \ind*{^b^{C]}} C \ind{_{abcd}} \xi \ind*{^{cd}^D} } \, .
\end{align*}
For the second term, we have
\begin{align*}
\xi^{a[A} \nabla_a \left( \left( \nabla_b \xi^{bB} \right) \xi^{C]} \right) & = \left( \xi^{a[A} \nabla_a  \nabla_b \xi^{bB}\right) \xi^{C]} - \left( \xi^{a[A} \nabla_a \xi^B \right) \nabla_b \xi^{bC]} \\
& = \left( \xi^{a[A} \nabla_b  \nabla_a \xi^{bB}\right) \xi^{C]} + \frac{1}{4} \xi^{a[A|} R_{abcd} \xi^{cdD} \gamma \ind{^b_D^{|B}} \xi^{C]}  - \left( \xi^{a[A} \nabla_a \xi^B \right) \nabla_b \xi^{bC]} \\
& = \nabla_b  \left( \xi^{a[A} \nabla_a \xi^{bB} \right) \xi^{C]} - \cancel{ \left( \nabla_b  \xi^{a[A} \right) \left( \nabla_a \xi^{bB}\right) \xi^{C]} } - \cancel{ \frac{1}{2} \xi^{a[A} R_{ab} \xi^B \xi^{C]} }  - \left( \xi^{a[A} \nabla_a \xi^B \right) \nabla_b \xi^{bC]} \\
& = \nabla_a \left( \left( \xi^{b[A} \nabla_b \xi^{aB} \right) \xi^{C]} \right) - \left( \xi^{a[A} \nabla_a  \xi^{bB} \right) \nabla_b \xi^{C]} - \left( \xi^{a[A} \nabla_a \xi^B \right) \nabla_b \xi^{bC]} 
\end{align*}
while the first term simply becomes
\begin{align*}
\xi ^{a[A} \nabla_a \left( \xi^{bB} \nabla_b \xi^{C]} \right) & = \left( \xi ^{a[A} \nabla_a \xi^{bB} \right)  \nabla_b \xi^{C]}  - \cancel{ \frac{1}{8} \xi^{a[A} \xi^{bB} C_{abcd} \xi^{cdC]} } \, ,
\end{align*}
The last step is to use  \eqref{eq-Robinson-co-zrm-proof} and \eqref{eq-Robinson-co-zrm-psi} to express the covariant derivative of $\xi^A$ in all these expressions in terms of $\xi^A$, $A^A$, $B^A$, $C$, $D^A$, $E$ and $F^A$. Thus, we get $\xi \ind*{^a^{[A}} \nabla \ind{_a} \left( \xi \ind*{^B} D \ind{^{C]}} \right)  = - \xi^{[A} D^B \left( A^{C]} + B^{C]} \right)$ and similarly for the other terms. Applying \eqref{eq-Robinson-co-zrm-psi} to the LHS of \eqref{eq-Robinson-co-zrm-intcond} reveals that \eqref{eq-Robinson-co-zrm-intcond} is indeed satisfied.

Finally, in both cases, adding any holomorphic function constant along the leaves of $\mcN_\xi$ to $\psi$, i.e.\ annihilated by $\xi^{[A} \xi^{aB]} \nabla_a$, leaves the relevant field equations unchanged.
\end{proof}

We omit the proof of the following theorem, which follows roughly the one given in \cite{Hughston1988}.
\begin{thm}[Non-conformally invariant Robinson theorem]\label{thm-Robinson-nc-zrm}
 Let $\xi^A$ be a holomorphic pure spinor field on $(\mcM,g)$ with almost null structure $\mcN_\xi$. Let $\psi$ be a holomorphic function and suppose that $\phi^{A_1 A_2 \ldots A_k} := \ee^\psi \xi^{A_1} \xi^{A_2} \ldots \xi^{A_k}$ is both a zrm field and a co-zrm field, i.e.\ $\phi^{A_1 A_2 \ldots A_k}$ satisfies \eqref{eq-nc-zrm}. Then locally $\xi^A$ satisfies \eqref{eq-cond-12}, i.e.\ $\mcN_\xi$ is totally co-geodetic.

Suppose that $\xi^{A}$ satisfies \eqref{eq-cond-12}, i.e.\ $\mcN_\xi$ is totally co-geodetic. Then locally there exists a holomorphic function $\psi$ such that $\phi^{AB} := \ee^\psi \xi^A \xi^B$ satisfies \eqref{eq-nc-zrm}. Suppose further that $\xi^A$ satisfies the curvature conditions \eqref{eq-simple-zrm-int-cond} and \eqref{eq-simple-zrm-int-cond-xtra}. Then, for every $k>2$, there exists a holomorphic function $\psi$ such that $\phi^{A_1 A_2 \ldots A_k} := \ee^\psi \xi^{A_1} \xi^{A_2} \ldots \xi^{A_k}$ satisfies \eqref{eq-nc-zrm}. In both cases, there is the freedom of adding to $\psi$ a holomorphic function constant along the leaves of $\mcN^\perp_\xi$.
\end{thm}

\begin{rem}
In flat even-dimensional space, the Robinson theorem is often used in conjunction with the \emph{Kerr theorem} \cites{Kerr2009,Penrose1967,Hughston1988}, by means of which one (locally) generates null structures in terms of geometric data in a `twistor space'. It is interesting to note that one can also distinguish three odd-dimensional counterparts of the Kerr theorem as presented in \cite{Taghavi-Chabert2017} depending the various `degrees' of integrability of an almost null structure.
\end{rem}

\subsubsection{Conformal Killing spinor}
Complementary to \eqref{eq-Weyl-Dirac}, one defines the \emph{twistor equation}
\begin{align}\label{eq-twistor-spinor}
 \nabla_a \xi^{A} + \frac{1}{\sqrt{2}} \gamma \ind{_a _B^{A}} \zeta^B & = 0 \, ,
\end{align}
for any holomorphic spinor field $\xi^A$. Here, \eqref{eq-twistor-spinor} determines $\zeta^B = \frac{\sqrt{2}}{n} \gamma \ind{^a _{A}^B} \nabla_a \xi^{A}$. A solution $\xi^A$ will be referred to as a \emph{conformal Killing spinor} or \emph{twistor-spinor}. The spinor field $\zeta^A$ can be shown to satisfy
\begin{align}\label{eq-twistor-spinor2}
\nabla \ind{_a} \zeta \ind*{^B} + \frac{1}{\sqrt{2}} \Rho \ind{_{ab}} \gamma \ind{^b_A^B} \xi \ind*{^A} & = 0 \, ,
\end{align}
where $\Rho_{ab} := \frac{1}{2-n} \Phi_{ab} - R \frac{1}{2n(n-1)} g_{ab}$ is the \emph{Rho} or \emph{Schouten} tensor (see Appendix \ref{sec-conformal}). Equations \eqref{eq-twistor-spinor} and \eqref{eq-twistor-spinor2} are conformally invariant provided that $\xi^A$ and $\zeta^A$ transform as
\begin{align}\label{eq-twistor-spinor-transform}
\xi \ind*{^A} & \mapsto \hat{\xi} \ind{^A} = \xi \ind*{^A} \, , & \zeta \ind*{^A} & \mapsto \hat{\zeta} \ind{^A} = \Omega^{-1} \left( \zeta \ind*{^A} + \frac{1}{\sqrt{2}} \Upsilon \ind{_a} \xi \ind*{^a^A} \right) \, .
\end{align}
The equivalence class of pairs of spinors $(\xi^A,\zeta^A) \sim (\hat{\xi}^A,\hat{\zeta}^A)$ related by \eqref{eq-twistor-spinor-transform} can be thought of as a section $(\xi^A,\zeta^A)$ of the \emph{local twistor bundle} \cites{Penrose1986,Bailey1994} or \emph{spin tractor bundle} \cite{Hammerl2011}, and we shall refer to such a section as a \emph{tractor-spinor}. These are spinors for the group $\Spin(2m+3,\C)$. Tracing \eqref{eq-twistor-spinor2} yields
\begin{align}\label{eq-twistor-spinor3}
\nabla \ind{_a} \zeta \ind*{^a^B} & = - \frac{1}{2\sqrt{2}(n-1)} R \xi \ind*{^B} \, .
\end{align}
The integrability condition for the existence of a conformal Killing spinor is well-known, see e.g. \cite{Baum2010}. Here, we restate it in the context of pure spinor fields.
\begin{prop} \label{prop-int-cond-twistor-spinor} Let $\xi^A$ be a pure conformal Killing spinor on $(\mcM,g)$ with $\zeta^B := \frac{\sqrt{2}}{n} \gamma \ind{^a _{A}^B} \nabla_a \xi^{A}$. Then
\begin{align}\label{eq-twistor-spinor-int_cond}
\begin{aligned}
  C \ind{_{abcd}} \xi \ind*{^{cd}^D} & = 0 \, , \qquad \mbox{i.e.} \qquad {}^\mfC _\xi \Pi_2^0 (C) = 0 \, , \\
  C \ind{_{abcd}} \zeta \ind*{^{bcC}} - 2 \sqrt{2} A \ind{_{cab}} \xi \ind*{^c^E} & = 0 \, , \\
 A \ind{_{cab}} \xi \ind*{^c^A} \xi \ind*{^{ab}^B} & = 0 \, , \qquad \mbox{i.e.} \qquad {}^\mfA _\xi \Pi_0^0 (A) = 0 \, , \\
\end{aligned}
\end{align}
where $A \ind{_{abc}} := 2 \nabla_{[b} \Rho_{c]a}$ is the \emph{Cotton-York tensor} (see appendix \ref{sec-conformal}).
\end{prop}

\begin{prop}\label{prop-conformal-Killing-foliating}
 Let $\xi^{A}$ be a pure conformal Killing spinor on $(\mcM,g)$ with almost null structure $\mcN_\xi$. Set $\zeta^B := \frac{\sqrt{2}}{n} \gamma \ind{^a _{A}^B} \nabla_a \xi^{A}$. Then $\xi^{A}$ satisfies \eqref{eq-cond-22}, i.e.\
 \begin{align*}
  \xi \ind*{^{[A}} \left( \xi \ind*{^{aB]}} \nabla \ind{_a} \xi \ind*{_b^E}  \right) \xi \ind*{^b^{[C}} \xi \ind*{^{D]}} + \xi \ind*{^{[C}} \left( \xi \ind*{^{aD]}} \nabla \ind{_a} \xi \ind*{_b^E} \right) \xi \ind*{^b^{[A}} \xi^{B]} &  = 0 \, .
  \end{align*}

Further, $\xi^A$ satisfies \eqref{eq-foliating}, i.e. $\mcN_\xi$ is integrable and co-integrable, if and only if
\begin{align}\label{eq-pure-twistor-spinor}
  \zeta \ind*{^a^A} \zeta \ind*{_a^B} & = - \zeta^A \zeta^B  \, , & \zeta \ind*{^a^A} \xi \ind*{_a^B} & = \zeta \ind*{^A} \xi \ind*{^B} - 2 \, \xi \ind*{^A} \zeta \ind*{^B} \, ,
\end{align}
i.e. $\zeta^A$, if non-zero, is pure and its almost null structure $\mcN_\zeta$ intersects $\mcN_\xi$ in a totally null plane of dimension $m-1$ or $m$ at every point.

Suppose that $\xi^A$ satisfies \eqref{eq-foliating} so that $\zeta^A$ satisfies \eqref{eq-pure-twistor-spinor}. Then
\begin{align}\label{eq-twistor-spinor2-almost_foliating}
\left( \zeta \ind*{^a^{[A}} \nabla \ind{_a} \zeta \ind*{^b^B} \right) \zeta \ind*{_b^C} \zeta \ind*{^{D]}} & = 0 \, .
\end{align}
\end{prop}

\begin{proof}
To prove that $\xi^A$ satisfies \eqref{eq-cond-22}, it suffices to contract equation \eqref{eq-twistor-spinor} with $\xi \ind*{^a^A}$ and $\gamma \ind{^b_D^C} \xi \ind*{_b^A}$. We find
\begin{align*}
\left( \xi \ind*{^a^A} \nabla \ind{_a} \xi \ind*{^b^B} \right) \xi \ind*{_b^C} + \frac{1}{\sqrt{2}} \left( \xi \ind*{^{aA}} \zeta \ind*{_{ab}^B} \xi \ind*{^b^C} + \zeta \ind*{^B} \xi \ind*{^A} \xi \ind*{^C} \right) & = 0 \, .
\end{align*}
The second term is skew-symmetric in $AC$. Therefore, symmetrising over $AC$ yields \eqref{eq-cond-22}.

Next, suppose that $\xi^A$ satisfies \eqref{eq-foliating}, which is equivalent to
\begin{align*}
\xi^{aA} \nabla_a \xi^B = - \frac{1}{\sqrt{2}} \left( \xi^A \alpha^B + \beta^A \xi^B \right) \in \left( \mfS^{\frac{m}{2}} \otimes \mfS^{\frac{m-2}{2}} \right) \oplus \left( \mfS^{\frac{m-2}{2}} \otimes \mfS^{\frac{m}{2}} \right)
\end{align*}
at every point -- here $\mfS^{\frac{m}{2}} = \langle \xi^A \rangle$ and $\mfS^{\frac{m-2}{2}} = \im \xi_a^A$. By \eqref{eq-twistor-spinor}, the LHS is $-\frac{1}{\sqrt{2}} \xi^{aA} \zeta_a^B$ and must lie in the same module as the RHS. This in particular means that $\xi^A$ and $\zeta^A$ must satisfy \eqref{eq-pure-twistor-spinor} -- checking that indeed $\alpha^A = - \frac{1}{2} \beta^A = \zeta^A$ can be done by aplying \eqref{prop-pure-intersect2}. The converse, that \eqref{eq-pure-twistor-spinor} implies \eqref{eq-foliating}, is immediate.

Finally, assume $\xi^A$ satisfies \eqref{eq-foliating} so that \eqref{eq-pure-twistor-spinor} holds. Contracting equation \eqref{eq-twistor-spinor2} with $\zeta \ind*{^a^A}$ and $\gamma \ind{^b_D^C} \zeta \ind*{_b^A}$ leads to
\begin{align*}
\left( \zeta \ind*{^a^A} \nabla \ind{_a} \zeta \ind*{^b^B} \right) \zeta \ind*{_b^C} - \frac{1}{\sqrt{2}} \zeta \ind*{^a^A} \Rho \ind{_{ab}} \xi \ind*{^b^B} \zeta \ind*{^C} + 2 \sqrt{2} \, \zeta \ind*{^a^A} \Rho \ind{_{ab}} \zeta \ind*{^b^{[B}} \xi \ind*{^{C]}} & = 0 \, ,
\end{align*}
and the result \eqref{eq-twistor-spinor2-almost_foliating} follows by symmetry considerations.
\end{proof}

\begin{rem}
Using \eqref{eq-twistor-spinor-transform}, one checks that the statements of Proposition \ref{prop-conformal-Killing-foliating} are conformally invariant.

Further, the condition that the conformal Killing spinor $\xi^A$ be pure and $\zeta^A$ satisfy \eqref{eq-pure-twistor-spinor} is equivalent to the corresponding tractor-spinor $(\xi^A, \zeta^A)$ being a pure section of the local twistor bundle, i.e.\ it is a pure spinor for $\Spin(2m + 3, \C)$. See \cites{Hughston1988,Taghavi-Chabert2017}.
\end{rem}

\begin{exa}\label{exa-235}
Using the method of equivalence, Cartan \cite{Cartan1910} showed how to encode the invariance properties of certain
ODEs of Monge type in terms of a $(2, 3, 5)$-distribution, i.e. a rank-2 distribution $\mcN$ on a five-dimensional
smooth manifold, that bracket-generates the tangent bundle. This is more invariantly expressed as a $\G_2$-principal bundle equipped with a Cartan connection. In \cite{Nurowski2005}, Nurowski associates to this $(2, 3, 5)$-distribution a five-dimensional split-signature conformal structure, with respect to which $\mcN$ is totally null, with orthogonal complement $[\mcN , \mcN ]$. The general theory, expounded in the language of parabolic geometries, is given in  \cites{Cap2009,Hammerl2009}, more particularly, in \cite{Hammerl2011}, where it is shown how such manifolds are characterised by the existence of a real conformal Killing spinor, generic in the sense that $\xi^A \zeta_A \neq 0$. In five dimensions, this is consistent since \eqref{eq-cond-22} implies \eqref{eq-almost-foliating}. This example works equally in the holomorphic category.
\end{exa}

\paragraph{Killing spinors}
A holomorphic spinor field $\xi^A$ that is both a solution to the twistor equation \eqref{eq-twistor-spinor} and an eigenspinor of the Dirac operator, i.e.\ $\gamma \ind{^a_B^C} \nabla_a \xi^B = \lambda \, \xi^C$ for some holomorphic function $\lambda$ on $\mcM$, is known as a \emph{Killing spinor}. Otherwise put, $\xi^A$ satisfies the \emph{Killing equation}
\begin{align}\label{eq-Killing_spinor}
 \nabla \ind{_a} \xi \ind*{^A} + \lambda \frac{1}{n} \xi \ind*{_a^A} & = 0 \, .
\end{align}
That this equation is not conformally invariant is reflected in the geometric properties of its solutions. In particular, as a special case of \eqref{eq-twistor-spinor}, \eqref{eq-twistor-spinor2}, \eqref{eq-twistor-spinor3} and \eqref{eq-twistor-spinor-int_cond} with $\zeta \ind*{^A} = \lambda \, \xi \ind*{^A}$, we prove:
\begin{prop}\label{prop-int-cond-Killing-spinor}
 Let $\xi^A$ be a pure Killing spinor on $(\mcM,g)$ with almost null structure $\mcN_\xi$. Then 
\begin{align*}
{}^\mfC _\xi \Pi_2^0 (C) & = 0 \, , & & \mbox{i.e.} &   C \ind{_{abcd}} \xi \ind*{^{cd}^D} & = 0 \, , \\
{}^\mfA _\xi \Pi_2^0 (A) = {}^\mfA _\xi \Pi_2^1 (A) & = 0 & & \mbox{i.e.} & \xi \ind*{^a^A} A \ind{_{abc}} & = 0 \, , \\
{}^\mfF _\xi \Pi_0^1 ( \Phi ) & = 0 \, .
\end{align*}
Further, its eigenfunction $\lambda$ satisfies $\xi \ind*{^a^A} \nabla \ind{_a} \lambda = - \left( \lambda^2 + \frac{n}{4(n-1)} R \right) \xi \ind*{^A}$, and is thus constant along $\mcN_\xi$.
\end{prop}

The following proposition is straightforward.
\begin{prop}\label{prop-Killing-foliating}
 Let $\xi^A$ be a pure conformal Killing spinor on $(\mcM,g)$ with almost null structure $\mcN_\xi$. Set $\zeta^A := \frac{\sqrt{2}}{n} \nabla_a \xi^{aA}$. Then
 $\xi^A$ satisfies \eqref{eq-cond-12}, i.e.\ $\mcN_\xi$ is totally co-geodetic, if and only if $\xi^{[A} \zeta^{B]} = 0$, i.e. $\xi^A$ is a Killing spinor. This being the case, we have further $\left( \nabla \ind{_a} \xi \ind*{^b^{[A}} \right) \xi \ind*{_b^B} \xi \ind*{^{C]}} = 0$.
\end{prop}

\begin{rem}
The gist of Propositions \ref{prop-conformal-Killing-foliating} and \ref{prop-Killing-foliating} is the filtration of $\prb$-modules $\mfS^{\frac{m}{2}} \subset \mfS^{\frac{m-2}{2}} \subset \mfS^{\frac{m-4}{2}}$. The spinor $\zeta^A$ belonging to one of these submodules determines the geometric property of $\mcN_\xi$.
\end{rem}

The following result is analogous to the one given in even dimensions in \cite{Hammerl2016}.
\begin{prop}\label{prop-cKsp2par}
Let $\xi^A$ be a pure conformal Killing spinor whose associated null structure $\mcN_\xi$ is integrable and co-integrable. Then, locally, there exists a conformal rescaling such that $\xi^A$ is parallel, up to the freedom of adding to such a conformal rescaling any holomorphic function constant along the leaves of $\mathcal{N}_\xi^\perp$.
\end{prop}

\begin{proof}
We assume that $\mcN_\xi$ is integrable and co-integrable so that by Proposition \ref{prop-conformal-Killing-foliating}, $\xi^A$ and $\zeta^A := \frac{\sqrt{2}}{n} \nabla_a \xi^{aA}$ satisfy \eqref{eq-pure-twistor-spinor}. In particular, $\zeta^A \in \im \xi_a^A$. We must apply the transformation \eqref{eq-twistor-spinor-transform} to find a holomorphic conformal factor $\Omega$ such that $\hat{\zeta}^A =0$. First, we show that locally one can always find a holomorphic function $\phi$ such that $\xi^{[A} \zeta^{B]} = - \frac{1}{\sqrt{2}} \xi^{[A} \xi \ind{^a^{B]}} \nabla_a \phi$, which follows from the integrability of $\mathcal{N}_\xi$, the twistor equation \eqref{eq-twistor-spinor} and its prolongation \eqref{eq-twistor-spinor2}. This yields a conformal factor such that $\xi^A$ is a solution of the Killing spinor equation \eqref{eq-Killing_spinor}. One can then find a holomorphic function $\psi$ such that $\lambda \, \xi^{A} := \xi \ind{^a^{A}} \nabla_a \psi$, which yields a conformal factor that turns our Killing spinor into a parallel spinor. There is the freedom of adding to the scale a smooth function constant along $\mathcal{N}_\xi^\perp$.
\end{proof}

 A similar result is given in \cite{Lischewski2013a}.

\subsubsection{Relation to the Goldberg-Sachs theorem}
In four dimensions, the Goldberg-Sachs theorem \cite{Goldberg2009} gives a relation between the existence of integrable null structures and degeneracy conditions on the Weyl curvature -- for generalisations, see \cite{Gover2010}. A `coarse' higher-dimensional generalisation is given in \cite{Taghavi-Chabert2012}, which can be formulated in the following way in odd dimensions.

\begin{thm}[\cites{Taghavi-Chabert2011,Taghavi-Chabert2012}]
Assume $m \geq 2$. Let $[ \xi^A ]$ be a holomorphic projective pure spinor field on a $(2m+1)$-dimensional complex Riemannian manifold $(\mcM,g)$ with associated almost null structure $\mcN_\xi$. Suppose the Weyl tensor and the Cotton-York tensor satisfies the algebraic degeneracy conditions
\begin{align}\label{eq-GS}
\begin{aligned}
{}^\mfC _\xi \Pi_{-1}^0 (C) = {}^\mfC _\xi \Pi_{-1}^1 (C) = {}^\mfC _\xi \Pi_{-1}^2 (C) & = 0 \, , \qquad \mbox{i.e.} \qquad \xi \ind*{^a^A} \xi \ind*{^b^B} \xi \ind*{^c^{[C}} C \ind{_{abcd}} \xi \ind*{^{D]}} = 0 \, , \\
{}^\mfA _\xi \Pi_{-2}^0 (A) = {}^\mfA _\xi \Pi_{-2}^1 (A) & = 0 \, , \qquad \mbox{i.e.} \qquad \xi \ind*{^{[A}} \xi \ind*{^a^{B]}} \xi \ind*{^b^C} \xi \ind*{^c^D} A \ind{_{abc}} = 0 \, .
\end{aligned}
\end{align}
Suppose further that the Weyl tensor is otherwise generic. Then $[\xi^A]$ satisfies \eqref{eq-foliating}, i.e.\ $\mcN_\xi$ is integrable and co-integrable.
\end{thm}

In the light of Proposition \ref{prop-conformal-Killing-foliating} and Example \ref{exa-235}, there are pure spinor fields with non-integrable and non-co-integrable almost null structures, whose integrability condition satisfies \eqref{eq-GS}, but violates the genericity assumption by virtue of Proposition \ref{prop-int-cond-twistor-spinor}. This motivates the following conjecture improving \cite{Taghavi-Chabert2012}:
\begin{conjec}\label{conjec-GS}
Suppose that $[\xi^A]$ is a projective pure spinor field on a $(2m+1)$-dimensional non-conformally flat Einstein spin complex Riemannian manifold $(\mcM,g)$ such that the Weyl tensor satisfies $\xi \ind*{^a^A} \xi \ind*{^b^B} \xi \ind*{^c^{[C}} C \ind{_{abcd}} \xi \ind*{^{D]}} = 0$. Then $\xi \ind*{^A}$ satisfies \eqref{eq-cond-22}.
\end{conjec}
Weaker conditions such as \eqref{eq-cond-31} may well be possible too, but an investigation of the veracity of the above conjecture is beyond the scope of this article.

\begin{rem}
A non-conformally invariant Goldberg-Sachs theorem in dimension three is given in \cite{Nurowski2015}.
\end{rem}

\subsection{Application to real pseudo-Riemannian manifolds}\label{sec-ps-Riem}
Almost null structures on odd-dimensional real pseudo-Riemannian manifolds are subject to considerations regarding reality conditions and analyticity similar to the even-dimensional case -- see \cite{Taghavi-Chabert2016} for details. It suffices to say here that the real index of a pure spinor -- see section \ref{sec-real} -- allows for a wider range of geometric interpretations. For positive definite metric, the intrinsic torsion of an \emph{almost contact metric structure}, i.e.\ an odd-dimensional analogue of an almost Hermitian structure, was investigated in \cites{Alexiev1986,Chinea1990}. Finally, we emphasise that all the results obtained in the present article can be translated into the smooth category in the case of a spin oriented and time-oriented smooth peudo-Riemannian manifold of signature $(m,m+1)$ equipped with a \emph{real} projective pure spinor or a \emph{real} almost null structure.

\paragraph{Acknowlegments}
I would like to thank Andreas \v{C}ap, Pawe\l~Nurowski, Josef \v{S}ilhan, Dimitri Alekseevsky and more particularly, Michael Eastwood and the referee for useful discussions which led to the revision of this paper. I am also grateful to the University of Turin, at which the very final revisions were completed.

This work was funded by a SoMoPro (South Moravian Programme) Fellowship: it has received a financial contribution from the European
Union within the Seventh Framework Programme (FP/2007-2013) under Grant Agreement No. 229603, and is also co-financed by the South Moravian Region.

Revision of this paper was carried out while the author was on an Eduard \v{C}ech Institute postdoctoral fellowship GPB201/12/G028, and a GA\v{C}R (Czech Science Foundation) post-doctoral grant GP14-27885P.
\appendix
\section{Spinorial description of curvature tensors}\label{sec-spinor-descript}
We follow the notation of section \ref{sec-algebra} throughout, i.e. $\mfV$ is a $(2m+1)$-dimensional complex vector space equipped with a non-degenerate symmetric bilinear form $g_{ab}$ and a pure spinor $\xi^A$.

\subsection{Elements of the $\g_0$-modules of $\mfF$, $\mfA$ and $\mfC$}
We choose a pure spinor $\eta_A$ such that $\xi^A \eta_A= -\frac{1}{2}$ to split $\mfV$ as \eqref{eq-null-grading}. We shall use the elements $u_a$, $h_{ab}$ and $\omega_{ab}$ given by \eqref{eq-u-h} and \eqref{eq-omega}.  Upstairs and downstairs spinor indices will refer to $\mfS_{\frac{m-2}{2}} = \im \xi_a^A \cap \ker \eta_A$ and $\mfS_{-\frac{m-2}{2}} = \im \eta_{aA} \cap \ker \xi^A$ respectively.  A spinor will be referred to as \emph{(totally) tracefree}, if the contraction of any pair of indices with $I \ind*{_B^A}$, as given by \eqref{eq-identity-map}, vanishes, e.g. $\sigma \ind{_A^B} I \ind*{_B^A} = 0$. We now describe elements of the $\g_0$-modules given in Propositions \ref{prop-TFRicci-classification}, \ref{prop-Cotton-York-classification} and \ref{prop-Weyl-classification}.

\paragraph{The tracefree Ricci tensor}
 Let $\Phi \ind{_{ab}} \in \mfF$. Then
\begin{itemize}
 \item $\Phi \ind{_{ab}} \in \breve{\mfF}_0^1$ if and only if $\Phi_{ab} = \xi \ind*{_{\lp{a}} ^A} \eta \ind{_{\rp{b}} _B} \Phi \ind{_A^B}$
for some tracefree $\Phi \ind{_A^B}$;
\item $\Phi \ind{_{ab}} \in \breve{\mfF}_0^0$ if and only if $\Phi _{ab} = \Phi \left( u \ind{_a} u \ind{_b} + \frac{1}{n-1} h \ind{_{ab}} \right)$ for some complex $\Phi$;
\item $\Phi \ind{_{ab}} \in \breve{\mfF}_1^0$ if and only if $\Phi _{ab} = \xi \ind*{_{(a} ^A}  \Phi_A u \ind{_{b)}}$
for some $\Phi_{A}$;
\item $\Phi \ind{_{ab}} \in \breve{\mfF}_2^0$ if and only if $\Phi _{ab} = \xi \ind*{_a ^A} \xi \ind*{_b ^B} \Phi_{AB}$ for some $\Phi_{AB}= \Phi_{(AB)}$.
\end{itemize}
Using the duality $(\breve{\mfF}_{-i}^0)^* \cong \breve{\mfF}_i^0$, spinorial decompositions of elements of $\breve{\mfF}_{-i}^j$ for $i=1,2$ can be obtained by interchanging $\xi^A$ and $\eta_A$, and making appropriate changes of index structures.

\paragraph{The Cotton-York tensor}
 Let $A_{abc} \in \mfA$. Then
\begin{itemize}
\item $A_{abc} \in \breve{\mfA}_0^0$ if and only if $A \ind{_{abc}} = a \left( u \ind{_a} \omega \ind{_{bc}} - u \ind{_{\lb{b}}} \omega \ind{_{\rb{c}a}} \right)$ for some complex $a$;
\item $A_{abc} \in \breve{\mfA}_0^1$ if and only if $A \ind{_{abc}} = u \ind{_a} A \ind{_{bc}} - u \ind{_{\lb{b}}} A \ind{_{\rb{c}a}}$
where $A_{ab} = \xi \ind*{_{\lb{a}}^A} \eta \ind{_{\rb{b}}_B} A \ind{_A^B}$ for some tracefree $A \ind{_A^B}$;
\item $A_{abc} \in \breve{\mfA}_0^2$ if and only if $A \ind{_{abc}} = A \ind{_{a\lb{b}}} u \ind{_{\rb{c}}}$ where $A_{ab} = \xi \ind*{_{\lp{a}}^A} \eta \ind{_{\rp{b}}_B} A \ind{_A^B}$ for some tracefree $A \ind{_A^B}$;
 \item $A_{abc} \in \breve{\mfA}_1^0$ if and only if $A \ind{_{abc}} = A \ind{_a} \omega \ind{_{bc}} - A \ind{_{\lb{b}}} \omega \ind{_{\rb{c}a}} + \frac{3}{n-2} h \ind{_{a \lb{b}}} \omega \ind{_{\rb{c} d}} A \ind{^d}$ where $A_c = \xi \ind*{_c^C} A_C$ for some $A_A$;
 \item $A_{abc} \in \breve{\mfA}_1^1$ if and only if $A \ind{_{abc}} = u \ind{_a} u \ind{_{\lb{b}}} A \ind{_{\rb{c}}} + \frac{1}{n-2} h \ind{_{a \lb{b}}} A \ind{_{\rb{c}}}$ where $A_a = \xi \ind*{_a^A} A_A$ for some $A_A$;
 \item $A_{abc} \in \breve{\mfA}_1^2$ if and only if $A_{abc} =  \eta \ind{_a _C} \xi \ind*{_b ^A} \xi \ind*{_c ^B} A \ind{_{AB}^C}  - \xi \ind*{_a ^A} \xi \ind*{_{\lb{b}} ^B} \eta \ind{_{\rb{c}} _C} A \ind{_{AB}^C}$ for some tracefree $A \ind{_{AB}^C} = A \ind{_{[AB]}^C}$; 
 \item $A_{abc} \in \breve{\mfA}_1^3$ if and only if $A_{abc} = \xi \ind*{_a ^A} \xi \ind*{_{\lb{b}} ^B} \eta \ind{_{\rb{c}} _C} A \ind{_{AB}^C}$ for some tracefree $A \ind{_{AB}^C} = A \ind{_{(AB)}^C}$;
\item $A_{abc} \in \breve{\mfA}_2^0$ if and only if $A \ind{_{abc}} = u \ind{_a} A \ind{_{bc}} - u \ind{_{\lb{b}}} A \ind{_{\rb{c}a}}$
where $A_{ab} = \xi \ind*{_a^A} \xi \ind*{_b^B} A \ind{_{AB}}$ fo some $A \ind{_{AB}}=A \ind{_{[AB]}}$;
\item $A_{abc} \in \breve{\mfA}_2^1$ if and only if $A \ind{_{abc}} = A \ind{_{a\lb{b}}} u \ind{_{\rb{c}}}$ where $A_{ab} = \xi \ind*{_a^A} \xi \ind*{_b^B} A \ind{_{AB}}$ for some $A \ind{_{AB}}=A \ind{_{(AB)}}$;
 \item $A_{abc} \in \breve{\mfA}_3^0$ if and only if $A_{abc} = \xi \ind*{_a ^A} \xi \ind*{_b ^B} \xi \ind*{_c ^C} A \ind{_{ABC}}$
for some $A_{ABC} = A_{A[BC]}$ satisfying $A_{[ABC]} = 0$.
\end{itemize}
Using the duality $(\breve{\mfA}_{-i}^j)^* \cong \breve{\mfA}_i^j$, spinorial decompositions of elements of $\breve{\mfA}_{-i}^j$ for $i=1,2,3$ can be obtained by interchanging $\xi^A$ and $\eta_A$, and making appropriate changes of index structures.

\paragraph{The Weyl tensor}
Let $C_{abcd} \in \mfC$. Then
\begin{itemize}
 \item $C_{abcd} \in \breve{\mfC}_0^0$ if and only if $C \ind{_{abcd}} = c \left( 2 \, \omega \ind{_{ab}} \omega \ind{_{cd}} - 2 \, \omega \ind{_{a\lb{c}}} \omega \ind{_{\rb{d}b}} 
+ \frac{6}{n-2} \, h \ind{_{a \lb{c}}} h \ind{_{\rb{d}b}} \right)$ for some complex $c$;
 \item $C_{abcd} \in \breve{\mfC}_0^1$ if and only if
\begin{align*}
 C \ind{_{abcd}} & = \omega \ind{_{ab}} C \ind{_{cd}} + C \ind{_{ab}} \omega \ind{_{cd}} - 2 \, \omega \ind{_{\lb{a}|\lb{c}}} C \ind{_{\rb{d}|\rb{b}}}
- \frac{6}{n-3} \, \left( h \ind{_{\lb{a} | \lb{c}}} \omega \ind{_{\rb{d}}^e} C \ind{_{|\rb{b}e}} + h \ind{_{\lb{c} | \lb{a}}} \omega \ind{_{\rb{b}}^e} C \ind{_{|\rb{d}e}} \right) \, ,
\end{align*}
where $C \ind{_{cd}} := 2 \, \xi \ind*{_{\lb{c}} ^C} \eta \ind{_{\rb{d}} _D} C \ind{_C^D}$ for some tracefree $C \ind{_C^D}$;

 \item $C_{abcd} \in \breve{\mfC}_0^2$ if and only if $C_{abcd} = u \ind{_{\lb{a}}} C \ind{_{\rb{b}\lb{c}}} u \ind{_{\rb{d}}} 
- \frac{1}{n-3} \, h \ind{_{\lb{a} | \lb{c}}} C \ind{_{\rb{d}|\rb{b}}}$ where $C \ind{_{cd}} := 2 \, \xi \ind*{_{\lp{c}} ^C} \eta \ind{_{\rp{d}} _D} C \ind{_C^D}$ for some tracefree $C \ind{_C^D}$;

\item $C_{abcd} \in \breve{\mfC}_0^3$ if and only if
\begin{align*}
  C_{abcd} = \xi \ind*{_a ^A} \xi \ind*{_b ^B} \eta \ind{_c _C} \eta \ind{_d _D} C \ind{_{A B}^{C D}} + \xi \ind*{_c ^A} \xi \ind*{_d ^B} \eta \ind{_a _C} \eta \ind{_b _D} C \ind{_{A B}^{C D}} - 2 \, \xi \ind*{_{\lb{a}|} ^A} \xi \ind*{_{\lb{c}} ^C} \eta \ind{_{\rb{d} |} _D} \eta \ind{_{\rb{b}} _B} C \ind{_{A C}^{D B}}
\end{align*}
for some tracefree  $C \ind{_{A C}^{D B}} = C \ind{_{[A C]}^{[D B]}}$;

 \item $C_{abcd} \in \breve{\mfC}_0^4$ if and only if $C_{abcd} =  \xi \ind*{_{\lb{a}|} ^A} \xi \ind*{_{\lb{c}} ^C} \eta \ind{_{\rb{d}|} _D} \eta \ind{_{\rb{b}} _B} C  \ind{_{A C}^{D B}}$ for some tracefree $C \ind{_{A C}^{D B}} = C \ind{_{(A C)}^{(D B)}}$;

\item $C_{abcd} \in \breve{\mfC}_1^0$ if and only if
\begin{multline*}
 C_{abcd} = \omega \ind{_{ab}} C \ind{_{\lb{c}}} u \ind{_{\rb{d}}} + \omega \ind{_{cd}} C \ind{_{\lb{a}}} u \ind{_{\rb{b}}} - \omega \ind{_{\lb{a}|\lb{c}}} C \ind{_{\rb{d}}} u \ind{_{|\rb{b}}} - \omega \ind{_{\lb{c}|\lb{a}}} C \ind{_{\rb{b}}} u \ind{_{|\rb{d}}}  \\
+ \frac{3}{n-2} \left( h \ind{_{\lb{a} | \lb{c}}} u \ind{_{\rb{d}}} \omega \ind{_{|\rb{b}}^e} C \ind{_e}  + h \ind{_{\lb{c} | \lb{a}}} u \ind{_{\rb{b}}} \omega \ind{_{|\rb{d}}^e} C \ind{_e}\right) \, ,
\end{multline*}
where $C \ind{_a} = \xi \ind*{_a^A} C \ind{_A}$ for some $C \ind{_A}$;

\item $C_{abcd} \in \breve{\mfC}_1^1$ if and only if $C_{abcd} = u \ind{_{\lb{a}}} C \ind{_{\rb{b}cd}} + u \ind{_{\lb{c}}} C \ind{_{\rb{d}ab}}$ where $C \ind{_{cab}} = \eta \ind{_c _C} \xi \ind*{_a ^A} \xi \ind*{_b ^B} C \ind{_{A B}^C} - \xi \ind*{_c ^A} \xi \ind*{_{\lb{a}} ^B} \eta \ind{_{\rb{b}} _C} C \ind{_{A B}^C}$ for some tracefree $C \ind{_{AB}^C} = C \ind{_{[AB]}^C}$;

\item $C_{abcd} \in \breve{\mfC}_1^2$ if and only if $C_{abcd} = u \ind{_{\lb{a}}} C \ind{_{\rb{b}cd}} + u \ind{_{\lb{c}}} C \ind{_{\rb{d}ab}}$, where $C \ind{_{cab}} = \xi \ind*{_c ^A} \xi \ind*{_{\lb{a}} ^B} \eta \ind{_{\rb{b}} _C} C \ind{_{A B}^C}$ for some tracefree $C \ind{_{AB}^C} = C \ind{_{(AB)}^C}$;

\item $C_{abcd} \in \breve{\mfC}_2^0$ if and only if $C \ind{_{abcd}} = \omega \ind{_{ab}} C \ind{_{cd}} + C \ind{_{ab}} \omega \ind{_{cd}} - 2 \, \omega \ind{_{\lb{a}|\lb{c}}} C \ind{_{\rb{d}|\rb{b}}}$ where $C_{ab} := \xi \ind*{_a ^A} \xi \ind*{_b ^B} C_{A B}$ for some $C _{AB} = C _{[AB]}$;

\item $C_{abcd} \in \breve{\mfC}_2^1$ if and only if $C_{abcd} = u \ind{_{\lb{a}}} C \ind{_{\rb{b}\lb{c}}} u \ind{_{\rb{d}}} 
- \frac{1}{n-3} h \ind{_{\lb{a} | \lb{c}}} C \ind{_{\rb{d}|\rb{b}}}$ where $C \ind{_{cd}} := \xi \ind*{_c ^C} \xi \ind*{_d ^D} C \ind{_{CD}}$ for some $C _{AB} = C _{(AB)}$;

\item $C_{abcd} \in \breve{\mfC}_2^2$ if and only if $C_{abcd} = \xi \ind*{_a ^A} \xi \ind*{_b ^B} \xi \ind*{_{\lb{c}} ^C} \eta \ind{_{\rb{d}}_D} C \ind{_{A B C}^D} + \xi \ind*{_c ^A} \xi \ind*{_d ^B} \xi \ind*{_{\lb{a}} ^C} \eta \ind{_{\rb{b}}_D} C \ind{_{A B C}^D}$ for some $C \ind{_{ABC}^D} = C \ind{_{[AB]C}^D}$ satisfying $C \ind{_{[ABC]}^D} = 0$;

\item $C_{abcd} \in \breve{\mfC}_3^0$ if and only if $C_{abcd} = u \ind{_{\lb{a}}} C \ind{_{\rb{b}cd}} + u \ind{_{\lb{c}}} C \ind{_{\rb{d}ab}}$, where $C \ind{_{abc}} = \xi \ind*{_a ^A} \xi \ind*{_b ^B} \xi \ind*{_c ^C} C \ind{_{ABC}}$ for some $C _{ABC} = C _{[AB]C}$ satisfying $C _{[ABC]} = 0$;

\item $C_{abcd} \in \breve{\mfC}_4^0$ if and only if $C_{abcd} = \xi \ind*{_a ^A} \xi \ind*{_b ^B} \xi \ind*{_c ^C} \xi \ind*{_d ^D} C_{A B C D}$ for some $C _{ABCD} = C _{[AB][CD]}$ satisfying $C _{[ABC]D} = 0$.
\end{itemize}
Using the duality $(\breve{\mfC}_{-i}^j)^* \cong \breve{\mfC}_i^j$, spinorial decompositions of elements of $\breve{\mfC}_{-i}^j$ for $i=1,2,3$ can be obtained by interchanging $\xi^A$ and $\eta_A$, and making appropriate changes of index structures.

\subsection{Maps describing elements of $\prb$-modules of $\mfF$, $\mfA$ and $\mfC$}\label{sec-inv-maps}
The kernels of the following maps ${}^\mfF _\xi \Pi_i^j$, ${}^\mfA _\xi \Pi_i^j$ and ${}^\mfC _\xi \Pi_i^j$ are $\prb$-submodules of the spaces $\mfF$, $\mfA$ and $\mfC$, and are related to irreducible $\prb$-modules $\mfF_i^j$, $\mfA_i^j$ and $\mfC_i^j$ as described in Propositions \ref{prop-TFRicci-classification}, \ref{prop-Cotton-York-classification} and \ref{prop-Weyl-classification}.

\paragraph{The tracefree Ricci tensor}
For $\Phi \ind{_{ab}} \in \mfF$, define
\begin{align*}
{}^\mfF _\xi \Pi_{-2}^0 ( \Phi ) & := \xi \ind{^{\lb{A}}} \xi \ind{^a ^{\rb{B}}} \Phi \ind{_{a b}} \xi \ind{^b ^{\lb{C}}} \xi \ind{^{\rb{D}}} \, , &
{}^\mfF _\xi \Pi_{-1}^0 ( \Phi ) & := \xi \ind{^{\lb{A}}} \xi \ind{^a ^{\rb{B}}} \Phi \ind{_{a b}} \xi \ind{^b ^C} \, , \\
{}^\mfF _\xi \Pi_0^0 ( \Phi ) & := \xi \ind{^a ^A} \xi \ind{^b ^B} \Phi\ind{_{a b}} \, , &
{}^\mfF _\xi \Pi_0^1 ( \Phi ) & := \xi \ind{^{\lb{A}}} \xi \ind{^a ^{\rb{B}}} \Phi \ind{_{a b}} + \frac{1}{n-1} \gamma \ind{_b _C ^{\lb{A}}} \xi \ind{^c ^{\rb{B}}} \xi \ind{^d ^C} \Phi \ind{_{c d}} \, , \\
{}^\mfF _\xi \Pi_1^0 ( \Phi ) & := \xi \ind{^a ^A} \Phi \ind{_{a b}} \, , &
\end{align*}

\paragraph{The Cotton-York tensor}
For $A \ind{_{abc}} \in \mfA$, define
\begin{flalign*}
& {}^\mfA _\xi \Pi^0_{-3} ( A ) := \xi \ind{^{\lb{A}}} \xi \ind{^a^{\rb{B}}} \xi \ind{^b^{\lb{C}}} \xi \ind{^c^D} \xi \ind{^{\rb{E}}} A \ind{_{a b c}} \, , & \\
& {}^\mfA _\xi \Pi^0_{-2} ( A ) := \xi^{[A} \xi \ind{^a^{B}} \xi \ind{^b^{C]}} \xi \ind{^c^D} A \ind{_{a b c}}  \, , & 
& {}^\mfA _\xi \Pi^1_{-2} ( A ) := \xi^{[A} \xi \ind{^a^{B]}} \xi \ind{^b^{[C}} \xi^{D]} \xi \ind{^c^E} A \ind{_{a b c}} \qquad + \qquad \left( [AB] \leftrightarrow [CD] \right)  \, , &
\end{flalign*}
\begin{flalign*}
& {}^\mfA _\xi \Pi^0_{-1} ( A ) := \xi \ind{^{\lb{A}}} \xi \ind{^a^{\rb{B}}} \xi \ind{^{bc}^C} A \ind{_{a b c}} - \frac{1}{n-2} \xi \ind{^a^C} \xi \ind{^b^A} \xi \ind{^c^B} A \ind{_{abc}} \, ,  \qquad \qquad  {}^\mfA _\xi \Pi^1_{-1} ( A ) := \xi \ind{^a^A} \xi \ind{^b^B} \xi \ind{^c^C} A \ind{_{a b c}} \, , & \\
& {}^\mfA _\xi \Pi^2_{-1} ( A ) := \xi \ind{^{[A}} \xi \ind{^a^B} \xi \ind{^b^{C]}} A \ind{_{a b c}} + \frac{1}{2(n-3)} \xi \ind{^{[A}} \xi \ind{^a^{B|}} \xi \ind{^{bd}^D} A \ind{_{a b d}} \gamma \ind{_c_D^{|C]}} - \frac{1}{2(n-3)} \xi \ind{^a^D} \xi \ind{^b^{[A}} \xi \ind{^{d}^B} A \ind{_{a b d}} \gamma \ind{_c_D^{C]}}   \, , & \\
& {}^\mfA _\xi \Pi^3_{-1} ( A ) := \xi \ind{^{[A}} \xi \ind{^a^{B]}} \xi \ind{^b^{[C}} \xi^{D]} A \ind{_{a b c}} + \frac{3}{2(n+1)} \xi \ind{^{[A}} \xi \ind{^a^{B]}} \xi \ind{^{bd}^E} A \ind{_{a b d}} \gamma \ind{_c_E^{[C}} \xi^{D]} + \frac{1}{2(n+1)} \xi \ind{^a^E} \xi \ind{^b^C} \xi \ind{^{d}^D} A \ind{_{a b d}} \gamma \ind{_c_E^{[A}} \xi^{B]} \\ & \qquad \qquad \qquad \qquad \qquad \qquad \qquad \qquad \qquad \qquad \qquad \qquad \qquad \qquad + \qquad \qquad \left( [AB] \leftrightarrow [CD] \right)  \, , & 
\end{flalign*}
\begin{flalign*}
& {}^\mfA _\xi \Pi^0_0 ( A ) := \xi \ind{^a^A} \xi \ind{^{bc}^B} A \ind{_{a b c}} \, ,  \qquad \qquad  {}^\mfA _\xi \Pi^1_0 ( A ) := \xi \ind{^a^A} \xi \ind{^b^{[B}} A \ind{_{a b c}} \xi^{C]} + \frac{1}{n-1} \xi \ind{^a^A} \xi \ind{^{bd}^D} A \ind{_{a b d}} \gamma \ind{_c_D^{[B}} \xi^{C]} \, , \\
& {}^\mfA _\xi \Pi^2_0 ( A ) := \xi^{[A} \xi \ind{^a^{B]}} \xi \ind{^b^C} A \ind{_{a b c}} - \frac{1}{2} A \ind{_{c a b}} \xi \ind{^{ab}^{[A}} \xi^{B]} \xi \ind{^C} - \frac{1}{n-1} \xi \ind{^a^C} \xi \ind{^{bd}^D} A \ind{_{abd}} \gamma \ind{_c_D^{[A}} \xi^{B]} \, , & 
\end{flalign*}
\begin{flalign*}
& {}^\mfA _\xi \Pi^0_1 ( A ) := \xi \ind{^{bc}^C} A \ind{_{a b c}} \xi \ind{^D} + \frac{2}{n-2} \xi \ind{^b^B} \xi \ind{^b^C} A \ind{_{bca}} \, , &
& {}^\mfA _\xi \Pi^1_1 ( A ) := \xi \ind{^a^A} \xi \ind{^b^B} A \ind{_{a b c}}  \, ,  & 
\end{flalign*}
\begin{flalign*}
& {}^\mfA _\xi \Pi^2_1 ( A ) := A \ind{_{[a b] c}} \xi \ind{^c^{[C}} \xi \ind{^{D]}} + \frac{1}{2(n-3)} \gamma \ind{_{[a}_E^{[C}} \xi \ind{^{D]}} A \ind{_{b]cd}} \xi \ind{^{cd}^E}  + \frac{1}{n-3} \xi \ind{^c^E} \xi \ind{^d^{[C}} A \ind{_{cd[a}} \gamma \ind{_{b]}_E^{D]}}  \, , & \\
& {}^\mfA _\xi \Pi^3_1 ( A ) := A \ind{_{(a b) c}} \xi \ind{^c^{[C}} \xi \ind{^{D]}} - \frac{3}{2(n+1)} \gamma \ind{_{(a}_E^{[C}} \xi \ind{^{D]}} A \ind{_{b)cd}} \xi \ind{^{cd}^E}  - \frac{1}{n+1} \xi \ind{^c^E} \xi \ind{^d^{[C}} A \ind{_{cd(a}} \gamma \ind{_{b)}_E^{D]}} \, , & 
\end{flalign*}
\begin{flalign*}
& {}^\mfA _\xi \Pi^0_2 ( A ) := \xi \ind{^c^C} A \ind{_{[a b] c}} \, ,  & 
& {}^\mfA _\xi \Pi^1_2 ( A ) := \xi \ind{^c^C} A \ind{_{(a b) c}} \, .  & 
\end{flalign*}

\paragraph{The Weyl tensor}
For $C \ind{_{abcd}} \in \mfC$, define
\begin{flalign*}
 & {}^\mfC _\xi \Pi_{-4}^0 (C) := \xi \ind{^{\lb{A}}} \xi \ind{^a ^B} \xi \ind{^b ^{\rb{C}}} C \ind{_{abcd}} \xi \ind{^c ^{\lb{D}}} \xi \ind{^d ^E} \xi \ind{^{\rb{F}}} \, ,  & & {}^\mfC _\xi \Pi_{-3}^0 (C) := \xi \ind{^{\lb{A}}} \xi \ind{^a ^B} \xi \ind{^b ^{\rb{C}}} C \ind{_{abcd}} \xi \ind{^c ^D} \xi \ind{^d ^E} \, ,  &
\end{flalign*}
\begin{flalign*}
 & {}^\mfC _\xi \Pi_{-2}^0 (C) := \xi \ind{^{\lb{A}}} \xi \ind{^a ^B} \xi \ind{^b ^{\rb{C}}} C \ind{_{abcd}} \xi \ind{^{cd} ^D} \, ,  & 
 & {}^\mfC _\xi \Pi_{-2}^1 (C) := \xi \ind{^a ^A} \xi \ind{^b ^B} C \ind{_{abcd}} \xi \ind{^c ^C} \xi \ind{^d ^D} \, ,  & 
\end{flalign*}
\begin{multline*}
 {}^\mfC _\xi \Pi_{-2}^2 (C) := \xi \ind{^a ^A} \xi \ind{^b ^B} C \ind{_{abcd}} \xi \ind{^c ^{[C}} \xi \ind{^{D]}} + \frac{1}{n+1} \left( \xi \ind{^a^A} \xi \ind{^b^B} C \ind{_{abce}} \xi \ind{^{ce}^E} \gamma \ind{_d_E^{[C}} - \xi \ind{^a^{[C|}} \xi \ind{^b^{[A|}} C \ind{_{abce}} \xi \ind{^{ce}^E} \gamma \ind{_d_E^{|B]}} \right) \xi \ind{^{D]}} \\
 - \frac{1}{n-3} \gamma \ind{_d_E^{[A|}} \xi \ind{^a^E} \xi \ind{^b^{|B]}} C \ind{_{abde}} \xi \ind{^d^C} \xi \ind{^e^D} \pmod {\xi \ind{^{[A}} \alpha \ind{_d^{B][C}} \xi \ind{^{D]}} }\, ,
\end{multline*}
\begin{flalign*}
& {}^\mfC _\xi \Pi_{-1}^0 (C) := \xi \ind{^a ^A} \xi \ind{^b ^B} C \ind{_{abcd}} \xi \ind{^{cd} ^C} \, ,  & \\
& {}^\mfC _\xi \Pi_{-1}^1 (C) := \xi \ind{^a ^A} \xi \ind{^b ^{B}} C \ind{_{abcd}} \xi \ind{^d ^D} - \xi \ind{^{[A}} \xi \ind{^{ab} ^{B]}} C \ind{_{abcd}} \xi \ind{^d ^D} + \frac{1}{n-3} \xi \ind{^{ab}^D} C \ind{_{abed}}  \xi \ind{^e^E} \xi \ind{^d^{[A}} \gamma \ind{_c_E^{B]}} \, , & \\
& {}^\mfC _\xi \Pi_{-1}^2 (C) := \xi \ind{^a ^A} \xi \ind{^b ^{B}} C \ind{_{abcd}} \xi \ind{^d ^D} - \xi \ind{^{ab} ^D} C \ind{_{abcd}} \xi \ind{^d ^{[A}} \xi \ind{^{B]}} + \xi \ind{^{ab} ^{[A}} C \ind{_{abcd}} \xi \ind{^d ^{B]}} \xi \ind{^D} \\
& \qquad \qquad \qquad - \frac{1}{n+1} \left( \xi \ind{^{ab}^E} C \ind{_{abed}} \xi \ind{^e^A} \xi \ind{^d^B} \gamma \ind{_c_E^D} - \xi \ind{^{ab}^{[A}} C \ind{_{abed}} \xi \ind{^e^{B]}} \xi \ind{^d^E} \gamma \ind{_c_E^D} \right) \\
& \qquad \qquad \qquad \qquad \qquad - \frac{1}{n+1} \left( \xi \ind{^{ab}^E} C \ind{_{abde}} \xi \ind{^{de}^{[A}} \gamma \ind{_c_E^{B]}} \xi \ind{^D} - \xi \ind{^{ab}^E} C \ind{_{abde}} \xi \ind{^{de}^D} \gamma \ind{_c_E^{[A}} \xi \ind{^{B]}} \right) \pmod { \xi \ind{^D} \xi \ind{^{[A}} \alpha \ind{_d^{B]}} }  \, , & 
\end{flalign*}
\begin{flalign*}
& {}^\mfC _\xi \Pi_0^0 (C) := \xi \ind{^{ab} ^A} C \ind{_{abcd}} \xi \ind{^{cd} ^B} \, ,  \\
& {}^\mfC _\xi \Pi_0^1 (C) := \xi \ind{^{ab} ^{(A}} C \ind{_{abcd}} \xi \ind{^d ^{B)}} \xi \ind{^C} - \frac{1}{n-1} \xi \ind{^{ab} ^{E}}  \xi \ind{^{de} ^{(A}} C \ind{_{abde}} \gamma \ind{_c _E ^{B)}} \xi \ind{^C} - 2 \frac{n-1}{n-3} \xi \ind{^a ^C} \xi \ind{^b ^{(A}} C \ind{_{abcd}} \xi \ind{^d ^{B)}} \pmod{ \xi^A \xi^B \alpha \ind{_d^C} }\, ,   \\
& {}^\mfC _\xi \Pi_0^2 (C) := \xi \ind{^a ^A} \xi \ind{^b ^B} C \ind{_{abcd}} \xi \ind{^c ^C} \, ,  & 
\end{flalign*}
\begin{multline*}
{}^\mfC _\xi \Pi_0^3 (C) := \xi \ind{^a ^{[A}} C \ind{_{a[bc]d}} \xi \ind{^d ^B} \xi \ind{^{C]}} + \frac{1}{n-5} \left( \gamma \ind{_{[b}_E^{[A}} \xi \ind{^{dB}} \xi \ind{^{C]}} C \ind{_{c]dae}} \xi \ind{^{ae}^E} + \xi \ind{^a^{[A}} \xi \ind{^e^{B|}} C \ind{_{ae[b|f}} \xi \ind{^f^E} \gamma \ind{_{c]}_E^{|C]}} \right) \\
 - \frac{1}{2(n-3)(n-5)} \left( \xi \ind{^{aeE}} C \ind{_{aedf}} \xi \ind{^{df}^{[A}} \gamma \ind{_{[b}_E^B} \xi \ind*{_{c]}^{C]}} - \xi \ind{^{aeE}} C \ind{_{aedf}} \xi \ind{^{df}^{[A}} \gamma \ind{_{bc}_E^B} \xi ^{C]} \right) \, ,
\end{multline*}
\begin{multline*}
 {}^\mfC _\xi \Pi_0^4 (C) := \xi^{[A} \xi \ind{^a ^{B]}} C \ind{_{a(bc)d}} \xi \ind{^d ^{[C}} \xi \ind{^{D]}} + \frac{1}{n+3} \left( \xi^{[A} \xi \ind*{_{(b}^{B]}} C \ind{_{c)dae}} \xi ^{d[C} \xi \ind{^{ae}^{D]}} + \xi^{[A} \gamma \ind{_{(b}_E^{B]}} C \ind{_{c)dae}} \xi^{dE} \xi \ind{^{ae}^{[C}} \xi ^{D]} \right. \\
\left. + 2 \, \xi^{[A} \gamma \ind{_{(b}_E^{B]}} C \ind{_{c)dae}} \xi^{d[C} \xi^{D]} \xi \ind{^{ae}^E} - \xi^{[A} \gamma \ind{_{(b}_E^{B]}} C \ind{_{c)dae}} \xi^{dE} \xi \ind{^{aC}} \xi ^{eD}
   \right) \\
 - \frac{1}{(n+1)(n+3)} \left( \xi^{[A} \gamma \ind{_{(b}_E^{B]}} \gamma \ind{_{c)}_F^{[C}} \xi^{D]} \xi \ind{^{aeE}} C \ind{_{aedf}} \xi \ind{^{dfF}} + \frac{1}{2} \xi \ind*{_b^{[A}} \xi \ind{^{aeB]}} C \ind{_{aedf}} \xi \ind{^{df[C}} \xi^{D]} \right. \\
\left.  + \xi^{[A} \gamma \ind{_{(b}_E^{B]}} \gamma \ind{_{c)}_F^{[C}} \xi^{aD]} \xi^{fF} \xi \ind{^{aeE}} C \ind{_{aedf}} - \gamma \ind{_{(b}_E^{[A}} \xi \ind{^{aeB]}} \xi \ind*{_{c)}^{[C}} \xi^{dD]} \xi^{fE} C \ind{_{aedf}}  \right) \\
+ \left( [AB] \leftrightarrow [CD] \right)  \, , 
\end{multline*}
\begin{flalign*}
 & {}^\mfC _\xi \Pi_1^0 (C) := \xi \ind{^{ab} ^B} C \ind{_{abcd}} \xi \ind{^d ^C} \, , &
 \end{flalign*}
 \begin{multline*} 
 {}^\mfC _\xi \Pi_1^1 (C) := \xi \ind{^a ^A} C \ind{_{a[bc]d}} \xi \ind{^d ^D} + \frac{1}{2(n-3)} \left( \gamma \ind{_{[b|}_E^{[A|}} \xi \ind{^{ae}^E} C \ind{_{ae|c]d}} \xi \ind{^d^{|D]}} + \gamma \ind{_{[b|}_E^{[A}} \xi \ind{^{ae}^{D]}} C \ind{_{ae|c]d}} \xi \ind{^d^E} \right) \\
 - \frac{1}{2(n-1)(n-3)} \left( \xi \ind{^{ae}^E} C \ind{_{aefd}} \xi \ind{^{fd}^{[A}} \gamma \ind{_{bc}_E^{D]}} \right) \, ,
\end{multline*}
\begin{multline*} 
 {}^\mfC _\xi \Pi_1^2 (C) := \xi \ind{^a ^A} C \ind{_{a(bc)d}} \xi \ind{^d ^D} + \frac{3}{2(n+1)} \left( \gamma \ind{_{(b|}_E^{(A|}} \xi \ind{^{ae}^E} C \ind{_{ae|c)d}} \xi \ind{^d^{|D)}} + \gamma \ind{_{(b|}_E^{(A}} \xi \ind{^{ae}^{D)}} C \ind{_{ae|c)d}} \xi \ind{^d^E} \right) \\
  + \frac{3}{2(n-1)(n+1)} \left( \xi \ind{^{ae}^E} C \ind{_{aefd}} \xi \ind{^{fd}^F} \gamma \ind{_{(b}_E^A}\gamma \ind{_{c)}_F^D} \right) \pmod { \xi^A \xi^D C \ind{_{bc}} } \, , 
 \end{multline*}
\begin{flalign*}
 & {}^\mfC _\xi \Pi_2^0 (C) := \xi \ind{^{ab} ^A} C \ind{_{abcd}}  \, ,  & 
 & {}^\mfC _\xi \Pi_2^1 (C) := \xi \ind{^a ^A} C \ind{_{a(bc)d}} \xi \ind{^d ^B} \, ,  & 
\end{flalign*}
\begin{multline*}
 {}^\mfC _\xi \Pi_2^2 (C) := \xi \ind{^{[A}} \xi \ind{^a^{D]}} C \ind{_{abcd}}  - \frac{1}{n-3} \left( \xi \ind{^a^{[A}} C \ind{_{ab[c|e}} \gamma \ind{_{|d]E}^{D]}} \xi ^{eE} - \xi \ind{^a^{[A}} C \ind{_{a[c|be}} \gamma \ind{_{|d]E}^{D]}} \xi ^{eE}  \right) \\
 - \frac{1}{n+1}  \left( \xi \ind{^{ae}^E} C \ind{_{aecd}} \gamma \ind{_b_E^{[A}} \xi^{D]} - \xi \ind{^{ae}^E} C \ind{_{aeb[c}} \gamma \ind{_{d]}_E^{[A}} \xi^{D]} \right) \\
- \frac{3}{2(n+1)(n-3)} \left( \xi \ind{^{ae}^E} C \ind{_{aebf}} \xi \ind{^f^F} \gamma \ind{_{[c}_E^{[A}} \gamma \ind{_{d]}_F^{D]}} - \xi \ind{^{ae}^E} C \ind{_{ae[c|f}} \xi \ind{^f^F} \gamma \ind{_{|d]}_E^{[A}} \gamma \ind{_b_F^{D]}} \right) \\
+ \frac{1}{2(n+1)(n-3)} \left( \xi \ind{^{ae}^{[A|}} C \ind{_{aebf}} \xi \ind{^f^E} \gamma \ind{_{cd}_E^{|D]}} - \xi \ind{^{ae}^{[A|}} C \ind{_{ae[c|f}} \xi \ind{^f^E} \gamma \ind{_{|d]b}_E^{|D]}} \right) \\
- \frac{2}{(n+1)(n-3)} \left( \xi \ind{^{ae}^E} C \ind{_{aebf}} \xi \ind{^f^{[A}} \gamma \ind{_{cd}_E^{D]}} - \xi \ind{^{ae}^E} C \ind{_{ae[c|f}} \xi \ind{^f^{[A}} \gamma \ind{_{|d]b}_E^{D]}} \right) \\
+ \frac{2}{(n+1)(n-1)(n-3)} \xi \ind{^{ae}^E} C \ind{_{aefg}} \xi \ind{^{fg}^F} \left( \gamma \ind{_b_E^{[A}} \gamma \ind{_{cd}_F^{D]}} - \gamma \ind{_{[c}_E^{[A}} \gamma \ind{_{d]b}_F^{D]}} \right) \, ,
\end{multline*}
\begin{flalign*}
 & {}^\mfC _\xi \Pi_3^0 (C) := \xi \ind{^a ^A} C \ind{_{abcd}} \, ,  &
\end{flalign*}

\section{Spinor calculus in three and five dimensions}\label{sec-3-5dim}
In this appendix, we give a brief description of spinor calculus in dimensions three and five.

\subsection{Three dimensions} \label{sec-3dim}
Let $(\mcM,g)$ be a three-dimensional complex Riemannian manifold equipped with a holomorphic volume form and a holomorphic spin structure. The spin group is the complex special linear group $\SL(2,\C)$ acting on two-dimensional spinor space $\mfS$ and its dual $\mfS^*$, which we shall identify by means of volume forms $\varepsilon_{AB}$ and $\varepsilon^{AB}$. All spinors are pure. By and large, this is analogous to the two-spinor calculus of \cite{Penrose1986}, except that there is no `primed' spinor space.
We can convert tensorial quantities into spinorial ones by means of the normalised $\gamma$-matrices $\frac{1}{\sqrt{2}} \gamma \ind{_a^{AB}}$, which are symmetric in their spinor indices, and satisfy the identity
\begin{align*}
	\gamma \ind{_a _A ^B} \gamma \ind{^a _C ^D} & = - \delta \ind*{_A ^D} \delta \ind*{_C ^B} + \varepsilon \ind{_{AC}} \varepsilon \ind{^{BD}} \, , & \mbox{i.e.} & &
	\gamma \ind{_a _{AB}} \gamma \ind{^a _{CD}} & = - 2 \, \varepsilon \ind{_{A\lp{C}}} \varepsilon \ind{_{\rp{D}B}}  \, .
\end{align*}
The standard representation $\mfV$ of $\SO(3,\C)$ is isomorphic to $\odot^2 \mfS$, and, by Hodge duality, to $\wedge^2 \mfV$. There is no Weyl tensor in dimension three, while the tracefree Ricci tensor and the Cotton-York tensor are represented by totally symmmetric spinors $\Phi_{ABCD}$ and $A_{ABCD}$ respectively.

\subsubsection{Projective spinor fields}
Let $[\xi^A]$ be a holomorphic projective pure spinor field. Then, unlike in in higher odd dimensions, its stabiliser $P$, with Lie algebra $\prb$, at a point induces a $|1|$-grading on the Lie algebra $\g \cong \wedge^2 \mfV$ of $\Spin(3,\C)$. As in dimension four, 
the spinor $\xi^A$ defines a $P$-invariant filtration $\mfS^{\frac{k}{2}} \subset \mfS^{\frac{k}{2}-1} \subset \ldots \subset \mfS^{-\frac{k}{2}+1} \subset \mfS^{-\frac{k}{2}}$
on $\mfS^{-\frac{k}{2}} := \odot^k \mfS$, where $\mfS^{\frac{k-2\ell+2}{2}} := \left\{ \phi \ind{_{A_1 \ldots A_k}} \in \mfS^{-\frac{k}{2}} : \phi \ind{_{A_1 \ldots A_\ell A_{\ell+1} \ldots A_k}} \xi \ind{^{A_1}} \ldots \xi \ind{^{A_\ell}} \right\}$, and $\xi^A$ is said to be a \emph{principal spinor} of $\phi \ind{_{A_1 \ldots A_k}}$ if it lies in $\mfS^{-\frac{k}{2}+1}$.
%


\paragraph{Intrinsic torsion}
The projective spinor field $[\xi^A]$ induces a $P$-invariant filtration $\mfW^0 \subset \mfW^{-1} \subset \mfW^{-2}$ on the $\prb$-module $\mfW := \mfV \otimes \left( \g/\prb \right)$ of intrinsic torsions. From a geometric point of view, the associated almost null structure $\mcN_\xi$ of $[\xi^A]$ is of rank-$1$ and thus always integrable. The relation between $\mfW$ and the geometric properties of $\mcN_\xi$ and $\mcN_\xi^\perp$ is given below.
\begin{prop}
 Let $[ \xi^A ]$ be a holomorphic projective spinor field on $(\mcM,g)$ with associated null structure $\mcN_\xi$. Denote by $\nabla_{AB}$ the Levi-Civita connection of $g$. Then, pointwise, the intrinsic torsion of $[ \xi^A ]$
 \begin{itemize}
  \item lies in $\mfW^{-1}$ if and only if $\xi \ind{^A} \xi \ind{^B} \xi \ind{^C} \nabla \ind{_{AB}} \xi \ind*{_C} = 0$
 if and only if $\mcN_\xi$ is co-integrable if and only if $\mcN_\xi$ is (totally) geodetic;
  \item lies in $\mfW^0$ if and only if $\xi \ind{^B} \xi \ind{^C} \nabla \ind{_{AB}} \xi \ind*{_C} = 0$
 if and only if $\mcN_\xi$ is (totally) co-geodetic;
  \item vanishes if and only if $\xi \ind{^C} \nabla \ind{_{AB}} \xi \ind*{_C} = 0$.
 \end{itemize}
\end{prop}

\begin{rem}
The above conditions are equivalent to the null vector field $k \ind{^{AB}} := \xi \ind{^A} \xi \ind{^B}$ being geodetic, dilation free and recurrent respectively. The properties of null structures in dimension three were also studied in \cite{Nurowski2015} in the context of a Goldberg--Sachs-type theorem.
\end{rem}

\subsection{Five dimensions} \label{sec-5dim}
Let $(\mcM,g)$ be a five-dimensional complex Riemannian manifold equipped equipped with holomorphic volume form and a holomorphic spin structure. We first work at a point. The spin group is isomorphic to the complex symplectic group $\Sp(4,\C)$, so that the spinor space $\mfS$ is a four-dimensional complex vector space equipped with non-degenerate skew-symmetric bilinear form $\gamma_{AB}$ with inverse $\gamma^{AB}$, i.e. $\gamma_{AC} \gamma^{BC} = \delta_A^B$, by means of which we shall lower and raise indices. All spinors are pure. Tensor indices are converted into spinorial ones by means of the normalised skewsymmetric $\gamma$-matrices $\frac{\ii}{2} \gamma \ind{_a^{AB}}$, tracefree with respect to $\gamma \ind{_{AB}}$, which satisfy
\begin{align}\label{eq-g-epsilon}
	\gamma \ind{_a _A ^B} \gamma \ind{^a _C ^D} & = \delta \ind*{_A ^B} \delta \ind*{_C ^D} -  2 \, \delta \ind*{_A ^D} \delta \ind*{_C ^B} - 2 \, \gamma \ind{_{AC}} \gamma \ind{^{BD}} \, ,
&
\text{i.e.} & & 
	\gamma \ind{_a _{AB}} \gamma \ind{^a _{CD}} & = \gamma \ind{_{AB}} \gamma \ind{_{CD}} + 4 \, \gamma \ind{_{A \lb{C}}} \gamma \ind{_{\rb{D} B}} \, .
\end{align}
In particular, we have $\mfV \cong ( \wedge^2 \mfS )_\circ$ and $\wedge^2 \mfV \cong \odot^2 \mfS$ where $\mfV$ is the standard representation of $\SO(5,\C)$. The tracefree Ricci tensor, the Weyl tensor and the Cotton tensor admit the spinorial expressions
\begin{align*}
\Phi \ind{_{ABCD}} & = \Phi_{[AB][CD]} \, , &
C \ind{_{ABCD}} & = C_{(ABCD)} \, , &
 A_{A B C D} & = A_{[AB](CD)} \, , 
\end{align*}
respectively, all of which are completely tracefree, and where $\Phi_{[ABC]D} = 0$, $A_{[A B C] D} = 0$.

\subsubsection{Projective spinor fields}
Let $[ \xi^A ]$ be a projective spinor field on $(\mcM,g)$ with stabiliser $P \subset \Spin(5,\C)$ at a point.
Following section \ref{sec-algebra}, we have the induced $P$-invariant filtrations $\mfS^1 \subset \mfS^0 \subset \mfS^{-1}$ and $\mfV^1 \subset \mfV^0 \subset \mfV^{-1}$ where
\begin{align*}
 \mfS^{-1} & := \mfS \, , & \mfS^0 & := \{ \alpha^A \in \mfS : \alpha_A \xi^A = 0 \} \, , & \mfS^1 & := \langle \xi^A \rangle = \{ \alpha^A \in \mfS : \alpha_{[A} \xi_{B]} = 0 \} \, , \\
  \mfV^{-1} & := \mfV \, , & 
 \mfV^0 & := \left\{ V \ind{^{AB}} \in \mfV : \xi_C V \ind{^{C\lb{A}}} \xi \ind{^{\rb{B}}} = 0 \right\} \, , &
 \mfV^1 & := \left\{ V \ind{^{AB}} \in \mfV : \xi_C V \ind{^{CA}} = 0 \right\} \, .
\end{align*}
Equivalently, $\mfV^1 = \left\{ V \ind{^{AB}} \in \mfV : \xi_{[A} V_{BC]} = 0 \right\}$.
Similarly, we can express the various $P$-invariant submodules of $\g \cong \odot^2 \mfS$ in terms of the maps
\begin{align*}
 {}^\g _\xi \Pi^0_{-2} ( \phi ) & := \xi^A \xi^B \phi_{A B} \, , &
 {}^\g _\xi \Pi^0_{-1} ( \phi ) & := \xi^A \phi \ind{_A ^{\lb{B}}} \xi^{\rb{C}} \, , \\
 {}^\g _\xi \Pi^0_0 ( \phi ) & := \xi^A \phi \ind{_A ^B} \, , &
 {}^\g _\xi \Pi^1_0 ( \phi ) & := \xi_{\lb{A}} \phi \ind{_{\rb{B} \lb{C}}} \xi_{\rb{D}} \, , &
 {}^\g _\xi \Pi^0_1 ( \phi ) & := \phi \ind{_{A \lb{B}}} \xi_{\rb{C}} \, ,
\end{align*}
where $\phi \ind{_{AB}} = \phi \ind{_{(AB)}}$.

The explicit expressions for the maps ${}^\mfF _\xi \Pi_i^j$, ${}^\mfA _\xi \Pi_i^j$ and ${}^\mfC _\xi \Pi_i^j$ defined in section \ref{sec-curvature} can be significantly simplified.
For $\Phi \ind{_{ABCD}} \in \mfF$, we have
\begin{align*}
{}^\mfF _\xi \Pi^0_{-2} ( \Phi ) & := \xi_{\lb{F}} \Phi_{\rb{A} B C \lb{D}} \xi^B \xi^C \xi_{\rb{E}} \, , &
{}^\mfF _\xi \Pi^0_{-1} ( \Phi ) & := \Phi_{A B C \lb{D}} \xi^B \xi^C \xi_{\rb{E}} \, , \\
{}^\mfF _\xi \Pi^0_{0} ( \Phi ) & := \Phi_{A B C D} \xi^B \xi^C \, , &
{}^\mfF _\xi \Pi^1_{0} ( \Phi ) & := \xi \ind*{_{[A}} \Phi \ind{_{B] E C D}} \xi^E + \varepsilon \ind{_{[A|[C}} \Phi \ind{_{D]EF|B]}} \xi \ind{^E} \xi \ind{^F} \, , \\
{}^\mfF _\xi \Pi^0_{1} ( \Phi ) & := \Phi_{A B C D} \xi^B \, .
\end{align*}
For $A \ind{_{ABCD}} \in \mfA$, we have
\begin{align*}
 {}^\mfA _\xi \Pi^0_{-3} ( A ) & := \xi_{\lb{E}} A_{\rb{A} B C D} \xi^B \xi^C \xi^D \, , \\
 {}^\mfA _\xi \Pi^0_{-2} ( A ) & := A \ind{_{A B C D}} \xi \ind{^B} \xi \ind{^C} \xi \ind{^D} \, , \qquad \qquad \qquad \qquad
 {}^\mfA _\xi \Pi^1_{-2} ( A ) := \xi \ind*{_{[A}} A \ind{_{B] EF [C}} \xi \ind{^E} \xi \ind{^F} \xi \ind*{_{D]}} + \xi \ind*{_{[C}} A \ind{_{D] EF [A}} \xi \ind{^E} \xi \ind{^F} \xi \ind*{_{B]}} \, , \\
 {}^\mfA _\xi \Pi^0_{-1} ( A ) & := 4 \, \xi \ind*{_{[A}} A \ind{_{B]DEC}} \xi \ind{^D} \xi \ind{^E} - A \ind{_{ABDE}} \xi \ind{^D} \xi \ind{^E} \xi \ind*{_C} \, , \qquad \qquad \qquad
 {}^\mfA _\xi \Pi^1_{-1} ( A ) := A_{A B C \lb{D}} \xi^B \xi^C \xi_{\rb{E}} \, , \\
 {}^\mfA _\xi \Pi^3_{-1} ( A ) & := \xi_{\lb{A}} A_{\rb{B} G \lb{C} | \lb{E}} \xi^G \xi_{\rb{F}} \xi_{|\rb{D}} + \xi_{\lb{C}} A_{\rb{D} G \lb{E} | \lb{A}} \xi^G \xi_{\rb{B}} \xi_{|\rb{F}} + \xi_{\lb{E}} A_{\rb{F} G \lb{A} | \lb{C}} \xi^G \xi_{\rb{D}} \xi_{|\rb{B}} \, ,  \\
 {}^\mfA _\xi \Pi^0_{0} ( A ) & := A_{A B C D} \xi^B \xi^C  \, , \qquad \qquad \qquad \qquad \qquad \qquad 
 {}^\mfA _\xi \Pi^1_{0} ( A ) := \xi \ind{^F} A \ind{_{A F [B| [D}} \xi \ind*{_{E]}} \xi \ind*{_{|C]}} \, , \\
 {}^\mfA _\xi \Pi^2_{0} ( A ) & := \xi \ind*{_{[A}} \xi \ind{^F} A \ind{_{B] F E [C}} \xi \ind*{_{D]}} + \xi \ind*{_{[A}} \xi \ind{^F} A \ind{_{B] E F [C}} \xi \ind*{_{D]}} \, , \\
 {}^\mfA _\xi \Pi^0_{1} ( A ) & := \xi \ind*{_{[A}} \xi \ind{^E} A_{B] E C D} - \xi \ind{^E} A \ind{_{ABE(C}} \xi \ind*{_{D)}} \, , \qquad \qquad \qquad \qquad \qquad \qquad 
 {}^\mfA _\xi \Pi^1_{1} ( A ) := \xi \ind{^E} A \ind{_{AEB[C}} \xi \ind*{_{D]}} \\
 {}^\mfA _\xi \Pi^3_{1} ( A ) & := A_{A B \lb{C} | \lb{E}} \xi_{\rb{F}} \xi_{|\rb{D}} + A_{C D \lb{E} | \lb{A}} \xi_{\rb{B}} \xi_{|\rb{F}} + A_{E F \lb{A} | \lb{C}} \xi_{\rb{D}} \xi_{|\rb{B}} \, , \\
 {}^\mfA _\xi \Pi^0_{2} ( A ) & := \xi \ind{^D} A \ind{_{ADBC}} \, ,  \qquad \qquad \qquad \qquad \qquad \qquad \qquad \qquad 
 {}^\mfA _\xi \Pi^1_{2} ( A ) := A \ind{_{ABC[D}} \xi \ind*{_{E]}} + A \ind{_{DEC[A}} \xi \ind*{_{B]}} \, .
\end{align*}
Finally, for $C \ind{_{ABCD}} \in \mfC$, we have
\begin{align*}
 {}^\mfC _\xi \Pi^0_{-4} ( C ) & := C_{A B C D} \xi^A \xi^B \xi^C \xi^D \, , &
 {}^\mfC _\xi \Pi^0_{-3} ( C ) & := C_{A B C \lb{D}} \xi^A \xi^B \xi^C \xi_{\rb{E}} \, , \\
 {}^\mfC _\xi \Pi^0_{-2} ( C ) & := C_{A B C D} \xi^A \xi^B \xi^C \, , &
 {}^\mfC _\xi \Pi^2_{-2} ( C ) & := \xi_{\lb{F}} C_{\rb{A} B C \lb{D}} \xi^B \xi^C \xi_{\rb{E}} \, , \\
 {}^\mfC _\xi \Pi^0_{-1} ( C ) & := C_{A B C \lb{D}} \xi^B \xi^C \xi_{\rb{E}} \, , &
 {}^\mfC _\xi \Pi^1_{-1} ( C ) & := \xi_{\lb{F}} C_{\rb{A} B \lb{C} | \lb{D}} \xi^B \xi_{\rb{E}} \xi_{|\rb{F}} \, , \\
 {}^\mfC _\xi \Pi^0_0 ( C ) & := C_{A B C D} \xi^B \xi^C \, , &
 {}^\mfC _\xi \Pi^2_0 ( C ) & := \xi_{\lb{F}} C_{\rb{A} B C \lb{D}} \xi^B \xi_{\rb{F}} \, , \\
 {}^\mfC _\xi \Pi^4_0 ( C ) & := \xi_{\lb{G}} \xi_{|\lb{F}} C_{\rb{A} | \rb{B} \lb{C} | \lb{D}} \xi_{\rb{E}} \xi_{|\rb{F}} \, , \\
 {}^\mfC _\xi \Pi^0_1 ( C ) & := \xi_{\lb{F}} C_{\rb{A} B C D} \xi^B \, , &
 {}^\mfC _\xi \Pi^1_1 ( C ) & := \xi_{\lb{F}} C_{\rb{A} B \lb{C} | \lb{D}} \xi_{\rb{E}} \xi_{|\rb{F}} \, , \\
 {}^\mfC _\xi \Pi^0_2 ( C ) & := C_{A B C D} \xi^B \, , &
 {}^\mfC _\xi \Pi^2_2 ( C ) & := \xi_{\lb{F}} C_{\rb{A} B C \lb{D}} \xi_{\rb{E}} \, , \\
 {}^\mfC _\xi \Pi^0_3 ( C ) & := C_{A B C \lb{D}} \xi_{\rb{E}} \, .
\end{align*}

\paragraph{Intrinsic torsion}
Denote by $\nabla_{AB}$ the Levi-Civita connection of $g$. Then the differential characterisations of the intrinsic torsion of $[ \xi^A ]$ can be re-expressed as
\begin{align}
 \eqref{eq-cond-31} & & \Longleftrightarrow &  & \xi_{\lb{A}} \left( \xi^C \nabla_{\rb{B}C} \xi^D \right) \xi_D & = 0 \, , \label{eq-cond5d-31}  \\
\eqref{eq-cond-20} & & \Longleftrightarrow &  & \left( \xi^B \nabla_{AB} \xi^C \right) \xi_C & = 0 \, , \label{eq-cond5d-20} \\
 \eqref{eq-cond-21} & & \Longleftrightarrow &  &  \left( \xi^D \nabla \ind{_{D[A}} \xi \ind*{_{B}} \right) \xi \ind*{_{C]}} & = 0 \, , \label{eq-cond5d-21} \\
 \eqref{eq-cond-22} & & \Longleftrightarrow &  &  \xi_{\lb{A}} \left( \xi^E \nabla_{\rb{B}E} \xi_{\lb{C}} \right) \xi_{\rb{D}} + \xi_{\lb{C}} \left( \xi^E \nabla_{\rb{D}E} \xi_{\lb{A}} \right) \xi_{\rb{B}} & = 0 \, , \label{eq-cond5d-22} \\
\eqref{eq-cond-10} & & \Longleftrightarrow &  &  (\nabla_{A B} \xi^C ) \xi_C & = 0 \, , \label{eq-cond5d-10} \\
\eqref{eq-cond-12} & & \Longleftrightarrow &  &  \left( \xi^D \nabla_{A D} \xi_{\lb{B}} \right) \xi_{\rb{C}} & = 0 \, , \label{eq-cond5d-12} \\
 \eqref{eq-cond00} & & \Longleftrightarrow &  &  (\nabla_{A C} \xi^C ) \xi^B - (\xi^C \nabla_{A C}  \xi^B  )& = 0 \, , \label{eq-cond5d-00} \\
\eqref{eq-cond01} & & \Longleftrightarrow &  &  \left( \nabla_{A B} \xi_{\lb{C}} \right) \xi_{\rb{D}} + \xi_{[C} \varepsilon _{D][A} \nabla_{B]E} \xi^E + \varepsilon _{[C|[A} \xi^E  \nabla_{B]E} \xi_{|D]} & = 0 \, . \label{eq-cond5d-01}
 \end{align}
Finally, denote by $\mcN_\xi$ the almost null structure associated to $[\xi^A]$. Then condition \eqref{eq-almost-foliating} for $\mcN_\xi$ to satisfy $[\Gamma(\mcN_\xi),\Gamma(\mcN_\xi)] \subset \Gamma(\mcN_\xi^\perp)$ reduces to \eqref{eq-cond5d-31}. Condition \eqref{eq-Nperp-int} for $\mcN_\xi$ to be co-integrable can be expressed as 
\begin{align}\label{eq-Nperp-foliating-5d}
 \xi \ind*{_{[A}} \left( \xi \ind{^E} \nabla \ind{_{B]E}} \xi \ind*{_{[C}} \right) \xi \ind*{_{D]}} + \xi \ind*{_{[C}} \varepsilon \ind{_{D][A}} \left( \xi \ind{^E} \nabla \ind{_{B]E}} \xi \ind*{_F} \right) \xi \ind{^F} & = 0 \, .
\end{align}
Condition \eqref{eq-foliating} for $\mcN_\xi$ to be integrable and co-integrable can be expressed as
 \begin{align}\label{eq-foliating-5d}
\left( \xi^B \nabla_{A B} \xi^C \right) \xi_C & = 0 \, , &  \xi_{\lb{A}} \left( \xi^E \nabla_{\rb{B}E} \xi_{\lb{C}} \right) \xi_{\rb{D}} & = 0 \, .
 \end{align}

As an example, one can check that a solution $\xi^A$ of the twistor equation
\begin{align*}
 \nabla \ind{_{AB}} \xi \ind*{_C} + \frac{1}{5} \varepsilon \ind{_{AB}} \zeta \ind{_C} + \frac{4}{5} \zeta \ind{_{\lb{A}}} \varepsilon \ind{_{\rb{B}C}} & = 0 \, , & \zeta \ind{^A} & = \nabla \ind{^{AB}} \xi \ind*{_B} \, ,
\end{align*}
satisfies equations \eqref{eq-foliating-5d} if and only if $\xi^A \zeta_A = 0$ as claimed in Proposition \ref{prop-conformal-Killing-foliating}.

\section{Conformal structures}\label{sec-conformal}
Background information on (holomorphic) conformal structures is already given \cite{Taghavi-Chabert2016} to which the reader should refer. Here, we merely collect useful formulae concerning spinor transformations under a conformal change of holomorphic metrics $\hat{g} \ind{_{ab}} = \Omega^2 g \ind{_{ab}}$ for some non-vanishing holomorphic function $\Omega$ on $\mcM$. Correspondingly, the $\gamma$-matrices can be chosen to transform as $\gamma \ind{_a_A^B}  \mapsto \hat{\gamma} \ind{_a_A^B} = \Omega \gamma \ind{_a_A^B}$ where $\hat{\gamma} \ind{_a_A^B}$ denote the $\gamma$-matrices for metric $\hat{g}_{ab}$. In addition, we can choose the spin invariant bilinear forms $\gamma_{AB}$ on $\mcS$ to rescale with a conformal weight of $1$, and their dual with a conformal weight of $-1$. This means in particular that the quantities $\gamma \ind{_a^{AB}}$ and $\gamma \ind{^a_{AB}}$ have conformal weight $0$. Then the spin connection $\hat{\nabla}_a$ is related to $\nabla_a$ by
\begin{align}\label{eq-conformal_spin_connection2}
 \hat{\nabla} \ind{_a} \xi \ind{^B} & = \nabla \ind{_a} \xi \ind{^B} - \frac{1}{2} \Upsilon \ind{_b} \gamma \ind{^b_a_C^B} \xi \ind{^C} + \frac{1}{2} \Upsilon \ind{_a} \xi \ind{^B} 
 = \nabla \ind{_a} \xi \ind{^B} - \frac{1}{2} \Upsilon \ind{_b} \gamma \ind{^b_C^D} \gamma \ind{_a_D^B} \xi \ind{^C} \, , 
\end{align}
for any holomorphic spinor field $\xi^{A'}$, and similarly for dual spinors. This connection preserves the hatted $\gamma$-matrices \emph{and} the hatted bilinear forms on $\mcS$, in agreement with the convention of \cite{Penrose1984}.

If we now assume that $\xi^A$ is a pure spinor field, we then obtain from \eqref{eq-conformal_spin_connection2}
\begin{align*}
 \left( \hat{\nabla} \ind{_a} \hat{\xi} \ind{^{bB}} \right) \hat{\xi} \ind{_b^C}
& = \left( \nabla \ind{_a} \xi \ind{^{bB}} \right) \xi \ind*{_b^C} + \frac{1}{2} \Upsilon \ind{_b} \xi \ind{^b^D} \gamma \ind{_a_D^B} \xi \ind{^C} + 2 \, \Upsilon \ind{_b} \xi \ind{^b^{[B}} \xi \ind*{_a^{C]}} \, ,  \\
 \left( \hat{\gamma} \ind{^a_B^A} \hat{\nabla} \ind{_a} \hat{\xi} \ind{^{bB}} \right) \hat{\xi} \ind{_b^C}
& = \Omega^{-1} \left( \left( \gamma \ind{^a_B^A} \nabla \ind{_a} \xi \ind{^{bB}} \right) \xi \ind*{_b^C} + \frac{n-2}{2} \left( 2 \, \Upsilon \ind{_b} \xi \ind{^b^C} \xi \ind{^A} - \Upsilon \ind{_b} \xi \ind{^b^{A}} \xi \ind{^{C}} \right) \right) \, ,  
\\
\hat{\xi} \ind{^a^A} \hat{\nabla} \ind{_a} \xi \ind{^B}
& = \Omega^{-1} \left( \xi \ind{^a^A} \nabla \ind{_a} \xi \ind{^{B}}  - \frac{1}{2} \Upsilon \ind{_b} \xi \ind{^b^B} \xi \ind{^A} + \Upsilon \ind{_b} \xi \ind{^b^A} \xi \ind{^B} \right) \, ,  \\
 \left( \hat{\xi} \ind{^a^A} \hat{\nabla} \ind{_a} \hat{\xi} \ind{^{bB}} \right) \hat{\xi} \ind{_b^C}
& = \Omega^{-1} \left( \left( \xi \ind{^a^A} \nabla \ind{_a} \xi \ind{^{bB}} \right) \xi \ind*{_b^C} - \frac{1}{2} \Upsilon \ind{_b} \xi \ind{^b^B} \xi \ind{^A} \xi \ind{^C} - 2 \Upsilon \ind{_b} \xi \ind{^b^{[A}} \xi \ind{^{C]}} \xi \ind{^B} \right) \, ,  
\\
 \left( \hat{\nabla} \ind{_a} \hat{\xi} \ind{^{aB}} \right) \xi \ind{^C} -  \hat{\xi} \ind{^a^B} \hat{\nabla} \ind{_a} \xi \ind{^{C}}  
& = \Omega^{-1} \left( \left( \nabla \ind{_a} \xi \ind{^{aB}} \right) \xi \ind{^C} -  \xi \ind{^a^B} \nabla \ind{_a} \xi \ind{^{C}} + \frac{n-2}{2} \Upsilon \ind{_b} \xi \ind{^b^B} \xi \ind{^C} + \frac{1}{2} \Upsilon \ind{_b} \xi \ind{^b^C} \xi \ind{^B} \right) \, ,  
\end{align*}
where we have set $\hat{\xi} \ind{^a^A} := \hat{\gamma} \ind{^a_B^A} \xi \ind{^B}$. In particular, from the first three expressions, we get
\begin{align*}
 \left( \hat{\nabla} \ind{_a} \hat{\xi} \ind{^{b[B}} \right) \hat{\xi} \ind{_b^C} \xi^{D]}
& = \left( \nabla \ind{_a} \xi \ind{^{b[B}} \right) \xi \ind*{_b^C} \xi^{D]} + 2 \, \Upsilon \ind{_b} \xi \ind{^b^{[B}} \xi \ind*{_a^{C}} \xi^{D]} , ,  \\
 \left( \hat{\gamma} \ind{^a_B^A} \hat{\nabla} \ind{_a} \hat{\xi} \ind{^{bB}} \right) \hat{\xi} \ind{_b^{[C}}  \xi^{D]}
& = \Omega^{-1} \left( \left( \gamma \ind{^a_B^A} \nabla \ind{_a} \xi \ind{^{bB}} \right) \xi \ind*{_b^{[C}} \xi^{D]}  + (n-2) \, \xi \ind{^A} \Upsilon \ind{_b} \xi \ind{^b^{[C}}   \xi^{D]} \right) \, ,
\\
 \left( \hat{\xi} \ind{^a^A} \hat{\nabla} \ind{_a} \xi \ind{^{[B}} \right) \xi \ind{^{C]}}
& = \Omega^{-1} \left(\left(  \xi \ind{^a^A} \nabla \ind{_a} \xi \ind{^{[B}} \right) \xi \ind{^{C]}}  - \frac{1}{2} \xi \ind{^A} \Upsilon \ind{_b} \xi \ind{^b^{[B}}  \xi \ind{^{C]}} \right) \, ,
\end{align*}
from which the conformal invariance of \eqref{eq-cond-11} follows.

\bibliography{biblio}

\end{document}